\documentclass{amsart}
\usepackage[T1]{fontenc}
\usepackage{lmodern}

\usepackage{amscd, amssymb, amsmath, amsthm}
\usepackage{enumerate, latexsym, mathrsfs}
\usepackage{xypic}
\usepackage[pdftex]{geometry}
\usepackage[
		bookmarks=true, bookmarksopen=true,%
    bookmarksdepth=3,bookmarksopenlevel=2,%
    colorlinks=true,%
    linkcolor=blue,%
    citecolor=blue]{hyperref}
\newtheorem{theorem}{Theorem}[section]
\newtheorem{lemma}[theorem]{Lemma}
\newtheorem{proposition}[theorem]{Proposition}
\newtheorem{corollary}[theorem]{Corollary}

\theoremstyle{definition}
\newtheorem{definition}[theorem]{Definition}
\newtheorem{notation}[theorem]{Notation}
\newtheorem{example}[theorem]{Example}

\newtheorem{remark}[theorem]{Remark}

\numberwithin{equation}{section}

\newcommand{\Real}{{\mathbb R}}
\newcommand{\Rational}{{\mathbb Q}}
\newcommand{\Complex}{{\mathbb C}}
\newcommand{\Integral}{{\mathbb Z}}
\newcommand{\Natural}{{\mathbb N}}

\newcommand{\Hyp}{{\mathbf H}}

\newcommand{\Sft}{{\widetilde{\mathrm{SL}}(2,\Real)}}

\newcommand{\id}{{\mathrm{id}}}

\newcommand{\opEuler}{{\mathbf{e}}}

\newcommand{\winding}{{\mathrm{wind}}}

\title[Volume of Seifert representations]{Volume of Seifert representations \\for graph manifolds and their finite covers}

\author{Pierre Derbez}
\address{LATP UMR 7353, 39 rue  Joliot-Curie, 13453 Marseille Cedex 13, France}
\email{pderbez@gmail.com}

\author[Yi Liu]{Yi Liu}
\address{
    Beijing International Center for Mathematical Research, Peking University, Beijing 100871, China}
\email{liuyi@bicmr.pku.edu.cn}
\author{Shicheng Wang}
\address{School of Mathematical Sciences, Peking University, Beijing 100871, China}
\email{wangsc@math.pku.edu.cn}


\date{\today}

\begin{document}
\begin{abstract}
	For any closed orientable 3-manifold, there is a volume function 
	defined on the space of all Seifert representations of the fundamental group.
	The maximum absolute value of this function
	agrees with the Seifert volume of the manifold due to Brooks and Goldman.
	
	For any Seifert representation of a graph manifold,
	the authors establish an effective formula for computing its volume,
	and obtain restrictions to the representation	as analogous to the Milnor--Wood inequality 
	(about transversely projective foliations on Seifert fiber spaces).
	It is shown that
	the Seifert volume of any graph manifold is a rational multiple of $\pi^2$.
	
	Among all finite covers of a given non-geometric graph manifold,
	the supremum ratio of the Seifert volume over the covering degree
	can be a positive number, and can be infinite.
	Examples of both possibilities are discovered, and confirmed, 
	with the explicit values determined for the finite ones. 
\end{abstract}

\maketitle

\section{Introduction}
The Seifert volume $\mathrm{SV}(M)\in[0,+\infty)$
of any orientable closed $3$--manifold $M$
is introduced by Brooks and Goldman \cite{BG2},
as a generalization of Gromov's simplicial volume \cite{Gr}.
This invariant can be obtained as the maximum absolute value 
of the volume function 
defined on the set of all $\Sft\times_\Integral\Real$--representations of $\pi_1(M)$,
(see Section \ref{Subsec-volume_of_representations}).
For any continuous map $f\colon M'\to M$ between $3$--manifolds,
the Seifert volume satisfies a similar inequality as with the simplicial volume:
$$\mathrm{SV}(M')\geq|\mathrm{deg}(f)|\times\mathrm{SV}(M),$$
however, it does not have to achieve the equality for covering projections.
Moreover,
the Seifert volume is very intricately related 
with the geometric decomposition of the prime factors.
The goal of the present paper is to unravel the relationship in the case of graph manifolds,
and to understand the behavior of this invariant in finite covers.

We recall some recent advances 
to put our study into a context.
First suppose $M$ is geometric in the sense of Thurston \cite{Thurston-book}.
Then $\mathrm{SV}(M)$ equals zero
unless $M$ supports the Seifert geometry $\Sft$ or the hyperbolic geometry $\Hyp^3$.
In the Seifert case,
$\mathrm{SV}(M)$ actually equals the $\Sft$--geometric volume.
For any finite cover $M'$ of $M$,
it follows that 
$\mathrm{SV}(M')$ is nonzero and is proportional to the covering degree $[M':M]$.
In the hyperbolic case, $\mathrm{SV}(M)$ may happen to be zero,
but there are always finite covers $M'$ of $M$ with nonzero $\mathrm{SV}(M')$.
Moreover, there are always some tower of finite covers 
$\cdots\to M'_n\to \cdots\to M'_2\to M'_1$ of $M$,
such that the ratio $\mathrm{SV}(M'_n)/[M'_n:M]$ is unbounded as $n$ increases.
In other words, 
virtual Seifert volume grows super-linearly fast 
(in some finite-covering towers),
for hyperbolic $3$--manifolds \cite{DLSW-virtual_SV}.
More generally, 
it is shown in \cite{DLSW-virtual_SV} that
virtual Seifert volume grows super-linearly fast
for any orientable closed $3$--manifold of nonzero simplicial volume.
Since orientable closed $3$--manifolds of zero simplicial volume 
are just connected sum of graph manifolds,
essentially it remains interesting to ask,
in the case of non-geometric graph manifolds,
how fast virtual Seifert volume could possibly grow.

The \emph{covering Seifert volume} as introduced in \cite[Section 6]{DLSW-virtual_SV}
captures the supremum linear growth rate of virtual Seifert volume.
For any orientable closed $3$--manifold $M$, it is defined as follows with value in $[0,+\infty]$:
\begin{equation}\label{def_CSV}
\mathrm{CSV}(M)=\sup\left\{\frac{\mathrm{SV}(M')}{[M':M]}\colon M'\mbox{ a finite cover of }M\right\}.
\end{equation}
It satisfies the covering property 
$$\mathrm{CSV}(M')=[M':M]\times \mathrm{CSV}(M)$$
for finite covering projections $M'\to M$. 
For instance, the above mentioned result from \cite{DLSW-virtual_SV}
simply says that nonzero simplicial volume implies infinite covering Seifert volume.
Besides, it is proved in \cite{DLW-cs_sep_vol} that
any non-geometric graph manifold must have nonzero covering Seifert volume.

As the motivational application of our main results,
we determine the covering Seifert volume 
for a sampling series of non-geometric graph manifolds.
There we see examples of both possibilities,
of finite or of infinite covering Seifert volume.
The infinite ones are somewhat more unexpected.
Their existence brings about the new question 
of understanding the topological distinction
behind the different growth behaviors.

\subsection{Summary of results}
Below we illustrate our main results
with a family of very simply designed graph manifolds.
Meanwhile, we point out the general theorems
as actually proved in this paper.

\begin{example}[Twisted doubling]\label{Mgabcd}
Let $\Sigma_{g,1}$ be an orientable compact surface of genus $g$ and with a single boundary component.
Denote by $\Sigma_+$ and $\Sigma_-$ a pair of oppositely oriented copies of $\Sigma_{g,1}$,
and by $J_\pm=\Sigma_\pm\times S^1$ the oriented product manifolds of $\Sigma_\pm$ 
with the oriented circle $S^1$.
For any entries $a,b,c,d\in\Integral$ with $ad-bc=1$,
we obtain an orientation-reversing homeomorphism $\varphi\colon \partial J_-\to \partial J_+$,
such that the induced isomorphism $\varphi_*\colon H_1(\partial J_-)\to H_1(\partial J_+)$ 
operates on the standard bases as
$$\varphi_*\left(\begin{array}{cc}[\partial\Sigma_-]&[S^1]\end{array}\right)
=\left(\begin{array}{cc}[\partial\Sigma_+]&[S^1]\end{array}\right)
\,\left[\begin{array}{cc}a&b\\c&d\end{array}\right].$$
We obtain a closed oriented $3$--manifold
$$M\left(g;\left[\begin{array}{cc}a&b\\c&d\end{array}\right]\right)=J_-\cup_{\varphi} J_+$$
from $J_-$ and $J_+$ by identifying their boundaries using $\varphi$,
or $M(g;a,b,c,d)$ for brevity.

Suppose $g>0$ and $b\neq0$.
Then $M(g;a,b,c,d)$ is a graph manifold with two JSJ pieces $J_-$ and $J_+$
and a single JSJ torus $T=\partial J_+=\varphi(\partial J_-)$.
There is an obvious orientation-reversing homeomorphism
$M(g;a,b,c,d)\cong M(g;d,-b,-c,a)$,
as obtained by switching $\Sigma_\pm$ and inversing $\varphi$.
\end{example}

For any oriented graph manifold $M=M(g;a,b,c,d)$ as above and
any representation $\rho\colon\pi_1(M)\to\Sft\times_\Integral\Real$,
we obtain data $\xi_\pm=\xi_\pm(\rho)\in\Real$ as follows.
Fix an auxiliary base point on $T$, and obtain the van Kampen splitting 
$\pi_1(M)=\pi_1(J_-)*_{\pi_1(T)}\pi_1(J_+)$.
Denote by $f_\pm\in\pi_1(M)$ the central element of the subgroup
$\pi_1(J_\pm)\cong\pi_1(\Sigma_\pm)\times\pi_1(S^1)$
represented by the oriented fiber $S^1$ of $J_\pm$.
If $\rho(f_\pm)$ lies in the center $\Real$ of $\Sft\times_\Integral\Real$,
one may simply take $\xi_\pm$ as the value of $\rho(f_\pm)$ in $\Real$.
In general, there is a continuous, real-valued conjugacy-class invariant 
for the elements of $\Sft\times_\Integral\Real$,
which we call the \emph{essential winding number} and introduce in Definition \ref{def_essential_winding_number}.
Then we define $\xi_\pm$ as the essential winding number of $\rho(f_\pm)$.
Hence we note that $\xi_\pm$ does not depend on the particular choice of the base point.

Specialized to twisted doublings and their representations, 
our main results can be summarized as follows:
\begin{itemize}
\item \emph{Volume formula}. Setting $k_+=d/b$ and $k_-=a/b$,
$$\mathrm{vol}_{\Sft\times_\Integral\Real}(M,\rho)=
4\pi^2\cdot\left(k_+\cdot\xi_+^2 -2b^{-1}\cdot\xi_+\xi_- +k_-\cdot\xi_-^2\right).$$
\item \emph{Generalized Milnor--Wood inequalities}.
$$
\begin{array}{cc}
\left|k_+\cdot\xi_+ -b^{-1}\cdot\xi_-\right|\leq 2g-1;&
\left|-b^{-1}\cdot\xi_+ +k_-\cdot\xi_-\right|\leq 2g-1.
\end{array}$$
\item \emph{Rationality}.
$$\frac{1}{\pi^2}\cdot
\mathrm{vol}_{\Sft\times_\Integral\Real}(M,\rho)\,\in\,\Rational.$$
\end{itemize}
It follows by definition that $\mathrm{SV}(M)$
is also a rational multiple of $\pi^2$.
We notice that the volume formula is a homogeneous quadratic function
of $\xi_\pm(\rho)$ with coefficients determined by the topology of $M$.
Moreover, for any $\eta_\pm\in\Rational$ satisfying the strict inequality
$|k_\pm\cdot\eta_\pm-b^{-1}\cdot\eta_\mp|<2g-1$,
we are able to construct some finite cover $M'$ of $M$ and
some $\Sft\times_\Integral\Real$--representation $\rho'$ of $\pi_1(M')$,
such that
the volume of $(M',\rho')$ equals
$[M':M]$ times $4\pi^2\cdot\left(k_+\eta_+^2-2b^{-1}\cdot\eta_+\eta_-+k_-\eta_-^2\right)$.

In this paper, we establish similar results for general graph manifolds.
We consider any oriented closed graph manifold $M$
with a simplicial JSJ graph $(V,E)$ and with oriented Seifert fibrations in the JSJ pieces,
or what we call a \emph{formatted graph manifold} (Definition \ref{def_format}).
We show that 
the volume for any representation $\rho\colon\pi_1(M)\to\Sft\times_\Integral\Real$
is a homogeneous quadratic function of
the essential winding numbers of $\rho(f_v)$, 
where $f_v$ corresponds to the oriented fiber of the JSJ piece $J_v$
for each vertex $v$. 
The quadratic function can be written down explicitly,
involving the charges $k_v$ associated to the JSJ pieces $J_v$,
and the fiber intersection numbers $b_{v,w}$
associated to the JSJ tori $T_{v,w}$.
The symmetric matrix $\opEuler_M\in\mathrm{End}_\Real(\Real^V)$
that defines the quadratic function is called the \emph{Euler operator}. 
Indeed, we find it playing a similar role in the general volume formula
as the Euler number in the oriented Seifert fibration case,
(like quantization of classical observables).
We obtain generalized Milnor--Wood inequalities, one with each vertex.
We also prove the commensurability of volume values with $\pi^2$.
These results constitute Theorem \ref{volume_general_GM}.
Moreover, 
we establish the generic virtual existence of $\Sft\times_\Integral\Real$--volume values 
for formatted graph manifolds. This is Theorem \ref{volume_existence_general_GM}.

\begin{example}\label{CSV_Mgm}
	Suppose $g>0$.
	$$\mathrm{CSV}\left(M\left(g;\left[\begin{array}{cc}m&1\\m^2-1&m\end{array}\right]\right)\right)=
	\begin{cases}+\infty&m=0,\pm1\\
	8\pi^2\cdot(2g-1)^2/(m-1)&m=\pm2,\pm3,\pm4,\cdots
	\end{cases}$$
\end{example}

Hence we see that the covering Seifert volume can be finite and can be infinite
for non-geometric graph manifolds. 
Example \ref{CSV_Mgm} is a consequence of Theorem \ref{cyclic_k_b},
where the covering Seifert volume is determined for any graph manifold
with constant charges and constant fiber intersection numbers, 
and with a cyclic JSJ graph, (see Remark \ref{cyclic_k_b_remark}).
The infinite case immediately leads to lots of other graph manifolds 
of infinite covering Seifert volume, 
once we require those manifolds 
to project $M(g;0,1,-1,0)$ or $M(g;\pm1,1,0,\pm1)$ of nonzero degree.
On the other hand, 
for a large class of what we call \emph{strictly diagonally dominant} graph manifolds,
we are able to confirm finiteness of their covering Seifert volume:
Among twisted doublings,
Example \ref{CSV_Mgabcd} below is the special case of Theorem \ref{bound_diagonally_dominant}.
In certain occasions such as Example \ref{CSV_Mgm},
the upper bounds are indeed sharp, and the conditions also happen to be necessary.

\begin{example}\label{CSV_Mgabcd}
	Suppose $g>0$. If $|a|>1$ and $|d|>1$, and hence $b\neq0$,
	$$\mathrm{CSV}\left(M\left(g;\left[\begin{array}{cc}a&b\\c&d\end{array}\right]\right)\right)
	\leq 4\pi^2\cdot(2g-1)^2\cdot\left(\frac{|b|}{|a|-1}+\frac{|b|}{|d|-1}\right).$$
\end{example}

None of the twisted doublings in Example \ref{CSV_Mgabcd} admit finite covers fibering over a circle.
The manifold $M(g;1,1,0,1)$ does fiber over a circle,
as does its orientation-reversal $M(g;-1,1,0,-1)$.
In fact, it is easy to identify $M(g;1,1,0,1)$ with the mapping torus of a Dehn twist,
which acts on the doubling of $\Sigma_{g,1}$ twisting along the doubling curve.
The manifold $M(g;0,1,-1,0)$ is virtually fibered;
it actually admits a Riemannian metric of nonpositive sectional curvature,
because it is chargeless (meaning $k_+=k_-=0$).
These topological descriptions all follow from 
Buyalo and Svetlov's characterization 
on virtual properties of graph manifolds \cite{BS-graph_manifold},
plus simple observation.

Therefore, our examples seem to suggest that
the difference between finite and infinite covering Seifert volume
lies somewhere between the strictly diagonally dominant class
and the virtually fibered class,
for non-geometric graph manifolds in general.

For orientable closed $3$--manifolds of nonzero simplicial volume,
it remains unclear how to compute volume of their Seifert representations effectively.
In particular, the authors do not know 
if any of those manifolds have Seifert volume incommensurable with $\pi^2$.
It might be interesting to compare our phenomenon here
with the rationality results in \cite{Re-rationality,NY-rationality}
about $\mathrm{PSL}(2,\Complex)$--Chern--Simons invariants.

\subsection{Ingredients of proofs}
Below we explain some key ideas that are developed in this paper.
There is another more technical outline in Section \ref{Subsec-outline_volume_general_GM},
where we are about to prove Theorem \ref{volume_general_GM}.
As a general strategy, 
we wish to work out a fundamental case by direct means,
and reduce any other case to that one,
but we must also have a suitable formulation to proceed.

Suppose that $M$ is a graph manifold and $\rho$ is a $\Sft\times_\Integral\Real$--representation of $\pi_1(M)$.
If $\rho$ sends the center of any vertex group $\pi_1(J_v)$ to the central subgroup $\Rational$,
we can apply the additivity principle from \cite{DLW-cs_sep_vol} to compute volume.
The volume turns out to depend only on the numbers $\rho(f_v)\in\Rational$
and the structural data $k_v$ and $b_{v,w}$ about $M$.
Inspired by the well-known volume formula and the Milnor--Wood inequality for Seifert geometric manifolds,
we are able to formulate our analogous results in this case.

To deal with general representations, 
there are two prominent issues in front: 
One is to convert $\rho(f_v)$ somehow into a number;
the other is to deform $\rho$ somehow into a representation as above.
Both of these issues are related with the topological group structure of $\Sft\times_\Integral\Real$.
By introducing the essential winding number (Definition \ref{def_essential_winding_number}),
we classify elements of $\Sft\times_\Integral\Real$ up to conjugacy
based on the classification of conjugacy classes in $\mathrm{PSL}(2,\Real)$.
This is done in Section \ref{Subsec-classification_conjugacy}, 
and resolves the first issue.
Our solution to the second issue relies heavily on
commutator factorization of paths in $\Sft\times_\Integral\Real$.
Generally speaking, for any topogical group $G$,
the commutator map $G\times G\to G\colon (a,b)\mapsto aba^{-1}b^{-1}$
does not satisfy the path lifting property.
However, 
we are able to construct lifts for certain nice paths in $\Sft\times_\Integral\Real$.
The relevant techniques are developed in Section \ref{Sec-pdc}.

When studying existence of transverse foliations in Seifert fiber spaces,
Eisenbud, Hirsch, and Neumann \cite{EHN} considered 
the topological group $\mathscr{D}^+$ 
which consists of all the diffeomorphisms $\varphi\colon\Real\to\Real$ such that 
$\varphi(r+1)=\varphi(r)+1$ for all $r\in\Real$.
They made use of the maximum $\overline{m}(\varphi)$ and the minimum $\underline{m}(\varphi)$
of the function $\varphi(r)-r$, 
and analyzed the problem of factorizing any element $\varphi$ into a given number of commutators.
Restricted to the subgroup $\Sft\times_\Integral\Real$ of $\mathscr{D}^+$,
our techniques can be regarded as refinement of theirs:
Our essential winding number of $\varphi$ is a precise value 
to replace the former interval $[\underline{m}(\varphi),\overline{m}(\varphi)]$.
Our commutator path lifting lemmas
are enhanced versions of former factorization results.

To deform a general representation $\rho$ continuously as desired,
we actually have to work with the pull-back $\rho'$ of $\rho$ to another manifold $M'$.
Here $M'$ is some nicely constructed graph manifold which pinches nicely onto $M$.
In particular, the volume of $(M',\rho')$ is the same as that of $(M,\rho)$.
Our deformation path for $\rho'$ is constructed by compatibly assembling deformation paths of 
the restricted representations to the vertex groups of $\pi_1(M')$.
We analyze the local types of $\rho'$ at the vertices,
and employ our commutator path liftings to construct the local deformations case by case.
The local types are explained in Section \ref{Subsec-local_types}.
The constructions of 
$M'$ and the deformation path for $\rho'$ are done in Section \ref{Sec-reduction_vc},
(see also Section \ref{Subsec-outline_volume_general_GM} for more explanation on the idea).
With these constructions, we are able to complete the proof of Theorem \ref{volume_general_GM}.
We also make use of similar constructions on some finite covers of $M$
to prove Theorem \ref{volume_existence_general_GM}.

\subsection{Organization}
This paper can be divided into three parts:
The first part is mostly preliminary and expository; 
the second part is about Theorem \ref{volume_general_GM} 
and its application pertaining to Example \ref{CSV_Mgabcd};
the third part is about Theorem \ref{volume_existence_general_GM} 
and its application pertaining to Example \ref{CSV_Mgm}.
These parts are organized as follows.

The first part consists of
Sections \ref{Sec-Seifert_geometry}--\ref{Sec-graph_manifold}.
In Section \ref{Sec-Seifert_geometry},
we study the structure of $\Sft\times_\Integral\Real$,
and in particular,
we introduce the essential winding number (Definition \ref{def_essential_winding_number}).
In Section \ref{Sec-Seifert_volume},
we review the Seifert volume from the approach of volume of representations.
In Section \ref{Sec-graph_manifold},
we review the topological theory about Seifert fiber spaces and graph manifolds.

The second part consists of 
Sections \ref{Sec-volume_general_GM_statement}--\ref{Sec-strictly_diagonally_dominant}.
In Section \ref{Sec-volume_general_GM_statement},
we state our main result (Theorem \ref{volume_general_GM}) and outline its proof.
In Sections \ref{Sec-vc_case}--\ref{Sec-proof_volume_general_GM},
we develop necessary techniques and prove Theorem \ref{volume_general_GM},
(see Section \ref{Subsec-outline_volume_general_GM} for details about their contents).
In Section \ref{Sec-strictly_diagonally_dominant},
we apply Theorem \ref{volume_general_GM}
to what we call strictly diagonally domianant graph manifolds,
and bound their covering Seifert volume (Theorem \ref{bound_diagonally_dominant}).

The third part consists of Sections \ref{Sec-virtual_existence} and \ref{Sec-cyclic}.
In Section \ref{Sec-virtual_existence},
we obtain the part converse of our main result
(Theorem \ref{volume_existence_general_GM}).
In Section \ref{Sec-cyclic},
we apply Theorems \ref{volume_general_GM} and \ref{volume_existence_general_GM}
to what we call constant cyclic graph manifolds,
and determine their covering Seifert volume (Theorem \ref{cyclic_k_b}).

There is an appendix Section \ref{Sec-normalization_SV},
where we clarify the normalization of the Seifert volume 
as mentioned in Section \ref{Subsec-volume_of_representations}.

\section{The motion group of Seifert geometry}\label{Sec-Seifert_geometry}
	Seifert geometry is one of the eight $3$--dimensional geometries 
	as classified by Thurston \cite{Thurston-book}.
	In this section, we investigate this geometry from a transformation group perspective.
	We take the simply connected Lie group $\Sft$ as the model of its space.
	The identity component of its isomorphism group, 
	which we refer to as the \emph{motion group} of Seifert geometry,
	can be identified	with $\Sft\times_\Integral\Real$.
	
	We introduce a useful invariant for Seifert-geometric motions
	called the \emph{essential winding number}.
	This is a continuous, conjugation-invariant function
	$\Sft\times_\Integral\Real\to \Real$, (see Definition \ref{def_essential_winding_number}).
	It plays a crucial role in the volume formula
	for Seifert-geometric motion representations of graph manifolds.
	
	\subsection{Construction via central extensions}\label{Subsec-central_extension}
	Let $\Sft$ be the universal covering Lie group
	of $\mathrm{SL}(2,\Real)$.
	(Throughout this paper,
	we think of Lie groups topologically as pointed spaces based at the identity.
	Universal covering Lie groups may be constructed canonically
	using relative homotopy classes of paths.)
	The expression
	$$k(r)=\left[\begin{array}{cc}\cos(\pi r)& -\sin(\pi r)\\ \sin(\pi r)& \cos(\pi r)\end{array}\right]$$
	for all $r\in\Real$ 
	defines a differentiable homomorphism $k\colon \Real\to \mathrm{SL}(2,\Real)$,
	which lifts to be	a differentiable homomorphism
	$$\widetilde{k}\colon \Real\to \Sft.$$
	The closed subgroup $\widetilde{\mathrm{SO}}(2)$ of $\Sft$ which projects $\mathrm{SO}(2)$
	is parametrized by $\Real$ via $\widetilde{k}$.
	The center $Z(\Sft)$ of $\Sft$ can be recognized as the image of $\Integral$ under $\widetilde{k}$.
		
	Let $\Sft\times_\Integral\Real$ be the Lie group that is constructed as the quotient of
	$\Sft\times\Real$ by the infinite cyclic, discrete, central subgroup $\{(\widetilde{k}(n),-n)\colon n\in\Integral\}$.
	We denote elements of $\Sft\times_\Integral\Real$ as $g[s]$ for $g\in\Sft$ and $s\in\Real$,
	so the multiplication rule reads
	$$g[s]g'[s']=(gg')[s+s'],$$
	and the relation
	$$g\widetilde{k}(n)[s]=g[s+n]$$
	holds for all $n\in\Integral$.
	There are natural embeddings of $\Sft$ and $\Real$ into $\Sft\times_\Integral\Real$,
	namely, $g\mapsto g[0]$ and $s\mapsto \mathrm{id}[s]$.
	We obtain the following short exact sequence of Lie groups homomorphisms:
	\begin{equation}\label{central_extension_diagram}
	\xymatrix{
	\{0\} \ar[r]  & \Real \ar[r]^-{\mathrm{incl}}  & \Sft\times_\Integral\Real \ar[r] & \mathrm{PSL}(2,\Real) \ar[r] & \{1\} 	
	}
	\end{equation}
	
	\subsection{Classification of conjugacy classes}\label{Subsec-classification_conjugacy}
	We introduce the essential winding number and use it to classify the conjugacy classes
	of $\Sft\times_\Integral\Real$.

	\begin{lemma}\label{pre_essential_winding_number}
	The expression
	$$\bar\omega\left(\pm\left[\begin{array}{cc}a &b \\ c &d\end{array}\right]\right)
	=\begin{cases} 
	\mathrm{sgn}(c-b)\cdot\arccos\left(\frac{a+d}2\right) \mod \pi\Integral & \mbox{if }|a+d|<2\\
	0 \mod \pi\Integral & \mbox{if }|a+d|\geq2
	\end{cases}
	$$	
	defines a continuous map 
	$\bar\omega\colon \mathrm{PSL}(2,\Real)\to \Real/\pi\Integral$,
	which is invariant under conjugation of $\mathrm{PSL}(2,\Real)$.	
	Moreover, 
	there exists a unique continuous, conjugation-invariant function	
	$\widetilde{\omega}\colon \Sft\to \Real$
	which lifts $\bar\omega$ and which satisfies
	$$\widetilde{\omega}\left(\widetilde{k}(r)\right)=\pi r$$
	for all $r\in\Real$.
	\end{lemma}
	
	\begin{proof}
		We observe 
		$\mathrm{sgn}(b-c)\cdot\arccos(-(a+d)/2)=\mathrm{sgn}(c-b)\cdot\arccos((a+d)/2)-\pi$,
		so $\bar\omega$ is well-defined.
		As $|a+d|$ approaches $2$ from below, the value of $\bar{\omega}$ approaches $0\bmod\pi\Integral$,
		so $\bar{\omega}$ is clearly continuous.
		For $|a+d|<2$, the condition $ad-bc=1$ implies either $b<0<c$ or $c<0<b$,
		so $\mathrm{sgn}(c-b)$ is either $+1$ or $-1$.
		Since $\mathrm{SL}(2,\Real)$ is connected, and since $(a+d)/2$ is conjugation invariant,
		$\bar{\omega}$ must be constant on every conjugation class of $\mathrm{PSL}(2,\Real)$.
		Note that $\bar{\omega}(\pm k(r))$ equals $\pi r\bmod\pi\Integral$, for all $r\in\Real$.
		It follows that
		$\bar{\omega}$ induces a group isomorphism $\pi_1(\mathrm{PSL}(2,\Real))\to\pi_1(\Real/\pi\Integral)$.
		Therefore, there exists a unique continuous function $\widetilde{\omega}\colon \Sft\to \Real$
		which lifts $\bar{\omega}$ and
		which satisfies $\widetilde{\omega}(\widetilde{k}(r))=\pi r$, for all $r\in\Real$.
		For any $g\in\Sft$, the conjugation map $\Sft\to\Sft\colon u\mapsto ugu^{-1}$ 
		is constant on the center $\Integral$,
		so it descends to a map $\mathrm{PSL}(2,\Real)\to\Sft$.
		The latter is the lift of the conjugation map 
		$\mathrm{PSL}(2,\Real)\to\mathrm{PSL}(2,\Real)\colon \bar{u}\mapsto \bar{u}\bar{g}\bar{u}^{-1}$.
		Then $\widetilde{\omega}$ is invariant under conjugation of $\Sft$,
		by the conjugation-invariance of $\bar{\omega}$ and the construction of $\widetilde{\omega}$.
	\end{proof}
	
	\begin{remark}
		The map $\bar{\omega}$ can be visualized as follows, (compare \cite[Section 2]{Kh}).
		Denote by
		$$\mathcal{D}=\left\{(x,y,t)\in\Real^3\colon t^2-x^2-y^2\leq 1\right\}$$
		the region in $3$--space between the two sheets of the hyperboloid $t^2-x^2-y^2=1$.
		Under the continuous map 
		\begin{eqnarray*}
		\mathcal{D} \to\mathrm{PSL}(2,\Real)&
		\colon&
		(x,y,t) \mapsto{
		\pm\left[\begin{array}{cc}
		\sqrt{1+x^2+y^2-t^2}+y& x+t\\
		x-t& \sqrt{1+x^2+y^2-t^2}-y
		\end{array}\right],
		}
		\end{eqnarray*}
		the interior of $\mathcal{D}$ projects homeomorphically onto the open subset of $\mathrm{PSL}(2,\Real)$
		consisting of the elements of nonzero trace.
		Each boundary component of $\mathcal{D}$ projects homeomorphically onto
		the closed subset of traceless elements, 
		such that the points $(x,y,t),(-x,-y,-t)\in\partial\mathcal{D}$ project the same element.
		Therefore, $\mathrm{PSL}(2,\Real)$ is homeomorphic to the quotient space of $\mathcal{D}$
		identifying every antipodal pair of points on the boundary.
		The quotient space is of course an open solid torus.
		Under the homeomorphism,
		every conjugacy class of a hyperbolic element in $\mathrm{PSL}(2,\Real)$
		is a one-sheet hyperboloid $t^2-x^2-y^2=c$, for some $c<0$;
		every conjugacy class of an elliptic element 
		is a sheet of a hyperboloid $t^2-x^2-y^2=c$, for some $0<c\leq1$;
		there are another two conjugacy classes of parabolic elements, corresponding to 
		the components of the cone $t^2-x^2-y^2=0$ without the origin.
		One may easily check that the map $\bar\omega\colon \mathrm{PSL}(2,\Real)\to \Real/\pi\Integral$
		collapses all the non-elliptic conjugacy classes to $0\bmod\pi\Integral$,
		and collapses every elliptic conjugacy class to a distinct point.
	\end{remark}

	\begin{definition}\label{def_essential_winding_number}
	For any $g[s]\in\Sft\times_\Integral\Real$,
	the \emph{essential winding number} of $g[s]$ is defined as
	$$\winding\left(g[s]\right)=\frac{\widetilde{\omega}(g)}\pi+s,$$
	which is a well-defined value in $\Real$.
	Here $\widetilde{\omega}\colon\Sft\to \Real$ 
	is as provided by Lemma \ref{pre_essential_winding_number}.
	\end{definition}
	
	We obtain a characterization of conjugate elements in $\Sft\times_\Integral\Real$ 
	in terms of the essential winding number, as follows.
	
	\begin{proposition}\label{prop_essential_winding_number}
		Let $g[s],g'[s']$ be a pair of elements in $\Sft\times_\Integral\Real$.
		Denote by $\bar{g},\bar{g}'$ 
		their images in $\mathrm{PSL}(2,\Real)$ under the natural quotient homomorphism.
		Then the following statements are equivalent:
		\begin{enumerate}
		\item The elements $g[s]$ and $g'[s']$
		are conjugate in $\Sft\times_\Integral\Real$.
		\item 
		The elements $g[s]$ and $g'[s']$ have equal essential winding number,
		and
		the elements $\bar{g}$ and $\bar{g}'$ are conjugate in $\mathrm{PSL}(2,\Real)$.
		\end{enumerate}
	\end{proposition}
	
	\begin{proof}
		Suppose that $g[s]$ is conjugate to $g'[s']$ in $\Sft\times_\Integral\Real$.
		Then $\bar{g}$ and $\bar{g}'$ are obviously conjugate in $\mathrm{PSL}(2,\Real)$.
		Below we argue $\winding(g[s])=\winding(g'[s'])$.
		Note that there is a natural quotient homomorphism 
		$\Sft\times_\Integral\Real\to \Real/\Integral$,
		which sends any $g[s]$ to $s\mod\Integral$.
		This implies that $s$ must be $s'+n$ for some $n\in\Integral$,
		due to the assumed conjugacy. 
		We obtain $g[s]=g[s'+n]=g\widetilde{k}(n)[s']$.
		Moreover, $g\widetilde{k}(n)$ must be conjugate to $g'$ in $\Sft$.
		In fact, $g[s]=u[t]g'[s'](u[t])^{-1}$ is equivalent to 
		$g\widetilde{k}(n)[s']=ug'u^{-1}[s']$, 
		and also equivalent to $g\widetilde{k}(n)=ug'u^{-1}$.
		By Lemma \ref{pre_essential_winding_number},
		we see
		$\winding(g[s])=\winding(g\widetilde{k}(n)[s'])=\winding(g'[s'])$.
	
		Conversely,
		suppose $\bar{g}=\bar{u}\bar{g}'\bar{u}^{-1}$ for some $\bar{u}\in\mathrm{PSL}(2,\Real)$,
		and also suppose $\winding(g[s])=\winding(g'[s'])$.
		Then for any lift $u\in\Sft$ of $\bar{u}$,
		the element $ug'u^{-1}g^{-1}$ is central, so there exists some $n\in\Integral$
		such that $ug'u^{-1}=g\widetilde{k}(n)$. 
		Using Lemma \ref{pre_essential_winding_number},
		we obtain
		$\winding(g[s])=\winding(g\widetilde{k}(n)[s-n])=\winding(ug'u^{-1}[s-n])
		=\winding(g'[s'-n])=\winding(g'[s'])-n$.
		This means $n=0$, and the conjugation relation
		$u[0]g'[s'](u[0])^{-1}=g[s]$ follows, as desired.		
	\end{proof}

	\subsection{Construction of the Seifert geometry}
	There is a natural action of the Lie group $\Sft\times_\Integral\Real$ on $\Sft$ by diffeomorphisms,
	namely, 
	$$g[r].h=gh\widetilde{k}(r),$$
	for all $g[r]\in\Sft\times_\Integral\Real$ and $h\in\Sft$.
	The action is differentiable, transitive, and proper.
	In other words,
	it turns $\Sft$ into a $3$--dimensional geometry in the sense of W. P. Thurston \cite{Thurston-book}.
	In fact, 
	$\Sft\times_\Integral\Real$ is the identity component of 
	the transformation group $\mathrm{Iso}(\Sft)$ of the Seifert geometry $\Sft$.
	
	One may actually describe the maximal transformation group $\mathrm{Iso}(\Sft)$ 
	of the Seifert geometry in terms of the central extension construction,
	as follows.
	Observe that there is a canonical involutive Lie group automorphism
	of $\mathrm{SL}(2,\Real)$, defined as
	$$\nu\colon 
	\left[\begin{array}{cc}a& b\\ c& d\end{array}\right]
	\mapsto 
	\left[\begin{array}{cc}a& -b\\ -c& d\end{array}\right].$$
	It lifts to a canonical involutive Lie group automorphism
	$$\widetilde{\nu}\colon \Sft\to \Sft.$$
	As it turns out, 
	$\mathrm{Iso}(\Sft)$ is the subgroup of $\mathrm{Diffeo}(\Sft)$ generated by $\Sft\times_\Integral\Real$
	and $\widetilde{\nu}$.
	For any $g[s]\in\Sft\times_\Integral\Real$,
	the relation
	$\widetilde{\nu}\circ g[s]=\widetilde{\nu}(g)[-s]\circ\widetilde{\nu}$
	holds.
	In other words, $\mathrm{Iso}\left(\Sft\right)$ factorizes as a semidirect product
	$$\mathrm{Iso}\left(\Sft\right)=\left(\Sft\times_\Integral \Real\right)\rtimes\left\{\mathrm{id},\widetilde{\nu}\right\}.$$
	
	\begin{remark}	
	A transverse foliation to the canonical fibration of $\Sft$
	can be visualized as follows.
	Let $\Hyp^2$ be the hyperbolic plane.
	We adopt the upper half-space model
	$\Hyp^2=\{z\in\Complex\colon \Im(z)>0\}$,
	so $\mathrm{PSL}(2,\Real)$, and hence $\Sft\times_\Integral\Real$,
	acts isometrically on $\Hyp^2$ by fractional linear transformations.
	For any point $p\in\Hyp^2$, define
	$$\widetilde{k}_p\colon \Real\to \Sft$$
	as $\widetilde{k}_p(r)=h\widetilde{k}(r)h^{-1}$,
	for all $r\in\Real$,
	where $h\in\Sft$ is any element that takes the imaginary unit $\mathbf{i}$ to $p$.
	Note that $k_p$ depends only on the left coset $h\widetilde{\mathrm{SO}}(2)$,
	and hence only on $p$.
	We adopt the ideal boundary $\partial_\infty\Hyp^2=\Real\cup\{\infty\}$
	as the model of the real projective line.
	Denote by $\widetilde{\Real P^1}$ be the universal covering space of $\Real P^1$
	with a distinguished base point $\widetilde{0}$ that lifts $0$.
	Then there is a canonical homeomorphism
	$$\Sft\cong \Hyp^2\times \widetilde{\Real P^1},$$
	such that $g\in\Sft$ corresponds to $(g.\mathbf{i},g.\widetilde{0})$.
	The action of $\Sft\times_\Integral\Real$ on $\Sft$ is conjugate to 
	the action
	$$[g,r].(p,s)=\left(g.p,g\widetilde{k}_p(r).s\right).$$
	The projection of $\Hyp^2\times \widetilde{\Real P^1}$ onto the $\Hyp^2$ factor
	is a fibration that is invariant under the action of the center $\Real$;
	the foliation with leaves parallel to $\Hyp^2$ is transverse to the above fibration,
	with structure group $\Sft$.
	\end{remark}

\section{Volume of motion representations in Seifert geometry}\label{Sec-Seifert_volume}
	In this section, we recall volume of representations,
	concentrating on the special case of Seifert geometry and its motion group.
	The corresponding representation volume 
	recovers the Seifert volume for $3$--manifolds
	as originally introduced by Brooks and Goldman \cite{BG2}.
	The reader is referred to \cite[Section 2]{DLSW-rep_vol} 
	for complete details in a general setting.	
	
	For Seifert-geometric circle bundles over closed surfaces,
	the space of Seifert motion representations 
	and the volume function	are very well understood.
	We determine the volume function as a basic example of the theory
	(Theorem \ref{volume_central_bundle}).
	In subsequent sections,
	this preliminary result is used to prove the graph-manifold case,
	and its statement also serves as a prototype of the general formula  
	(Theorem \ref{volume_general_GM}).			
		
	\subsection{Volume of representations}\label{Subsec-volume_of_representations}
	For any finitely generated group $\pi$, the set of homomorphisms $\pi\to\mathrm{PSL}(2,\Real)$ 
	naturally forms a real algebraic set, known as the $\mathrm{PSL}(2,\Real)$--representation variety of $\pi$.
	For any Lie group $G$, we still call any homomorphism $\pi\to G$ a \emph{representation} of $\pi$ in $G$.
	The set of representations is no longer an algebraic variety in general.
	However, the set can be naturally topologized, so that 
	a sequence of representations converges 
	if and only if it converges in $G$ evaluated at any element of $\pi$.
	We denote by $\mathcal{R}(\pi,G)$ the set of $G$--representations of $\pi$
	with this topology,	and refer to it the \emph{space of $G$--representations} for $\pi$.
	
	Given any representation $\rho\colon \pi\to \Sft\times_\Integral\Real$,
	we can twist $\rho$ with any homomorphism $\alpha\colon\pi\to \Real$,
	and obtain a new representation $\rho[\alpha]\colon \pi\to \Sft\times_\Integral\Real$,
	such that 
	\begin{equation}\label{twisting_representation}
	(\rho[\alpha])(\sigma)=\rho(\sigma)\cdot\mathrm{id}[\alpha(\sigma)].
	\end{equation}
	This gives rise to a free, continuous action of the Lie group $H^1(\pi;\Real)$
	on $\mathcal{R}(\pi,\Sft\times_\Integral\Real)$.
	A similar construction with integral coefficients yields
	free, discontinuous action of the discrete group 
	$H^1(\pi;\Integral)$ on $\mathcal{R}(\pi,\Sft)$.
	Given any representation $\bar\rho\colon \pi\to \mathrm{PSL}(2,\Real)$,
	we obtain an associated oriented circle bundle $B\pi\times_{\bar{\rho}}\Real P^1$ 
	over the classifying space $B\pi\simeq K(\pi,1)$,
	using the canonical action of $\mathrm{PSL}(2,\Real)$ on the real projective line
	$\Real P^1$.
	This gives rise to an Euler class of $e(\bar{\rho})\in H^2(\pi;\Integral)$.
	
	The structure of $\mathcal{R}(\pi,\Sft\times_\Integral\Real)$
	is described through the following theorem:
	
	\begin{theorem}[{\cite[Proposition 5.1]{DLSW-rep_vol}}]\label{lifting_representations}
	For any finitely generated group $\pi$, the following statements hold true:
	\begin{enumerate}
	\item 
	The space of representations $\mathcal{R}(\pi,\mathrm{PSL}(2,\Real))$ has only finitely many path-connected components.
	The Euler class $e\colon\mathcal{R}(\pi,\mathrm{PSL}(2,\Real))\to H^2(\pi;\Integral)$ is constant on each component.
	\item 
	The space of representations $\mathcal{R}(\pi,\Sft)$ naturally projects 
	the path-connected components of $\mathcal{R}(\pi,\mathrm{PSL}(2,\Real))$
	where the Euler class is trivial,
	and becomes a principal $H^1(\pi;\Integral)$--bundle over those components.
	\item	
	The space of representations $\mathcal{R}(\pi,\Sft\times_\Integral\Real)$ naturally projects 
	the path-connected components of $\mathcal{R}(\pi,\mathrm{PSL}(2,\Real))$
	where the Euler class is torsion,
	and becomes a principal $H^1(\pi;\Real)$--bundle over those components.
	\end{enumerate}
	\end{theorem}
	
	Let $M$ be an oriented connected closed $3$--manifold. 
	By denoting $\pi_1(M)$, we always assume that a universal covering space of $M$
	is implicitly fixed, and $\pi_1(M)$ refers to the deck transformation group.
	There is a well-defined continuous function
	\begin{eqnarray*}
	\mathcal{R}\left(\pi_1(M),\Sft\times_\Integral\Real\right) \to \Real 
	&\colon &
	\rho \mapsto \mathrm{vol}_{\Sft\times_\Integral\Real}(M,\rho),
	\end{eqnarray*}
	called the \emph{volume} of $\Sft\times_\Integral\Real$--representations for $M$.
	Moreover, the function $\mathrm{vol}_{\Sft\times_\Integral\Real}(M,\cdot)$
	is constant on every path-connected component of $\mathcal{R}\left(\pi_1(M),\Sft\times_\Integral\Real\right)$.
	It follows from Theorem \ref{lifting_representations}
	that there are only finitely many path-connected components.
	The \emph{Seifert volume} of $M$ is defined as
	\begin{equation}\label{SV_definition}
	\mathrm{SV}(M)=\sup\left\{\left|\mathrm{vol}_{\Sft\times_\Integral\Real}(M,\rho)\right|\colon
	\rho\in\mathcal{R}\left(\pi_1(M),\Sft\times_\Integral\Real\right)\right\},
	\end{equation}
	which values in $[0,+\infty)$.
	
	Strictly speaking, the volume function depends on a chosen left-invariant volume form
	on $\Sft$, while other choices change the function by a nonzero scalar factor,
	(see \cite[Section 2]{DLSW-rep_vol}).
	However, there is a preferred choice 
	such that the Seifert volume agrees 
	with Brooks and Goldman's original definition.
	Indeed, 
	this is the normalization that applies to Theorem \ref{volume_central_bundle} below.
	See Section \ref{Sec-normalization_SV} in the appendix for an elaboration about the normalization.
	
	\subsection{Seifert geometric circle bundles}
	Let $N$ be an oriented circle bundle of Euler number $e\in\Integral$
	over an oriented closed surface $\Sigma$ of genus $g$.
	If $N$ supports the Seifert geometry,	$g$ is at least $2$ and $e$ is nonzero.
	
	The space of representations $\mathcal{R}(\pi_1(N),\Sft\times_\Integral\Real)$
	can be completely described as follows.
	The abelianization $\pi_1(N)\to H_1(N;\Integral)$ 
	and the bundle projection $\pi_1(N)\to \pi_1(\Sigma)$
	induce the embeddings of 
	$\mathcal{R}(H_1(N;\Integral),\mathrm{PSL}(2,\Real))$
	and 
	$\mathcal{R}(\pi_1(\Sigma),\mathrm{PSL}(2,\Real))$
	into $\mathcal{R}(\pi_1(N),\mathrm{PSL}(2,\Real))$,
	as real algebraic sets.
	The intersection of their images can be identified as
	$\mathcal{R}(H_1(\Sigma;\Integral),\mathrm{PSL}(2,\Real))$
	embedded in
	$\mathcal{R}(\pi_1(N),\mathrm{PSL}(2,\Real))$.
	Observe $H_1(N;\Integral)\cong \Integral^{2g}\oplus\Integral/e\Integral$.
	There are $|e|$ path-connected components of $\mathcal{R}(H_1(N;\Integral),\mathrm{PSL}(2,\Real))$,
	indexed by the $\Integral/e\Integral$.
	There are $(4g-3)$ path-connected components of $\mathcal{R}(\pi_1(\Sigma),\mathrm{PSL}(2,\Real))$,
	indexed by $0,\pm1,\pm2,\cdots,\pm(2g-2)$,
	which indicate to the Euler number $e(\psi)\in\Integral$ 
	of the associated oriented circle bundle
	$\Sigma\times_\psi\Real P^1\to\Sigma$ for any representation 
	$\psi\colon \pi_1(\Sigma)\to \mathrm{PSL}(2,\Real)$ in them,
	\cite{Goldman-components}.
	The representation variety $\mathcal{R}(H_1(\Sigma;\Integral),\mathrm{PSL}(2,\Real))$
	is path-connected.
	Note that any representation $\pi_1(N)\to \mathrm{PSL}(2,\Real)$
	either factors through the abelianization of $\pi_1(N)$, 
	or trivializes the center of $\pi_1(N)$.	
	Therefore, 
	the representation variety $\mathcal{R}(\pi_1(N),\mathrm{PSL}(2,\Real))$
	has $(4g-4+|e|)$ path-connected components in total:
	Apart from one that contains the trivial representation,
	there are $(|e|-1)$ path-connected components that consist only of abelian representations,
	and another $(4g-4)$ that consist only of nonabelian ones.
	By Theorem \ref{lifting_representations} and the fact that $H^2(\Sigma;\Integral)\to H^2(N;\Integral)$
	has finite-cyclic image of order $|e|$,
	the space of representations $\mathcal{R}(\pi_1(N),\Sft\times_\Integral\Real)$ projects onto
	$\mathcal{R}(\pi_1(N),\mathrm{PSL}(2,\Real))$
	as a principal $H^1(N;\Real)$--bundle.
	
	We say that a representation $\phi\in\mathcal{R}(\pi_1(N),\Sft\times_\Integral\Real)$
	is of \emph{central type} if $\phi$ sends the center of $\pi_1(N)$
	to the center of $\Sft\times_\Integral\Real$.
	For any central type representation $\phi$,
	we obtain the following a commutative diagram of group homomorphisms:
	\begin{equation}\label{diagram_central_SFS}
	\xymatrix{
	\{0\} \ar[r]  & \Integral \ar[r] \ar[d]_{\phi_{\mathtt{fib}}} 
		& \pi_1(N) \ar[r] \ar[d]_{\phi} & \pi_1(\Sigma) \ar[r] \ar[d]^{\phi_{\mathtt{base}}}& \{1\} 	\\
	\{0\} \ar[r]  & \Real \ar[r]   & \Sft\times_\Integral\Real \ar[r] & \mathrm{PSL}(2,\Real) \ar[r] & \{1\} 	
	}
	\end{equation}
	We identify $\phi_{\mathtt{fib}}\colon \Integral\to \Real$ as a real scalar,
	and the Euler class $e(\phi_{\mathtt{base}})\in H^2(\Sigma;\Integral)$ as an integer.
		
	\begin{theorem}\label{volume_central_bundle}
		Let $N$ be an oriented circle bundle over an oriented closed surface $\Sigma$.
		Suppose the genus of $\Sigma$ is $g>0$ and the Euler number of $N\to\Sigma$ is $e\neq0$.
		Suppose that
		$\phi\colon\pi_1(N)\to \Sft\times_\Integral\Real$ is a representation of central type.
		Then the formulas hold:
		$$e\left(\phi_{\mathtt{base}}\right)=e\times \phi_{\mathtt{fib}},$$
		and
		$$\mathrm{vol}_{\Sft\times_\Integral\Real}(N,\phi)=4\pi^2\phi_{\mathtt{fib}}\times e(\phi_{\mathtt{base}}).$$
		In particular,
		$\phi_{\mathtt{fib}}$ is constant on 
		the path-connected component of $\mathcal{R}(\pi_1(N),\Sft\times_\Integral\Real)$
		that contains $\phi$,
		and the possible values are precisely
		$0,\pm 1/e,\cdots,\pm(2g-2)/e$.		
	\end{theorem}
	
	\begin{proof}
		If $e(\phi_{\mathtt{base}})=0$, the representation $\phi$ lies in the path-connected component
		of the trivial representation, by the above description of $\mathcal{R}(\pi_1(N),\Sft\times_\Integral\Real)$.
		As the trivial representation has volume $0$,
		we obtain $\mathrm{vol}(N,\phi)=0.$
		The path-connected component of the trivial representation 
		can be identified as the image of 
		the pull-back embedding $\mathcal{R}(\pi_1(\Sigma),\Sft\times_\Integral\Real)\to\mathcal{R}(\pi_1(N),\Sft\times_\Integral\Real)$,
		by Proposition \ref{lifting_representations},
		so we obtain $\phi_{\mathtt{fib}}=0$. 
		This establish the formulas for $e(\phi_{\mathtt{base}})=0$.
		
		It remains to prove the formulas for $e(\phi_{\mathtt{base}})\neq0$.
		To this end, we consider the associated circle bundle $\Sigma\times_{\phi_{\mathtt{base}}}\Real P^1$
		over $\Sigma$, whose Euler number equals $e(\phi_{\mathtt{base}})$.
		There is a natural transversely projective foliation $\mathfrak{F}$
		on $\Sigma\times_{\phi_{\mathtt{base}}}\Real P^1$, which gives rise to 
		a canonical holonomy representation
		$$\psi\colon\pi_1(\Sigma\times_{\phi_{\mathtt{base}}}\Real P^1)\to\Sft\times_\Integral\Real,$$
		which actually factors through $\Sft$.
		To be precise, 
		we treat $\pi_1(\Sigma\times_{\phi_{\mathtt{base}}} \Real P^1)$
		as the deck transformation group of 
		the universal covering space $\widetilde{\Sigma}\times\widetilde{\Real P^1}$
		of $\Sigma\times_{\phi_{\mathtt{base}}} \Real P^1$,
		where $\widetilde{\Sigma}$ and $\widetilde{\Real P^1}$
		are fixed universal spaces of $\Sigma$ and $\Real P^1$,
		respectively.
		The covering projection is the composition of
		the regular covering projections 
		$\widetilde{\Sigma}\times\widetilde{\Real P^1}\to \widetilde{\Sigma}\times\Real P^1$
		and
		$\widetilde{\Sigma}\times\Real P^1\to \Sigma\times_{\phi_{\mathtt{base}}}\Real P^1$,
		with deck transformation groups $\pi_1(\Real P^1)$ and $\pi_1(\Sigma)$,
		respectively.
		The transversely projective foliation $\mathfrak{F}$
		on $\Sigma\times_{\phi_{\mathtt{base}}} \Real P^1$
		is therefore the quotient of the folation on $\widetilde{\Sigma}\times\widetilde{\Real P^1}$
		with leaves parallel to $\widetilde{\Sigma}$.
		The holonomy representation $\psi$ is 
		the action of $\pi_1(\Sigma\times_{\phi_{\mathtt{base}}} \Real P^1)$
		on $\widetilde{\Real P^1}$ factor.
		Note that $\pi_1(\Real P^1)$ is canonically isomorphic
		to $\Integral$, and acts on $\widetilde{\Real P^1}$ 
		the same way as the center $\Integral$ of $\Sft$.
		This means 
		$$\psi_{\mathtt{fib}}=+1.$$
		Note
		$$\psi_{\mathtt{base}}=\phi_{\mathtt{base}}.$$
		The volume of $(\Sigma\times_{\phi_{\mathtt{base}}} \Real P^1,\psi)$ can be computed through
		Godbillon--Vey class of $\mathfrak{F}$:
		\begin{equation}\label{computation_gv}
		\mathrm{vol}_{\Sft\times_\Integral\Real}(\Sigma\times_{\phi_{\mathtt{base}}} \Real P^1,\psi)
		=\int_\Sigma\int_{\Real P^1}\mathrm{gv}(\mathfrak{F})
		=4\pi^2e\left(\phi_{\mathtt{base}}\right),
		\end{equation}
		see \cite[Proposition 2]{BG2} (and also Section \ref{Sec-normalization_SV}). 
		
		As oriented circle bundles over $\Sigma$ are classified by $H^2(\Sigma;\Integral)\cong\Integral$,
		denote by $\Sigma\times_m S^1$ the oriented circle bundle over $\Sigma$	of Euler number $m\in\Integral$.
		If $m$ is nonzero, there is a fiber-preserving map $\Sigma\times_{1}S^1\to\Sigma\times_mS^1$ of degree $m$,
		which is an cyclic covering of signed degree $m$ restricted to every fiber.
		In particular, 
		there are such maps $\Sigma\times_1S^1\to N$ of degree $e$, and 
		$\Sigma\times_1S^1\to \Sigma\times_{\phi_{\mathtt{base}}}\Real P^1$ 
		of degree $e(\phi_{\mathtt{base}})\neq0$.
		Denote by $\bar\phi\colon \pi_1(N)\to \mathrm{PSL}(2,\Real)$ and 
		$\bar\psi\colon \pi_1(\Sigma\times_{\phi_{\mathtt{base}}}\Real P^1)\to \mathrm{PSL}(2,\Real)$
		the pull-backs of $\phi_{\mathtt{base}}$.
		The $\Sft\times_\Integral\Real$--representation $\phi$ lifts $\bar\phi$.
		Take some $\Sft\times_\Integral\Real$--representation $\psi$ that lifts $\bar\psi$.
		Denote by 
		$$\phi'\colon \pi_1\left(\Sigma\times_1S^1\right)\longrightarrow \pi_1(N)\stackrel{\phi}\longrightarrow \Sft\times_\Integral\Real$$
		and 
		$$\psi'\colon \pi_1\left(\Sigma\times_1S^1\right)\longrightarrow 
		\pi_1\left(\Sigma\times_{\phi_{\mathtt{base}}}\Real P^1\right)
		\stackrel{\psi}\longrightarrow \Sft\times_\Integral\Real$$
		the pull-back representations of $\pi_1(\Sigma\times_1S^1)$.
		Note that both $\phi$ and $\psi$ factors through $\phi_{\mathtt{base}}\colon \pi_1(\Sigma)\to \mathrm{PSL}(2,\Real)$.
		It follows that $\phi'$ and $\psi'$ lies in the same component of 
		$\mathcal{R}(\pi_1(\Sigma\times_1S^1),\Sft\times_\Integral\Real)$.
		We obtain
		$$e\times\mathrm{vol}(N,\phi)=
		\mathrm{vol}\left(\Sigma\times_1S^1,\phi'\right)=
		\mathrm{vol}\left(\Sigma\times_1S^1,\psi'\right)=
		e(\phi_{\mathtt{base}})\times\mathrm{vol}\left(\Sigma\times_{\phi_{\mathtt{base}}}\Real P^1,\psi\right),$$
		using (\ref{computation_gv}).
		This implies
		\begin{equation}\label{computation_eb2}
		e\times\mathrm{vol}(N,\phi)=4\pi^2e(\phi_{\mathtt{base}})^2.
		\end{equation}
		
		We show that $\phi'_{\mathtt{fib}}$ and $\psi'_{\mathtt{fib}}$ are equal,
		by arguing that 
		the fiber image is constant on any path-connected component
		of $\mathcal{R}(\pi_1(\Sigma\times_1S^1),\Sft\times_\Integral\Real)$.
		In fact,
		because $H^2(\Sigma\times_1S^1;\Integral)\cong H^1(\Sigma;\Integral)$ is torsion-free,
		$\mathcal{R}(\pi_1(\Sigma\times_1S^1),\Sft)$ intersects
		every path-connected component of $\mathcal{R}(\pi_1(\Sigma\times_1S^1),\Sft\times_\Integral\Real)$,
		(Theorem \ref{lifting_representations}).
		The fiber image is constant on any path-connected component
		of $\mathcal{R}(\pi_1(\Sigma\times_1S^1),\Sft)$,
		as it takes discrete values in $\Integral$.
		The fiber image is also constant for under the action of $H^1(\Sigma\times_1S^1;\Real)$,
		as the fiber is null-homologous in $\Sigma\times_1 S^1$,
		(see (\ref{twisting_representation})).
		Therefore, the fiber image is constant on any path-connected component
		of $\mathcal{R}(\pi_1(\Sigma\times_1S^1),\Sft\times_\Integral\Real)$,
		by Theorem \ref{lifting_representations}.
		We obtain 
		\begin{equation}\label{computation_ef}
		e\times\phi_{\mathtt{fib}}=\phi'_{\mathtt{fib}}=\psi'_{\mathtt{fib}}=e\left(\phi_{\mathtt{base}}\right)\times\phi_{\mathtt{fib}}
		=e\left(\phi_{\mathtt{base}}\right).
		\end{equation}
		Then asserted formulas follow from (\ref{computation_eb2}) and (\ref{computation_ef}) for $e(\phi_{\mathtt{base}})\neq0$.
		The possible values for $\phi_{\mathtt{fib}}$ are determined,
		since $e(\phi_{\mathtt{base}})$
		can take precisely the values $0,\pm1,\cdots,\pm(2g-2)$.
	\end{proof}
		
	\begin{remark}\
	\begin{enumerate}
		\item 
		The volume function on $\mathcal{R}(\pi_1(N),\Sft\times_\Integral\Real)$
		is completely determined by Theorem \ref{volume_central_bundle}.
		This is because any representation that lies over an abelian component of 
		the $\mathrm{PSL}(2,\Real)$--representation variety	must have volume $0$.
		In fact, any such representation is path connected with one that lies over 
		a $\mathrm{PSL}(2,\Real)$--representation with finite cyclic image.
		\item
		Theorem \ref{volume_central_bundle} refines \cite[Proposition 6.3]{DLW-cs_sep_vol}
		in the case of Seifert-geometric circle bundles.
		The point is that here we have enumerated all the actual volume values,
		thanks to Goldman's result on $\mathrm{PSL}(2,\Real)$--representation varieties of surface groups
		\cite{Goldman-components}.
		As the $\mathrm{PSL}(2,\Real)$--representation varieties of Fuchsian groups
		have been described in Jankins--Neumann \cite{JN},
		it seems possible to obtain analogous refinement for all closed Seifert fiber spaces,
		with only extra notations and reductions.
	\end{enumerate}
	\end{remark}

\section{Topology of graph manifolds}\label{Sec-graph_manifold}
	In this section, we review Seifert fiber spaces and graph manifolds.
	We recall a collection of numerical data 
	that largely describes the topology of a graph manifold,
	including the intersection number on any JSJ torus between the adjacent fibers, 
	and the Euler number of the JSJ pieces relative to the framing given by adjacent fibers.
	We organize the data and study behavior of the data under finite covering.
	For standard facts in $3$--manifold topology, we refer to \cite{Hempel-book,AFW-3mg}.
	
	In the literature, different authors refer to wider or narrower classes of $3$--manifolds
	when they use the term \emph{graph manifold}. 
	We restrict ourselves to a reasonably general class, 
	where the Seifert fibrations of the JSJ pieces are all orientable over orientable bases.
	We actually remember the orientations as auxiliary extra data,
	and assume the JSJ graph to be simplicial,
	so as to write down our volume formula in a convenient way (Theorem \ref{volume_general_GM}).
	This leads to what we call a \emph{formatted graph manifold},
	as introduced in Definition \ref{def_format}.

	\subsection{Seifert fiber spaces}
		Recall that \emph{Seifert fiber space} is a closed $3$--manifold together with a foliation by circles.
		Every leaf admits a foliated tubular neighborhood that is topologically modeled on
		a foliated open solid torus $\mathrm{U}(1)\times\Complex$.
		The leaves are the orbits of the $\mathrm{U}(1)$--action
		$u.(w,z)=(u^\alpha w,u^{\beta}.z)$, 
		for all $u\in \mathrm{U}(1)$ and $(w,z)\in\mathrm{U}(1)\times \Complex$,
		where $\beta$ is some nonzero integer and $\alpha$
		is some integer coprime to $\beta$.
		There are at most finitely many leaves that are locally modeled with $|\beta|>1$,
		which are called the \emph{exceptional fibers}.
		Any other leaf is called an \emph{ordinary fiber}.
		
		Circle bundles over closed surfaces
		are Seifert fiber spaces without exceptional fibers.
		In general, 
		it is often instructive to think of a Seifert fiber space $N$
		analogously as a circle bundle over a closed $2$--orbifold $\mathcal{O}$,
		where $\mathcal{O}$ 
		can be concretely constructed as the leaf space of the foliation.
		The orbifold Euler characteristic $\chi(\mathcal{O})\in\Rational$
		plays the same role as the usual Euler characteristic of the base
		for genuine circle bundles. 
		When the Seifert fibration and the base $2$--orbifold are both oriented,
		the Euler number $e(N\to\mathcal{O})\in\Rational$
		can be defined as an analogue of the usual Euler number
		for oriented circle bundles over oriented closed surfaces.
		
		The isomorphism type of any closed orientable $2$--orbifold
		can be uniquely described with a \emph{symbol}
		$(g;\beta_1,\cdots,\beta_k)$,
		which means genus $g\in\{0,1,2,\cdots\}$ with $k$ cone points 
		of orders $\beta_i>1$, respectively.
		The orbifold Euler characteristic can be computed by the formula
		\begin{equation}\label{chi_O_symbol}
		\chi(\mathcal{O})=2-2g-k+\sum_{i=1}^{k}\frac{1}{\beta_i}.
		\end{equation}
		
		For any oriented closed $2$--orbifold $\mathcal{O}$ of symbol
		$(g;\beta_1,\cdots,\beta_k)$,
		the oriented isomorphism type
		of any oriented Seifert fibration $N\to\mathcal{O}$ 
		can be uniquely described with a \emph{normalized symbol}
		$(g;(1,\alpha_0),(\beta_1,\alpha_1),\cdots,(\beta_k,\alpha_k))$,
		where $\alpha_0$ is an integer,
		and where $\alpha_i\in\{1,2,\cdots,\beta_i-1\}$ 
		are coprime to $\beta_i$ for $i=1,\cdots,k$.
		The Euler number of the Seifert fibration can be computed by the formula
		\begin{equation}\label{e_N_symbol}
		e\left(N\to\mathcal{O}\right)=\alpha_0+\sum_{i=1}^{k}\frac{\alpha_i}{\beta_i}.
		\end{equation}		
		
	\subsection{Graph manifolds}		
		An orientable prime closed $3$--manifold 
		is said to be a \emph{graph manifold}	
		if there exists a possibly empty finite collection of 
		mutually disjoint incompressible tori,
		and if the complement of their union admits a foliation by circles.
		
		According to the Jaco--Shalen--Johanson (JSJ) decomposition theory
		in $3$--manifold topology,
		every orientable prime closed $3$--manifold
		contains an isotopically unique, minimal finite collection 
		of mutually disjoint incompressible tori,
		such that the following mutually exclusive dichotomy 
		holds for every component	of the complement of their union:
		Either the component admits a foliation by circles,
		or the component does not contain any incompressible tori 
		which does not bound any toral end.
		The tori are called the \emph{JSJ tori}, 
		and the complementary components are called the \emph{JSJ pieces},
		or more specifically,
		the \emph{Seifert fibered pieces} and the \emph{atoroidal pieces},
		according to the dichotomy.
		Therefore, graph manifolds are precisely those orientable prime $3$--manifolds
		whose JSJ decomposition has no atoroidal pieces.
		The \emph{JSJ graph} is the dual graph of the JSJ decomposition:
		As a cell $1$--complex, the vertices ($0$--cells) and the edges ($1$--cells) 
		correspond the JSJ pieces and the JSJ tori, respectively,
		and the attaching maps are indicated by the adjacency relation
		between the JSJ objects.	
		
		\begin{definition}\label{def_format}
			We say that a closed orientable prime $3$--manifold $M$ is 
			a \emph{formattable graph manifold},
			if the JSJ graph of $M$ is a simplicial graph,
			and if the JSJ pieces of $M$ all admit orientable Seifert fibration 
			over orientable $2$--orbifolds (possibly with punctures).
			A \emph{formatted graph manifold} is defined 
			as an oriented formattable graph manifold $M$
			enriched with an oriented Seifert fibration structure for every JSJ piece of $M$.
		\end{definition}
		
		\begin{notation}\label{notation_xkb}
			For any formatted graph manifold $M$ with a simplicial JSJ graph $(V,E)$, 
			we introduce the following notations.
			\begin{itemize}
			\item For every vertex $v\in V$, denote by $J_v\to\mathcal{O}_v$ 
			the oriented Seifert fibration of the corresponding JSJ piece.
			Denote by
			$$\chi_v=\chi(\mathcal{O}_v)\in\Rational$$
			the orbifold Euler characteristic of $\mathcal{O}_v$.
			\item For every edge $\{v,w\}\in E$, denote by $T_{v,w}$ the corresponding JSJ torus 
			with the outward orientation with respect to the orientation of $J_v$.
			Denote by
			$$b_{v,w}=I_{v,w}\left([f_v],[f_w]\right)\in\Integral$$
			the algebraic intersection number between the oriented slopes $f_v$ and $f_w$ on $T_{v,w}$ 
			which are parallel to the ordinary fibers of $J_v$ and $J_w$, respectively.	
			Note that $T_{w,v}$ refers to the same JSJ torus with the reversed orientation,
			so $b_{w,v}=b_{v,w}$ depends only on $\{v,w\}\in E$.			
			\item For every vertex $v\in V$, denote by 
			$\hat{J}_v$ the oriented closed $3$--manifold
			obtained from the Dehn filling of $J_v$ along the slopes $f_w$ on $T_{v,w}$,
			for all $w$ incident to $v$.
			Note that $J_v\to\mathcal{O}_v$ extends to a unique oriented Seifert fibration,
			over the closed oriented $2$--orbifold $\hat{\mathcal{O}}_v$ 
			obtained from $\mathcal{O}_v$	by filling all punctures.
			Denote by
			$$k_v=e\left(\hat{J}_v\to\hat{\mathcal{O}}_v\right)\in\Rational$$
			the Euler number of the oriented Seifert fibration $\hat{J}_v\to\hat{\mathcal{O}}_v$.
			\end{itemize}
		\end{notation}
		
		\begin{remark}\label{def_format_remark}
			The quantity $k_v$ is sometimes called the \emph{charge} of $M$ at the vertex $v$.
			The quantities as summarized in Notation \ref{notation_xkb} have been used in Buyalo--Svetlov \cite{BS-graph_manifold}
			to characterize a list of significant properties, such as the existence of a nonpositively curved Riemannian metric,
			and the existence of a virtual fibering over the circle, etc.
		\end{remark}
		
		For any formatted graph manifold $M$ with the simplicial graph $(V,E)$, 
		a \emph{Waldhausen basis} can be chosen for the outward oriented JSJ tori $T_{v,w}$
		that are adjacent to a JSJ piece $J_v$.
		To be precise, this is a basis $H_1(T_{v,w};\Rational)$ of the form
		$$\left([f_v],[s_{v,w}]\right)$$ as follows:
		The homology class $[f_v]\in H_1(T_{v,w};\Rational)$ is represented by any oriented ordinary fiber of $J_v$;
		the homology classes $[s_{v,w}]\in H_1(T_{v,w};\Rational)$ satisfy the relation
		$\sum_{\{v,w\}\in E}[s_{v,w}]=0$
		in $H_1(J_v;\Rational)$, under the natural inclusions $H_1(T_{v,w};\Rational)\to H_1(J_{v,w};\Rational)$;
		and the algebraic intersection number on $H_1(T_{v,w};\Rational)$ between $[f_v]$ and $[s_{v,w}]$ satisfies
		$I_{v,w}([f_v],[s_{v,w}])=+1$,	for all $\{v,w\}\in E$.
		With respect to given Waldhausen bases $([f_v],[s_{v,w}])$ for all $T_{v,w}$,
		the charge at vertices can be expressed explicitly as
		\begin{equation}\label{charge_Waldhausen}
		k_v=\sum_{\{v,w\}\in E} \frac{a_{v,w}}{b_{v,w}},
		\end{equation}
		where $a_{v,w}\in\Rational$ are coefficients determined uniquely by the linear combination relation
		\begin{equation}\label{b_Waldhausen}
		[f_w]=a_{v,w}[f_v]+b_{v,w}[s_{v,w}]
		\end{equation}
		in $H_1(T_{v,w};\Rational)$.
		The existence of Waldhausen bases follows easily from the homology of orientable Seifert fibrations
		over orientable $2$--orbifolds with punctures.
		If $J_v\to\mathcal{O}_v$ is a trivial circle bundle over an oriented surface,
		one may choose oriented slopes $s_{v,w}$ on $T_{v,w}$ 
		so that they bound some horizontal lift of the base surface.
		In general, we do not require $[s_{v,w}]$ to have slope representives.
		
	\subsection{Covering graph manifolds}
	
		\begin{definition}\label{def_formatted_covering}
			A \emph{formatted covering projection} is defined 
			as a covering projection $M'\to M$ between formatted graph manifolds
			which preserves the orientations of the manifolds,
			and which respects the oriented Seifert fibration structures on the JSJ pieces,
			mapping fibers onto fibers preserving their orientations.
		\end{definition}
		
		\begin{proposition}\label{virtual_xkb}
		Suppose that $M'\to M$ is a formatted covering projection between formatted graph manifolds.
		\begin{enumerate}
		\item For any covering pair of JSJ pieces $J'_{v'}\to J_v$,
		$$
		\chi_{v'} =\frac{[J'_{v'}:J_{v}]}{[f'_{v'}:f_v]}\times \chi_v. 
		$$
		\item For any covering pair of JSJ pieces $J'_{v'}\to J_v$,
		$$k_{v'}=\frac{[J'_{v'}:J_{v}]}{[f'_{v'}:f_v]^2}\times k_v.$$
		\item For any covering pair of JSJ tori $T'_{v',w'}\to T_{v,w}$,
		$$b_{v',w'}=\frac{[f'_{v'}:f_{v}]\times[f'_{w'}:f_{w'}]}{[T'_{v',w'}:T_{v,w}]}\times b_{v,w}.$$
		\item For any covering pairs of JSJ pieces $J'_{v'}\to J_v$ and any JSJ torus $T_{v,w}$,
		$$\sum_{\{v',w'\}\in E'}\frac{1}{b_{v',w'}}=
		\frac{[J'_{v'}:J_{v}]}{[f'_{v'}:f_v]^2}\times\frac{1}{b_{v,w}},$$
		\end{enumerate}
		Here, the notation $[-:-]$ stands for the unsigned covering degree.
		\end{proposition}
		
		\begin{proof}			
			For any covering pair of JSJ pieces $J'_{v'}\to J_v$,
			there is a covering projection between the base $2$--orbifolds $\mathcal{O}'_{v'}\to\mathcal{O}_{v}$.
			Denote by $J^*_{v'}\to\mathcal{O}'_{v'}$ the pull-back of the Seifert fibration $J_v\to\mathcal{O}_v$
			using $\mathcal{O}^*_{v'}\to \mathcal{O}_{v}$. 
			Then $J'_{v'}\to J_v$ factorizes through the intermediate cover $J^*_{v'}$ of $J_v$,
			and $J'_{v'}\to J^*_{v'}$ 
			is obtained from a fiber-preserving free action of a finite cyclic group, whose order equals $[f_{v'}:f_v]$.
			
			For any covering pair of JSJ tori $T'_{v',w'}\to T_{v,w}$,
			the factorization $J'_{v'}\to J^*_{v'}\to  J_v$ 
			induces an intermediate cover $T^*_{v',w'}$ and a factorization
			$T'_{v',w'}\to T^*_{v',w'}\to T_{v,w}$.
			The slope $f'_{w'}$ of $T'_{v',w'}$ projects homeomorphically onto a slope $f^*_{w'}$ of $T^*_{v',w'}$.
			Denote by $\hat{T}'_{v',w'}$ and $\hat{T}^*_{v',w'}$ 
			the solid tori obtained from $T'_{v',w'}$ and $T^*_{v',w'}$
			by Dehn fillings along $f'_{w'}$ and $f^*_w$, respectively.
			Therefore, the covering projection $T'_{v',w'}\to T_{v,w}$
			naturally extend to a covering projection 
			$\hat{T}'_{v',w'}\to \hat{T}^*_{v',w'}$.
			The covering projection $T^*_{v',w'}\to T_{v,w}$ 
			extends to a branched covering map $\hat{T}^*_{v',w'}\to \hat{T}_{v,w}$,
			ramifying along the core curve with index $[f'_{w'}:f_{w}]$.
			
			Perform similar constructions for all boundary components of $J'_{v'}$, $J^*_{v'}$, and $J_v$.
			Then we obtain a covering projection between the Dehn fillings $\hat{J}'_{v'}\to \hat{J}^*_{v'}$,
			which identifies the base $2$--orbifolds as $\hat{\mathcal{O}}'_{v'}$.
			The covering degree of $\hat{J}'_{v'}\to \hat{J}^*_{v'}$
			restricted to fibers equals $[f'_{v'}:f_v]$.
			We also obtain a branched covering map $\hat{J}^*_{v'}\to \hat{J}_{v}$,
			which identifies $\hat{J}^*_{v'}\to\hat{\mathcal{O}}^*_{v'}$ as the pull-back Seifert fibration
			of $\hat{J}_v\to\hat{O}_v$ 
			via a branched covering map $\hat{\mathcal{O}}'_{v'}\to \hat{\mathcal{O}}_{v'}$.
			The mapping degree of $\hat{\mathcal{O}}'_{v'}\to \hat{\mathcal{O}}^*_{v'}$
			equals $[\mathcal{O}'_{v'}:\mathcal{O}_{v}]$.
			The ramification index of $\hat{\mathcal{O}}'_{v'}\to \hat{\mathcal{O}}_{v}$ 
			at a cone point that corresponds to an edge $\{v,w\}$ equals 
			$[f'_{w'}:f_{w}]\times |b_{v,w}|\times|b_{v',w'}|^{-1}$.
			
			Compute the covering degree of $J'_{v'}\to J_v$ using the intermediate cover $J^*_{v'}$.
			We obtain the relation
			\begin{equation}\label{J_degree_compute}
			\left[J_{v'}:J_v\right]=\left[\mathcal{O}_{v'}:\mathcal{O}_{v}\right]\times\left[f_{v'}:f_{v}\right]
			=\frac{\chi_{v'}}{\chi_v}\times\left[f_{v'}:f_{v}\right].
			\end{equation}
						
			Compute charge using the composite map 
			$\hat{J}'_{v'}\to \hat{J}^*_{v'}\to J_v$.
			We obtain the relation 
			\begin{equation}\label{charge_compute}
			k_{v'}=e\left(\hat{J}'_{v'}\right)=\frac{1}{[f'_{v'}:f_v]}\times e\left(\hat{J}^*_{v'}\right)
			=\frac{[\mathcal{O}'_{v'}:\mathcal{O}_{v}]}{[f'_{v'}:f_v]}\times e\left(\hat{J}^*_{v'}\right)
			=\frac{1}{[f'_{v'}:f_v]}\times \frac{\chi_{v'}}{\chi_v}\times k_v.
			\end{equation}
			
			Compute the covering degree of $T'_{v',w'}\to T_{v,w}$ using 
			the composite map $\hat{T}'_{e'}\to \hat{T}^*_{v',w'}\to \hat{T}_{v,w}$.
			We obtain the relation
			\begin{equation}\label{T_degree_compute}
			\left[T'_{v',w'}:T_{v,w}\right]=\frac{[f'_{w'}:f_{w}]\times |b_{v,w}|}{|b_{v',w'}|}\times\left[f'_{v'}:f_v\right]
			=\left[f'_{w'}:f_{w}\right]\times\left[f'_{v'}:f_{v}\right]\times\frac{b_{v,w}}{b_{v',w'}}.
			\end{equation}
			
			We also observe the relation between covering degrees:
			\begin{equation}\label{T_degree_fixing_J}
			\frac{[J'_{v'}:J_v]}{[f'_{v'}:f_v]}=
			\sum_{\{v',w'\}\in E'}\frac{[T'_{v',w'}:T_{v,w}]}{[f'_{w'}:f_w]},			
			\end{equation}
			for any fixed $J'_{v'}\to J_v$ and $T_{v,w}$.
						
			The asserted formulas follow from
			(\ref{J_degree_compute}), (\ref{charge_compute}), (\ref{T_degree_compute}), and (\ref{T_degree_fixing_J})
			immediately.
		\end{proof}

\section{Seifert representations for graph manifolds: the statement}\label{Sec-volume_general_GM_statement}
	In this section, we give a precise statement of our main result (Theorem \ref{volume_general_GM}).
	The statement contains 
	a volume formula for $\Sft\times_\Integral\Real$--representations of formatted graph manifolds,
	and an estimate analogous to the Milnor--Wood inequality,
	and an assertion regarding rationality of the volume value.
	After that,
	we include some preliminary discussions
	to justify the formulation of Theorem \ref{volume_general_GM},
	and to explain the structure of its proof.
	The entire proof of Theorem \ref{volume_general_GM} occupies
	Sections \ref{Sec-vc_case}--\ref{Sec-proof_volume_general_GM} in the sequel.

	\begin{notation}\label{notation_xi_tau}
	For any formatted graph manifold $M$ with a simplicial JSJ graph $(V,E)$,
	and for any representation $\rho\colon \pi_1(M)\to \Sft\times_\Integral\Real$,
	we introduce the following notations.
	\begin{enumerate}
	\item 
	Denote by $\Real^V$ the linear space of real functions on $V$.
	We think of $\Real^V$ as furnished with the standard unordered basis
	$\{v^*\colon v\in V\}$, 
	where $v^*\in\Real^V$ stands for the characteristic function 
	$v^*(v)=1$ and $v^*(u)=0$ for all $u\neq v$.
	\item
	Denote by ${\chi}_M\in \Real^V$
	the rational vector such that
	$$\chi_M(v)=\chi_v.$$
	Denote by $\opEuler_M\in \mathrm{End}_\Real(\Real^V)$
	the linear operator such that
	$$\opEuler_Mv^*= k_v v^*-\sum_{\{v,w\}\in E}\frac{w^*}{b_{v,w}}.$$
	Note that $\opEuler_M$ is symmetric and rational 
	with respect to	the standard unordered basis.
	(See Notation \ref{notation_xkb}.)
	\item 
	Denote by $\xi_\rho\in\Real^V$ 
	the vector such that
	$$\xi_\rho(v)=
	\winding\left(\rho(f_v)\right).$$
	Here, $f_v$ is treated as an element of $\pi_1(M)$, and $\xi_\rho(v)$ depends only on its conjugacy class.
	(See Definition \ref{def_essential_winding_number} and Proposition \ref{prop_essential_winding_number}.)
	\item 
	For any edge $\{v,w\}\in E$, 
	denote by
	$$\tau_\rho(v,w)\in \Natural\cup\{\infty\}$$
	the order of the image of $\pi_1(T_{v,w})$ under 
	the induced representation $\bar\rho\colon \pi_1(M)\to \mathrm{PSL}(2,\Real)$.
	Here, $\pi_1(T_{v,w})$ is treated as a subgroup of $\pi_1(M)$,
	and $\tau_{\rho}(v,w)$ depends only on its conjugacy class.
	\end{enumerate}
	\end{notation}
	
	\begin{theorem}\label{volume_general_GM}
		Let $M$ be a formatted graph manifold with a simplicial JSJ graph $(V,E)$
		and	$\rho\colon\pi_1(M)\to \Sft\times_\Integral\Real$ be a representation. 
		Adopt Notation \ref{notation_xi_tau}.
		Then the following statements all hold true.
		\begin{itemize}
		\item 
		\emph{(Volume formula)}.
		$$\mathrm{vol}_{\Sft\times_\Integral\Real}(M,\rho)=4\pi^2\cdot\sum_{v\in V} \xi_\rho(v)\times(\opEuler_M\xi_\rho)(v).$$
		\item \emph{(Generalized Milnor--Wood Inequalities)}. For all vertices $v\in V$,
		$$|(\opEuler_M\xi_\rho)(v)|\leq\max\left\{0,-\chi_M(v)-\sum_{\{v,w\}\in E}\frac{1}{\tau_\rho(v,w)}\right\}.$$
		\item \emph{(Rationality)}.
		There is a solution $X\in\Rational^V$ to the equation
		$$\mathrm{vol}_{\Sft\times_\Integral\Real}(M,\rho)=4\pi^2\cdot\sum_{v\in V} X(v)\times(\opEuler_MX)(v).$$
		\end{itemize}
	\end{theorem}
		
	\begin{remark}\label{remark_volume_general_GM}\
	\begin{enumerate}
	\item Denote by	$(\cdot,\cdot)$ stands for the standard inner product on $\Real^V$.
	Then the volume formula in Theorem \ref{volume_general_GM} reads:
	$$\mathrm{vol}_{\Sft\times_\Integral\Real}(M,\rho)=4\pi^2\cdot(\xi_\rho,\opEuler_M\xi_\rho),$$
	which closely looks like the formula in Theorem \ref{volume_central_bundle}.
	\item Let $N$ be an oriented circle bundle of Euler number $e\neq0$ over an oriented surface of genus $g>1$.
	Suppose that $N$ admits a transversely projective codimension--$1$ foliation with holonomy 
	$\bar\phi\colon\pi_1(N)\to \mathrm{PSL}(2,\Real)$. 
	Then the estimate in Theorem \ref{volume_general_GM} (or Theorem \ref{volume_central_bundle})
	applies to any lift $\phi\colon\pi_1(N)\to \Sft\times_\Integral\Real$ of $\bar\phi$.
	The resulting inequality 
	$$|\xi_{\phi}|\leq |(2g-2)/e|$$
	agrees with the Milnor--Wood inequality \cite{Milnor,Wood}, as observed in \cite[Theorem 5 and Proof]{BG1}.
	\end{enumerate}
	\end{remark}
	
	Theorem \ref{volume_general_GM} essentially allows us
	to compute volume of Seifert representations 
	for any orientable closed graph manifold.
	This is because
	any non-geometric graph manifold admits a finite cover that is formattable.
	In fact, one may first construct a finite cover where all the JSJ pieces are product,
	(see \cite[Proposition 4.4]{LW});
	and then, construct a further finite cover
	that is	induced by a finite cover of the JSJ graph,
	such that the covering JSJ graph has no loops with two or fewer edges,
	(using residual finiteness of finitely generated free groups).
	The resulting cover is a formattable graph manifold.
	By contrast,
	for orientable closed $3$--manifolds of nonzero simplicial volume,
	no precise values of the Seifert volume have been determined,
	as far as the authors know.
	See \cite[Section 7]{Kh} for some examples with numerical computations.
	
	We obtain the following consequence from the rationality part in Theorem \ref{volume_general_GM},
	(see also Section \ref{Subsec-volume_of_representations}).
				
	\begin{corollary}\label{SV_rationality}
		For any orientable closed graph manifold $M$, the Seifert volume $\mathrm{SV}(M)$ is a rational multiple of $\pi^2$.
	\end{corollary}
		
		%
	As we mentioned before,
	the volume of $(M,\rho)$ depends only on the induced representation
	$\bar\rho\colon \pi_1(M)\to\mathrm{PSL}(2,\Real)$. 
	Indeed, our volume formula also carries this feature in an evident way:
	
	\begin{proposition}\label{volume_general_GM_PSL}
		In Theorem \ref{volume_general_GM},
		the vector $\opEuler_M\xi_\rho$ in $\Real^V$ depends only on the induced representation
		$\bar\rho\colon \pi_1(M)\to\mathrm{PSL}(2,\Real)$.
	\end{proposition}
	
	\begin{proof}
		Any representation $\pi_1(M)\to \Sft\times_\Integral\Real$
		that induces the same representation $\bar\rho\colon \pi_1(M)\to\mathrm{PSL}(2,\Real)$ as that of $\rho$
		is a twisted representation $\rho[\alpha]$ for some cohomology class $\alpha\in H^1(M;\Real)$,
		see (\ref{twisting_representation}) and Theorem \ref{lifting_representations}.
		We observe $\xi_{\rho[\alpha]}(v)=\xi_\rho(v)+\dot\alpha(v)$,
		where $\dot\alpha(v)=\alpha([f_v])$ is the value of $\alpha$ at $[f_v]\in H_1(M;\Real)$.
		We may write $\dot{\alpha}=\sum_{v\in V}\alpha([f_v])v^*$.
		With respect to any auxiliary Waldhausen bases $\{([f_{v,w}],[s_{v,w}])\}$,
		we use (\ref{charge_Waldhausen}) to	compute, for any vertex $v\in V$:
		\begin{eqnarray*}
			\left(\opEuler_M\dot{\alpha}\right)(v)
			&=& k_v\cdot\alpha\left([f_v]\right)-\sum_{\{v,w\}\in E} \frac{1}{b_{v,w}}\cdot\alpha\left([f_w]\right)\\
			&=& \sum_{\{v,w\}\in E} 
			\left(\frac{a_{v,w}}{b_{v,w}}\cdot\alpha\left([f_v]\right)-\frac{1}{b_{v,w}}\cdot\alpha\left(a_{v,w}[f_v]+b_{v,w}[s_{v,w}]\right)\right)\\
			&=& \alpha\left(\sum_{\{v,w\}\in E} [s_{v,w}]\right)\\
			&=&0.
		\end{eqnarray*}
		The last step uses the property 
		$\sum_{\{v,w\}\in E}[s_{v,w}]=0$ 
		in $H_1(J_v;\Rational)$ of Waldhausen bases.
		Therefore $(\opEuler_M\xi_{\rho[\alpha]})(v)$ is constant as $\alpha$ ranges over $H^1(M;\Real)$.
		In other words, $\opEuler_M\xi_\rho$ depends only on $\bar\rho$.
	\end{proof}
	
	In the rest of this section, we take a closer look at the local structure of Seifert representations
	for graph manifolds, and give an outline of the proof of Theorem \ref{volume_general_GM}.
	
	\subsection{Local types of Seifert representations}\label{Subsec-local_types}
	We analyze types of Seifert representations at vertices of the JSJ graph.
	Let $M$ be a formatted graph manifold
	and	$\rho\colon\pi_1(M)\to \Sft\times_\Integral\Real$ be a representation.
	Adopt Notation \ref{notation_xi_tau}.
	
	For any vertex $v\in (V,E)$,
	if the image of any conjugate of $f_v$ 
	under the induced representation $\bar\rho\colon \pi_1(M)\to \mathrm{PSL}(2,\Real)$ is nontrivial,
	the image of any conjugate of $\pi_1(J_v)$ under $\bar\rho$ must be abelian,
	centralizing some conjugate of $\bar\rho(f_v)$.
	In this case, $\bar\rho(\pi_1(J_v))$ can be conjugated into a $1$--dimensional closed subgroup
	of $\mathrm{PSL}(2,\Real)$, 
	which may be elliptic, hyperbolic, or parabolic,
	depending on the type of $\bar\rho(f_v)$.		
	
	We say that $v\in V$ is a vertex of \emph{central type} with respect to $\rho$,
	if $\bar\rho(f_v)$ is trivial in $\mathrm{PSL}(2,\Real)$.
	Similarly, we call $v$ a vertex of \emph{elliptic}, \emph{hyperbolic}, or \emph{parabolic type},
	if $\bar\rho(f_v)$ is nontrivial of the corresponding type,
	which depends only on the conjugacy class of $\bar\rho(f_v)$.
	We say that $\{v,w\}\in E$ is an edge of \emph{central type} with respect to $\rho$,
	if both $v$ and $w$ are central.

	For any simplicial subgraph $C=(V_C,E_C)$ of the simplicial JSJ graph $(V,E)$,
	we denote by $M_C$ the union of the JSJ pieces and the JSJ tori that occur in $C$,
	and call $M_C$ the \emph{JSJ block} that corresponds to $C$.
	Note that $M_C$ is a $3$--manifold possibly with toral ends.

	\begin{notation}\label{notation_delta}
	For any formatted graph manifold $M$ with a simplicial JSJ graph $(V,E)$,
	and for any representation $\rho\colon \pi_1(M)\to \Sft\times_\Integral\Real$,
	adopting Notation \ref{notation_xi_tau},
	we introduce the following notations:
	\begin{enumerate}
	\item
	Suppose that $C=(V_C,E_C)$ is a maximal connected subgraph of $(V,E)$ over which $\rho$ is abelian,
	(namely, the image of $\pi_1(M_C)$ under $\rho$ is abelian).
	For any vertex $v\in V$,	denote by
	$$\delta_\rho(v;C)\in\Natural\cup\{0\}$$
	the number as follows:
	If $v\in V\setminus V_C$, 
	$\delta_\rho(v,C)$ is the number of vertices $w\in V_C$ incident to $v$
	such that $\tau_\rho(v,w)$ infinite;
	otherwise, $\delta_\rho(v,C)$ is $0$.
	\item
	For any vertex $v$, denote by 
	$$\delta_\rho(v)\in\Natural\cup\{0\}$$
	the sum of all $\delta_\rho(v,C)$,
	where $C$ ranges over all the maximal connected subgraphs of $(V,E)$
	over which $\rho$ is abelian.
	\end{enumerate}
	\end{notation}
	
	\subsection{Outline of the proof}\label{Subsec-outline_volume_general_GM}
	We prove Theorem \ref{volume_general_GM} first 
	in the \emph{virtually central} case,
	where $\rho$ is virtually central restricted to any vertex group $\pi_1(J_v)$.
	This is equivalent to say that $\tau_\rho(v,w)$ are all finite.
	If $\tau_\rho(v,w)$ are all $1$,
	$\rho$ is actually central restricted to any vertex group $\pi_1(J_v)$.
	The volume in the central case can be suitably decomposed into contribution 
	from each of the vertices, thanks to the additivity principle (see Theorem \ref{additivity_principle}).
	Moreover, the individual contribution can be computed using Theorem \ref{volume_central_bundle}.
	Then the volume formula for the virtually central case 
	can be derived by passing to a suitable finite cover of $M$,
	and by approximating using virtually central representations with $\xi$--vectors in $\Rational^V$.
	The virtually central case is done in Section \ref{Sec-vc_case}.
	
	The proof of Theorem \ref{volume_general_GM} would be complete 
	if we were able to deform any representation to a virtually central one,
	through a continuous path in $\mathcal{R}(\pi_1(M),\Sft\times_\Integral\Real)$
	along which the $\xi$--vector stays constant.
	Unfortunately, such a deformation appears unlikely to exist in general.
	However, we prove that there exist a formatted graph manifold $M'$ 
	which maps nicely onto $M$ of degree $1$, such that the pull-back representation
	$\rho'\colon \pi_1(M')\to\Sft\times_\Integral\Real$ deforms continuously
	to a virtually central representation.
	Since the volume of $(M',\rho')$ is equal to the volume of $(M,\rho)$,
	the construction allows us to 
	reduce to the proof of the volume formula 
	to the virtually central case on the level of $(M',\rho')$.
	We are also able to obtain a slightly weaker estimate for $|\opEuler_M\xi_\rho(v)|$
	using $(M',\rho')$.
	To complete the proof of Theorem \ref{volume_general_GM},
	we apply the covering trick again,
	and promote the weak estimate to the desired one.
	
	The promotion appeals to certain efficient control 
	on the complexity of the above $M'$.
	To be more precise, the degree-one map $M'\to M$ 
	that we construct	induces an isomorphism between the JSJ graphs,
	so $\opEuler_{M'}=\opEuler_{M}$ holds as we identify their JSJ graphs.
	Meanwhile, $\chi_{M'}(v)\leq\chi_{M}(v)$ holds for any vertex $v$.
	In effect, therefore,	passing from $M$ to $M'$ increases
	the genera of the base $2$--orbifolds of the JSJ pieces.
	We need sufficient increment of the genera
	to make enough room for the desired deformation,
	but we also want the increment at any vertex
	to be relatively small compared to the original genus.
	It turns out that $\chi_{M'}(v)=\chi_{M}(v)-2\cdot\delta_\rho(v)$ 
	suffices for our purpose.
	As one may infer from Notation \ref{notation_delta},
	our deformation is obtained
	by assembling deformations restricted to the vertex subgroups 
	(or JSJ-block subgroups).
	In Section \ref{Sec-pdc}, 
	we develop techniques	to 	
	factorize certain (continuous or sequential) families of elements in $\Sft\times_\Integral\Real$ 
	nicely as families of commutators.
	In Section \ref{Sec-reduction_vc}, 
	we employ those techniques to construct $(M',\rho')$ 
	for the desired reduction.
	In Section \ref{Sec-proof_volume_general_GM},
	we put all the partial results together,
	and conclude the proof of Theorem \ref{volume_general_GM}.

\section{The virtually central case}\label{Sec-vc_case}
	In this section, we prove a special case of Theorem \ref{volume_general_GM},
	as the following Theorem \ref{volume_virtually_central_GM}.
	Note that the additional assumption implies 
	that $\rho$ is virtually central restricted to the vertex groups,
	or equivalently, that $\rho$ pulls back to 
	a representation that is central restricted to the vertex groups,
	for some finite cover	of $M$.

	\begin{theorem}\label{volume_virtually_central_GM}
		The statements of Theorem \ref{volume_general_GM} hold true 
		if $\tau_\rho(v,w)$ is finite for all $\{v,w\}\in E$.
	\end{theorem}
	
	The proof of Theorem \ref{volume_virtually_central_GM} occupies the rest of this section.

	\subsection{The fundamental computation}

	\begin{lemma}\label{volume_rational_central_GM}
		The statements of Theorem \ref{volume_general_GM} hold	true under the following additional conditions:
		\begin{itemize}
		\item $\xi_\rho(v)\in\Rational$, for all vertices $v\in V$, and
		\item $\tau_\rho(v,w)=1$, for all edges $\{v,w\}\in E$.
		\end{itemize}
	\end{lemma}

	The rest of this subsection is devoted to the proof of Lemma \ref{volume_rational_central_GM}.
	We employ the following additivity principle for volume computation:
	
	\begin{theorem}[{\cite[Theorem 3.5]{DLW-cs_sep_vol}}]\label{additivity_principle}
	Let $M$ be an oriented closed irreducible $3$--manifold with the JSJ tori $T_1,\cdots,T_r$ and the JSJ pieces
	$J_1,\cdots,J_k$. Let $h_1,\cdots,h_r$ be slopes on $T_1,\cdots,T_r$, respectively.
	Denote by $N_i$ the oriented closed $3$--manifold obtained from
	the Dehn filling of $J_i$ along all the slopes $h_j$ which occur on its boundary,
	for all $i\in\{1,\cdots,k\}$.
	
	Suppose that $\rho\colon \pi_1(M)\to \Sft\times_\Integral\Real$
	is a representation	which is trivial on all the slopes $h_j$.
	Denote by $\phi_i\colon\pi_1(N_i)\to\Sft\times_\Integral\Real$
	the representation induced by $\rho_i$,
	for all $i\in\{1,\cdots,k\}$.
	Then the following formula holds:
	$$\mathrm{vol}_{\Sft\times_\Integral\Real}(M,\rho)=
	\mathrm{vol}_{\Sft\times_\Integral\Real}(N_1,\phi_1)+\cdots+\mathrm{vol}_{\Sft\times_\Integral\Real}(N_k,\phi_k).$$	
	\end{theorem}
	
	Let us first unwrap the additional conditions of Lemma \ref{volume_rational_central_GM}.
	The additional condition $\tau_{\rho}(v,w)=1$
	means that
	the representation $\rho$	sends any JSJ torus subgroup $\pi_1(T_{v,w})$ to the central subgroup $\Real$.
	In particular, the image $\rho(f_v)$ of the regular fiber of any JSJ piece $J_v$
	can be identified as a real value,
	which equals the essential winding number of $\rho(f_v)$, by definition.
	Therefore, the additional condition $\xi_\rho(f_v)\in\Rational$ 
	means $\rho(f_v)\in\Rational$.
	It follows that $\rho(\pi_1(T_{v,w}))$ is also contained in the central subgroup $\Rational$
	of $\Sft\times_\Integral\Real$.	
	In particular, 
	there must exist a slope $h_{v,w}$ on $T_{v,w}$ which lies in the kernel of $\rho$.
	
	Choose a slope $h_{v,w}$ for every edge $\{v,w\}\in E$ as above,
	such that $h_{w,v}$ is identified with $h_{v,w}$ under the identification
	$T_{v,w}\cong T_{w,v}$.
	For any JSJ piece $J_v$ of $M$, denote by $N_v$ the Dehn filling of $J_v$ 
	along the all the slopes $h_{v,w}$ on $\partial J_v$.
	Denote by $\phi_v\colon \pi_1(N_v)\to\mathrm{PSL}(2,\Real)$ the representation
	induced by $\rho$.
		
	Theorem \ref{additivity_principle} allows us to compute the volume of $(M,\rho)$ from the volumes of $(N_v,\phi_v)$:
	\begin{equation}\label{volume_breakdown}
		\mathrm{vol}_{\Sft\times_\Integral\Real}(M,\rho)=\sum_{v\in V}\mathrm{vol}_{\Sft\times_\Integral\Real}(N_v,\phi_v).
	\end{equation}
				
	Note that 
	if some $h_{v,w}$ is parallel to the fiber $f_v$ (in the unoriented sense), 
	then $\xi_\rho(v)$ equals $0$.
	In this case,
	$\mathrm{vol}_{\Sft\times_\Integral\Real}(N_v,\phi_v)$
	also equals $0$,
	since $\phi_v$ factors through the fundamental group of the base $2$--orbifold of $J_v$,
	(whose group homology is torsion on the dimension $3$).
	
	If $J_v$ is any JSJ piece with $\xi_\rho(v)\neq0$,
	there are no $h_{v,w}$ parallel to the fiber $f_v$.
	In this case, the Seifert fibration $J_v\to\mathcal{O}_v$ naturally extends to 
	a Seifert fibration $N_v\to\mathcal{R}_v$, 
	where $\mathcal{R}_v$ is the $2$--orbifold obtained from $\mathcal{O}_v$ 
	by filling the punctures with cone points.
	The orders of the cone points are the geometric intersection numbers
	between $f_v$ and $h_{v,w}$ on the tori $T_{v,w}$ adjacent to $J_v$.
	
	\begin{lemma}\label{e_M_v}
		If $\xi_\rho(v)\neq0$, then
		$$(\opEuler_M\xi_\rho)(v)=e\left(N_v\to\mathcal{R}_v\right)\times\xi_\rho(v).$$
	\end{lemma}
	
	\begin{proof}
		Take auxiliary Waldhausen bases $([f_v],[s_{v,w}])$ for all $T_{v,w}$.
		Take an orientation for each $h_{v,w}$.
		We write
		\begin{eqnarray*}
		{[f_w]} &=& a_{v,w}[f_v]+b_{v,w}[s_{v,w}]\\
		{[h_{v,w}]} &=& p_{v,w}[f_v]+q_{v,w}[s_{v,w}].
		\end{eqnarray*}
		Note that $b_{v,w}$ and $q_{v,w}$ are nonzero.
		We compute
			$$\frac{\xi_\rho(w)}{\xi_\rho(v)}=\frac{I_{v,w}([f_w],[h_{v,w}])}{I_{v,w}([f_v],[h_{v,w}])}
			=\frac{a_{v,w}q_{v,w}-b_{v,w}p_{v,w}}{q_{v,w}},$$
		where
		$I_{v,w}\colon H_1(T_{v,w};\Rational)\times H_1(T_{v,w};\Rational)\to \Rational$
		stands for the algebraic intersection pairing on the oriented JSJ torus $T_{v,w}$.
		This can be rearranged into
			$$\xi_\rho(v)\times\left(\frac{a_{v,w}}{b_{v,w}}-\frac{p_{v,w}}{q_{v,w}}\right)=\xi_\rho(w)\times\frac{1}{b_{v,w}}.$$
		Sum up over all the vertices $w$ incident to $v$. Then we obtain
			$$\xi_\rho(v)\times(k_v-e(N_v\to\mathcal{R}_v))=\sum_{\{v,w\}\in E}\xi_\rho(w)\times\frac{1}{b_{v,w}},$$
		or equivalently,
			$$(\opEuler_M\xi_\rho)(v)=e\left(N_v\to\mathcal{R}_v\right)\times\xi_\rho(v).$$
		as asserted.
	\end{proof}
	
	\begin{lemma}\label{volume_central_filled}
		If $\xi_\rho(v)\neq0$, then
		$$\mathrm{vol}_{\Sft\times_\Integral\Real}\left(N_v,\phi_v\right)=
		4\pi^2\cdot e\left(N_v\to\mathcal{R}_v\right)\times\xi_\rho(v)^2,$$
		and moreover,
		$$\left|e\left(N_v\to\mathcal{R}_v\right)\times\xi_\rho(v)\right|\leq
		\max\left\{0,-\chi_M(v)-\mathrm{valence}_{(V,E)}(v)\right\}.$$
	\end{lemma}

	\begin{proof}
		For simplicity, we rewrite $N\to\mathcal{R}$ 
		for the Seifert fibration $N_v\to\mathcal{R}_v$,
		and $f$ for the ordinary fiber $f_v$,
		and $\phi$ for the representation $\phi_v\colon \pi_1(N_v)\to \Sft\times_\Integral\Real$.
		We rewrite $e_N$ for the Euler number $e\left(N_v\to\mathcal{R}_v\right)$.
		We remember $\xi_\rho(v)=\phi(f)$.
				
		The assertions are trivial if $e_N$ equals $0$. 
		If $e_N$ is not $0$, and if $\mathcal{R}$ has orbifold Euler number $\geq0$,
		$\pi_1(\mathcal{R})$ is either virtually nonabelian nilpotent, or finite.
		Note that virtually nilpotent subgroups of $\Sft\times_\Integral\Real$ are all abelian.
		In this case, we must have $\phi(f)=0$,
		so the assertions are again trivial.
		It remains to prove for $e_N\neq0$ and $\chi(\mathcal{R})<0$.

		Take covering projection $R'\to\mathcal{R}$ of $\mathcal{R}$ by a closed surface $R'$,
		(which always exists provided $\chi(\mathcal{R})<0$).
		Denote by $N'\to R'$ the oriented circle bundle obtained from $N\to\mathcal{R}$ by pull-back.
		For the pull-back representation $\phi'\colon \pi_1(N')\to\Sft\times_\Integral\Real$,
		we observe that $\phi'$ is of central type,
		with $\phi'_{\mathtt{fib}}=\phi(f)$ and $e_{N'}=[R':\mathcal{R}]\times e_N$.
		As $R'$ has negative Euler characteristic and $N'\to R'$ has nonzero Euler number,
		we apply Theorem \ref{volume_central_bundle} to obtain the asserted volume formula:
		$$\mathrm{vol}_{\Sft\times_\Integral\Real}(N,\phi)=
		\frac{\mathrm{vol}_{\Sft\times_\Integral\Real}\left(N',\phi'\right)}{[R':\mathcal{R}]}=
		\frac{4\pi^2\phi'_{\mathtt{fib}}\times \left(e_{N'}\times\phi'_{\mathtt{fib}}\right)}{[R':\mathcal{R}]}=
		4\pi^2e_N\times\xi_\rho(v)^2.
		$$
				
		Let $\mathcal{P}$ be the $2$--orbifold obtained from $\mathcal{O}_v$
		by filling the punctures with ordinary points.
		There is a natural orbifold map $\mathcal{R}\to\mathcal{P}$,
		which erases the filling cone points from $\mathcal{R}$.
		The orbifold Euler characteristic of $\mathcal{P}$ can be computed as
		$$\chi(\mathcal{P})
		= \chi_M(v)+\mathrm{valence}_{(V,E)}(v).$$
				
		Suppose $\chi(\mathcal{P})>0$.
		In this case, the $2$--orbifold $\mathcal{P}$ is either spherical or bad.
		The induced representation $\phi_{\mathtt{base}}\colon \pi_1(\mathcal{R})\to \mathrm{PSL}(2,\Real)$
		is necessarily trivial.
		The above volume formula implies $\xi_\rho(v)=0$,
		so the asserted inequality holds for $\chi(\mathcal{P})>0$, and indeed,
		$$|e_N\times\xi_\rho(v)|=0.$$
		
		Suppose $\chi(\mathcal{P})\leq0$.
		In this case, the $2$--orbifold $\mathcal{P}$ admits a universal cover homeomorphic to an open disk.
		Then there exists a finite cover $P''\to\mathcal{P}$ of $\mathcal{P}$ by a surface $P''$.
		Pull back $P''\to \mathcal{P}$ along the composite map $R'\to\mathcal{R}\to\mathcal{P}$.
		We obtain a finite covering projection $R''\to R'$,
		whose degree equals $[P'':\mathcal{P}]$.
		The construction comes with 
		a branched covering map $R''\to P''$,
		whose mapping degree equals $[R':\mathcal{R}]$.
		Denote by $N''\to R''$ the pull-back of $N'\to R'$ along $R''\to R'$.
		We obtain representations 
		$\phi'\colon \pi_1(N')\to \Sft\times_\Integral\Real$
		and $\phi''\colon \pi_1(N'')\to \Sft\times_\Integral\Real$,
		which are naturally induced from $\phi$ and of central type.
		Theorem \ref{volume_central_bundle} implies
		$$e\left(\phi''_{\mathtt{base}}\right)=
		[P'':\mathcal{P}]\times e\left(\phi'_{\mathtt{base}}\right)=
	  [P'':\mathcal{P}]\times	e_{N'}\times\phi'_{\mathtt{fib}}=
		[P'':\mathcal{P}]\times[R':\mathcal{R}]\times e_N\times\xi_\rho(v).$$
		On the other hand, 
		$\phi''_{\mathtt{base}}\colon \pi_1(R'')\to \mathrm{PSL}(2,\Real)$ 
		factors through a representation of $\pi_1(R'')$.
		The possible values for the Euler number 
		of any representation $\pi_1(P'')\to\mathrm{PSL}(2,\Real)$
		are precisely $0,\pm1,\pm2,\cdots,\pm|\chi(P'')|$.
		By construction, we obtain
		$$\left|e\left(\phi''_{\mathtt{base}}\right)\right|\leq[R':\mathcal{R}]\times |\chi(P'')|
		= \left[R':\mathcal{R}\right]\times[P'':\mathcal{P}]\times\left|\chi(\mathcal{P})\right|.$$
		Therefore, we obtain
		$$\left|e_N\times\xi_\rho(v)\right|\leq \left|\chi(\mathcal{P})\right|
		=-\chi_M(v)-\mathrm{valence}_{(V,E)}(v).$$
		so the asserted inequality also holds for $\chi(\mathcal{P})\leq0$.
	\end{proof}

		From Lemmas \ref{e_M_v} and \ref{volume_central_filled}, we obtain
		\begin{equation}\label{volume_N_phi}
		\mathrm{vol}_{\Sft\times_\Integral\Real}\left(N_v,\phi_v\right)=
		4\pi^2\cdot \xi_\rho(v)\times (\opEuler_M\xi_\rho)(v),
		\end{equation}
		and 
		\begin{equation}\label{bound_N_phi}
		\left|e\left(N_v\to\mathcal{R}_v\right)\times\xi_\rho(v)\right|\leq
		\max\left\{0,-\chi_M(v)-\mathrm{valence}_{(V,E)}(v)\right\}.
		\end{equation}
		Both (\ref{volume_N_phi}) and (\ref{bound_N_phi}) 
		hold for $\xi_\rho(v)\neq0$ as well as for $\xi_\rho(v)=0$.		
		The inequality of Theorem \ref{volume_general_GM} is exactly (\ref{bound_N_phi})
		under the additional condition $\tau_\rho(v,w)=1$ for all $\{v,w\}\in E$.
		The rationality of the inner product $(\xi_\rho,\opEuler_M\xi_\rho)$
		follows from Lemma \ref{e_M_v} and the additional condition $\xi_\rho(v)\in\Rational$
		for all $v\in\Rational$.
		The volume formula of Theorem \ref{volume_general_GM}
		follows from (\ref{volume_breakdown}) and (\ref{volume_N_phi}),
		under the additional conditions of Lemma \ref{volume_rational_central_GM}.
			
		This completes the proof of Lemma \ref{volume_rational_central_GM}.

	\subsection{The covering trick}\label{Subsec-covering_trick}
	
	\begin{lemma}\label{volume_rational_virtually_central_GM}
		The statements of Theorem \ref{volume_general_GM} hold	true under the following additional conditions:
		\begin{itemize}
		\item $\xi_\rho(v)\in\Rational$, for all vertices $v\in V$, and
		\item $\tau_\rho(v,w)<\infty$, for all edges $\{v,w\}\in E$.
		\end{itemize}
	\end{lemma}

	The rest of this subsection is devoted to the proof of Lemma \ref{volume_rational_virtually_central_GM}.
	We reduce the problem to Lemma \ref{volume_rational_central_GM}
	by passing to a suitable finite cover of $M$.
	This invokes the following well-known fact,
	often referred to as Selberg's lemma:
	
	\begin{theorem}[{\cite[Chapter 7, \S 7.6, Corollary 4]{Ratcliffe}}]\label{Selberg_lemma}
		Every finitely generated subgroup of $\mathrm{GL}(n,\Complex)$ contains a torsion-free normal subgroup of finite index.	
	\end{theorem}
	
	Under the additional condition $\tau_\rho(v,w)<\infty$,
	the image of $\pi_1(T_{v,w})$ under $\bar\rho\colon \pi_1(M)\to \mathrm{PSL}(2,\Real)$
	is finite (and cyclic), for all edges $\{v,w\}\in E$.
	Note that $\mathrm{PSL}(2,\Real)\cong\mathrm{SO}^+(2,1)$ is linear,
	and that $\bar\rho(\pi_1(M))$ is a finitely generated group of $\mathrm{PSL}(2,\Real)$.
	Then there exists some torsion-free finite-index normal subgroup of $\bar\rho(\pi_1(M))$,
	by Selberg's lemma (Theorem \ref{Selberg_lemma}).
	The preimage of that subgroup corresponds to a regular finite cover $M^*$ of $M$,
	where the JSJ tori subgroups all have trivial image under the pull-back representation 
	$\bar\rho^*\colon \pi_1(M')\to\mathrm{PSL}(2,\Real)$.
	For every covering pair of JSJ pieces $J^*\to J$,
	$J^*$ is furnished with an oriented Seifert fibration over an oriented bases,
	as it covers such a Seifert fibration of $J$.
	Take a characteristic finite cover of the JSJ graph of $M^*$ 
	which is a simplicial graph.
	Then it induces a characteristic finite cover $M'$ over $M^*$,
	which is a formatted graph manifold.
	
	From the above construction, we obtain a regular finite cover $M'$ over $M$,
	which is a formatted graph manifold.
	Denote by $\rho'\colon\pi_1(M')\to \Sft\times_\Integral\Real$
	the pull-back representation of $\rho$,
	and by $(V',E')$ the JSJ graph of $M'$.
	The construction also guarantees 
	\begin{equation}\label{vrvc_T_deg}
	\tau_{\rho'}(v',w')=	\frac{1}{[T'_{v',w'}:T_{v,w}]}\times\tau_\rho(v,w)=1
	\end{equation}
	for any covering pair of JSJ tori $T'_{v',w'}\to T_{v,w}$.
	We observe 
	\begin{equation}\label{vrvc_f_deg}
	\xi_{\rho'}(v')=[f'_{v'}:f_v]\times \xi_{\rho}(v)
	\end{equation}
	for all covering pair of JSJ pieces $J'_{v'}\to J_v$.
	In particular,
	Lemma \ref{volume_rational_central_GM} applies to $(M',\rho')$.
	
	For any JSJ piece $J_v$ of $M$,
	the number of	JSJ pieces $J'_{v'}$ of $M'$ that cover $J_v$
	is precisely $[M':M]/[J'_{v'}:J_v]$.
	For any boundary component $T_{v,w}$ of $J_v$,
	the number of boundary components $T'_{v',w'}$ of $J'_{v'}$	that cover $T_{v,w}$
	is precisely $[J'_{v'}:J_v]/[T_{v',w'}:T_{v,w}]$.
	The number $b_{v',w'}$ depends only on $\{v,w\}$ because of regular covering.
	By the formulas of Proposition \ref{virtual_xkb} and (\ref{vrvc_f_deg}),
	we compute:
	\begin{eqnarray*}
	\left(\opEuler_{M'}\xi_{\rho'}\right)(v')
	&=&k_{v'}\xi_{\rho'}(v')-\sum_{\{v',w'\}\in E'}\frac{\xi_{\rho'}(w')}{b_{v',w'}}\\
	&=& \frac{[J'_{v'}:J_{v}]}{[f'_{v'}:f_v]}\times k_v\xi_{\rho}(v)-
	\sum_{\{v,w\}\in E} \frac{[J'_{v'}:J_v]}{[T'_{v',w'}:T_{v,w}]}\times 
	\frac{[T'_{v',w'}:T_{v,w}]}{[f'_{v'}:f_v]}\times\frac{\xi_{\rho}(w)}{b_{v,w}}\\
	&=& \frac{[J'_{v'}:J_v]}{[f'_{v'}:f_v]}\times \left(k_{v}\xi_{\rho}(v)-\sum_{\{v,w\}\in E}\frac{\xi_{\rho}(w)}{b_{v,w}}\right)\\
	&=& \frac{[J'_{v'}:J_v]}{[f'_{v'}:f_v]}\times \left(\opEuler_{M}\xi_{\rho}\right)(v).
	\end{eqnarray*}
	
	By the volume formula in Theorem \ref{volume_general_GM} (Lemma \ref{volume_rational_central_GM}) for $(M',\rho')$
	and (\ref{vrvc_f_deg}), we compute:
	\begin{eqnarray*}
	\mathrm{vol}_{\Sft\times_\Integral\Real}(M,\rho)&=&
	\frac{1}{[M':M]}\times\mathrm{vol}_{\Sft\times_\Integral\Real}(M',\rho')\\
	&=&\frac{1}{[M':M]}\times4\pi^2\cdot\left(\xi'_{\rho'},\opEuler_{M'}\xi_{\rho'}\right)\\
	&=&\frac{4\pi^2}{[M':M]}\times\sum_{v'\in V'} \xi_{\rho'}(v')\times\left(\opEuler_{M'}\xi_{\rho'}\right)(v')\\
	&=&\frac{4\pi^2}{[M':M]}\times\sum_{v\in V} \frac{[M':M]}{[J'_{v'}:J_v]}
	\times[J'_{v'}:J_v]\times\xi_{\rho}(v)\times\left(\opEuler_{M}\xi_{\rho}\right)(v)\\
	&=&4\pi^2\cdot\left(\xi_{\rho},\opEuler_{M}\xi_{\rho}\right).
	\end{eqnarray*}
	
	By Proposition \ref{virtual_xkb} and (\ref{vrvc_T_deg}),
	we estimate, for any covering pair of JSJ pieces $J'_{v'}\to J_v$:
	\begin{eqnarray*}
	-\chi_{v'}-\sum_{\{v',w'\}\in E'}\frac{1}{\tau_{\rho'}(v',w')}
	&=&
	-\,\frac{[J'_{v'}:J_v]}{[f'_{v'}:f_v]}\times \chi_{v}
	-\sum_{\{v,w\}\in E}
	\frac{[J'_{v'}:J_v]}{[T'_{v',w'}:T_{v,w}]}\times\frac{[T'_{v',w'}:T_{v,w}]}{\tau_{\rho}(v,w)}\\
	&\leq&
	\frac{[J'_{v'}:J_v]}{[f'_{v'}:f_v]}\times \left(-\chi_{v}-\sum_{\{v,w\}\in E}\frac{1}{\tau_{\rho}(v,w)}\right).
	\end{eqnarray*}
	
	By the inequality in Theorem \ref{volume_general_GM} (Lemma \ref{volume_rational_central_GM}) for $(M',\rho')$,
	we obtain
	\begin{eqnarray*}
	|(\opEuler_M\xi_\rho)(v)|&=&\frac{[f'_{v'}:f_v]}{[J'_{v'}:J_v]}\times |(\opEuler_{M'}\xi_{\rho'})(v')|\\
	&\leq&\frac{[f'_{v'}:f_v]}{[J'_{v'}:J_v]}\times
	\max\left\{0,-\chi_{v'}-\sum_{\{v',w'\}\in E'}\frac{1}{\tau_{\rho'}(v',w')}\right\} \\
	&\leq&\max\left\{0,-\chi_v-\sum_{\{v,w\}\in E}\frac{1}{\tau_\rho(v,w)}\right\}.
	\end{eqnarray*}
	
	The assertion about rationality is again obvious provided $\xi_\rho\in\Rational^V$.
	
	This completes the proof of Lemma \ref{volume_rational_virtually_central_GM}.
	
	\subsection{The rational approximation}
		
	\begin{lemma}\label{lemma_rational_approx}
	For any finitely generated group $\pi$ 
	and any representation $\rho\colon \pi\to \Sft\times_\Integral\Real$,
	there is a dense subset of cohomology classes $\alpha\in H^1(\pi;\Real)$,
	such that the image of the twisted representations $\rho[\alpha]$ of $\pi$
	is contained in $\Sft\times_\Integral\Rational$.
	\end{lemma}
	
	To recall the notation $\rho[\alpha]$,
	see (\ref{twisting_representation}) and Section \ref{Subsec-central_extension}.
	
	\begin{proof}
	Denote by $\bar\rho\colon\pi\to \mathrm{PSL}(2,\Real)$ the induced representation.
	Let $u_1,\cdots,u_r$ be a generating set of $\pi$.
	For each generator $u_i$ of $\pi$, take a lift $g_i\in\Sft$ of $\bar\rho(u_i)$.
	Then for any relator $R(u_1,\cdots,u_r)$ of $\pi$,
	the element $R(g_1,\cdots,g_r)\in\Sft$ is central,
	so it can be identified as an integer $c_R\in\Integral$.
	Suppose $s_1,\cdots,s_r\in\Real$.
	Then the assignment $\rho'(u_i)=g_i[s_i]$ for all $i\in\{1,\cdots,r\}$
	determines a representation $\rho'\colon\pi\to \Sft\times_\Integral\Real$ 
	if and only if the equation $R(g_1[s_1],\cdots,g_r[s_r])=\id[0]$ holds in $\Sft\times_\Integral\Real$
	for all relators $R$ of $\pi$.
	This gives rise to a linear system of equations with $r$ unknowns.
	Because of $\rho$, there is a real solution, so there is also a rational solution $\rho'$.
	By Theorem \ref{lifting_representations},
	we see $\rho'=\rho[\alpha']$ for some $\alpha'\in H^1(\pi;\Real)$.
	Then the asserted subset of cohomology classes can be taken as 
	the coset $\alpha'+H^1(\pi;\Rational)$ in $H^1(\pi;\Real)$.
	\end{proof}
	
	By Lemma \ref{lemma_rational_approx}, 
	we obtain a sequence of twisted representations 
	$\rho[\alpha_n]\colon\pi_1(M)\to \Sft\times_\Integral\Rational$ 
	such that $\alpha_n\to 0$ in $H^1(M;\Real)$.
	Note that $\tau_{\rho[\alpha_n]}$ are all equal to $\tau_\rho$,
	so they are all finite by assumption.
	We also see $\xi_{\rho[\alpha_n]}\in\Rational^V$,
	because every $\pi_1(T_{v,w})$ 
	has virtually central and rational image in $\Sft\times_\Integral\Real$
	under any $\rho[\alpha_n]$.
	Apply Lemma \ref{volume_rational_virtually_central_GM} to all $\rho[\alpha_n]$
	and take the limit.
	Then we obtain the volume formula and the generalized Milnor--Wood inequalities
	in Theorem \ref{volume_virtually_central_GM}.
	Note that $\rho[\alpha_n]$ all lie on the same component of $\mathcal{R}(\pi_1(M),\Sft\times_\Integral\Rational)$
	as that of $\rho$, so they have the same volume of $\Sft\times_\Integral\Real$--representations,
	(see Section \ref{Subsec-volume_of_representations}).
	We obtain	the assertion on rationality, 
	since any $\xi_{\rho[\alpha_n]}$ serves as an asserted solution $X\in\Rational^V$.
	
	This completes the proof of Theorem \ref{volume_virtually_central_GM}.

\section{Perturbation, deformation, and conjugation}\label{Sec-pdc}
	In this section, we develop techniques for modifying certain $\Sft\times_\Integral\Real$--representations of 
	finitely generated abelian groups to obtain certain canonical families.
	The modification families are either sequential or continuous, and converge to the original representation.
	In particular, the modified representations all lie on the path-connected component of the original representation,
	in the space of $\Sft\times_\Integral\Real$--representations.
	We also study commutator factorizations in $\Sft\times_\Integral\Real$ of elements occuring in such modification.
	For given formatted graph manifolds, roughly speaking, 
	these techniques allow us to modify $\Sft\times_\Integral\Real$--representations
	restricted to the edge groups, and in certain circumstances,
	to extend the modification over the adjacent vertex groups.
	For both the modification and the commutator factorization,
	we develop parallel versions for elliptic, hyperbolic, and parabolic representations separately.

	\subsection{Noncentral representations of abelian groups}

	\begin{lemma}\label{ab_rep_ell}
		Let $H$ be a finitely generated abelian group.
		If $\eta\colon H\to \Sft\times_\Integral\Real$ is a representation of elliptic type,
		then there exists a sequence of elliptic representations 
		$\eta_n\colon H\to\Sft\times_\Integral\Real$, indexed by $n\in\Natural$,
		such that the following properties are all satisfied, for all $n\in\Natural$:
		\begin{itemize}
		\item
		For all $h\in H$,
		$\eta_n(h)$ lies in the coset $\eta(h)\Sft$.
		\item
		For all $h\in H$,
		$\eta_n(h)$ converges to $\eta(h)$ as $n$ tends to $\infty$.
		\item
		If $\eta(h)$ lies in the center $\Real$,
		then $\eta_n(h)$ equals $\eta(h)$.
		\item 
		The induced representation $\bar\eta_n\colon H\to \mathrm{PSL}(2,\Real)$ 
		has finite cyclic image.
		\end{itemize}
		In fact, $\eta_n$ can be constructed with image in the centralizer of $\eta(H)$
		in $\Sft\times_\Integral\Real$.
	\end{lemma}

	\begin{proof}
		Since $\eta$ is of elliptic type, the centralizer of $\eta(H)$
		is a conjugate of the closed abelian subgroup $\widetilde{\mathrm{SO}}(2)\times_\Integral\Real$.
		We decompose the finitely generated abelian group $H$ 
		as a Cartesian product of cyclic factors.
		Then it suffices to construct $\eta_n$ for each of the factors,
		and this will determine a sequence of representations $\eta_n$ for $H$ 
		with the asserted properties.
		So the problem reduces to the case when $H$ is cyclic.
		
		If $H$ is finite cyclic, 
		or	if $H$ is infinite cyclic but $\bar\eta(H)$ finite cyclic,
		we simply take $\eta_n$ to be $\eta$ for all $n\in\Natural$.
		
		If $H$ is infinite cyclic and $\bar\eta(H)$ also infinite cyclic,		
		we construct as follows.
		Identify $H$ with the additive group of integers $\Integral$.
		Denote by $g[s]\in\Sft\times_\Integral\Real$ the image $\eta(1)$,
		for some $g\in\Sft$ and $s\in\Real$.
		Since $\bar{\eta}(g)\in \mathrm{PSL}(2,\Real)$ is elliptic, 
		there is a sequence of finite-order elements $\bar{g}_n\in\mathrm{PSL}(2,\Real)$
		which commute with $\bar{\eta}(g)$ and which converge to $\bar{g}\in\mathrm{PSL}(2,\Real)$
		as $n$ tends to $\infty$.
		There is also a sequence of elements $g_n\in\Sft$ which lift $\bar{g}_n\in\mathrm{PSL}(2,\Real)$
		and which converge to $g$ as $n$ tends to $\infty$.
		We construct $\eta_n$ by setting $\eta_n(1)=g_n[s]$, for all $n\in\Natural$.
		Note that $0\in \Integral$ is the only element with image in the center $\Real$,
		since $\bar\eta$ is faithful.
		Then the asserted properties are all obviously satisfied.
	\end{proof}

	\begin{lemma}\label{ab_rep_hyp}
		Let $H$ be a finitely generated abelian group.
		If $\eta\colon H\to \Sft\times_\Integral\Real$ is a representation of hyperbolic type,
		then there exists a continuous family of hyperbolic representations 
		$\eta_t\colon H\to\Sft\times_\Integral\Real$, parametrized by $t\in(0,1]$,
		such that the following properties are all satisfied, for all $t\in(0,1]$:
		\begin{itemize}
		\item
		For all $h\in H$,
		$\eta_t(h)$ lies in the coset $\eta(h)\Sft$.
		\item
		For all $h\in H$,
		$\eta_1(h)$ equals $\eta(h)$.
		\item 
		For all $h\in H$,
		$\eta_t(h)$	converges to 
		the central element $\mathrm{wind}(\eta(h))\in\Real$ as $t$ tends to $0$.
		\item
		If $\eta(h)$ lies in the center $\Real$,
		then $\eta_t(h)$ equals $\eta(h)$.
		\end{itemize}
		In fact, $\eta_t$ can be constructed with image in the centralizer of $\eta(H)$
		in $\Sft\times_\Integral\Real$.
	\end{lemma}
	
	\begin{proof}
		It suffices to argue for $H$ infinite cyclic, for similar reasons as in the proof of Lemma \ref{ab_rep_ell}.
		We identify $H$ as $\Integral$, and denote by $\bar\eta\colon\Integral\to \mathrm{PSL}(2,\Real)$ the induced respresentation of $\eta$.
		Possibly after conjugation, we may assume that
		$$\bar\eta(1)=\left\{\pm\left[\begin{array}{cc} e^\lambda & 0 \\
	  0 & e^{-\lambda}\end{array}\right]\right\}$$
		for some $\lambda\in(1,+\infty)$.
		For any $t\in(0,1]$ and $h\in\Integral$, the expression
		$$\bar\eta_t(h)=\left\{\pm\left[\begin{array}{cc} e^{\lambda ht} & 0 \\
	  0 & e^{-\lambda ht}\end{array}\right]\right\}$$
		defines a continuous family of representations 
		$\bar\eta_t\colon \Integral\to \mathrm{PSL}(2,\Real)$
		parametrized by $t$.
		Note that $\bar{\eta}_t$ are all of hyperbolic type
		with image in the centralizer of $\bar{\eta}(1)$.
		Then for any $h\in\Integral$,
		the path $(0,1]\to \mathrm{PSL}(2,\Real)\colon t\mapsto \eta_t(h)$
		can be lifted uniquely to be a path $(0,1]\to \Sft\times_\Integral\Real \colon t\mapsto \eta_t(h)$,
		such that $\eta_1(h)$ equals $\eta(h)$ and $\eta_t(h)$ lies in the coset $\eta(h)\Sft$.
		We also observe that 
		$\eta_t(h)$ all lie in the centralizer of $\eta(1)$,
		which guarantees that
		$\eta_t\colon \Integral\to \Sft\times_\Integral\Real$ are homomorphisms.
		For any nontrivial $h\in\Integral$,
		$\winding(\eta_t(h))$ remains constant as $t$ ranges in $(0,1]$,
		since 
		$\bar{\eta}_t(h)$ remains hyperbolic.
		Because $\bar{\eta}_t(h)$ tends to the identity element of $\mathrm{PSL}(2,\Real)$
		as $t$ tends to $0+$,
		$\eta_t(h)$ converges to the central element 
		$\winding(\eta(h))\in\Real$ as $t$ tends to $0$.
		Therefore, the representations $\eta_t$ satisfy all the asserted properties,
		as desired.
	\end{proof}

	\begin{lemma}\label{ab_rep_par}
		Let $H$ be a finitely generated abelian group.
		If $\eta\colon H\to \Sft\times_\Integral\Real$ is a representation of parabolic type,
		then there exists a continuous family of elements $g_t\in\Sft$,
		parametrized by $t\in(0,1]$,
		such that the following properties are all satisfied:
		\begin{itemize}
		\item
		At $t=1$, $g_t$ equals the identity.
		\item
		For all $h\in H$, $g_t\eta(h)g_t^{-1}$ converges to 
		the central element $\mathrm{wind}(\eta(h))\in\Real$ as $t$ tends to $0$.
		\end{itemize}
		In fact, $g_t$ can be constructed to normalize the centralizer of $\eta(H)$.
	\end{lemma}
	
	\begin{proof}
		It suffices to argue for $H$ infinite cyclic, for similar reasons as in the proof of Lemma \ref{ab_rep_ell}.
		We identify $H$ as $\Integral$, and denote by $\bar\eta\colon\Integral\to \mathrm{PSL}(2,\Real)$ the induced respresentation of $\eta$.
		Possibly after conjugation, we may assume that
		$$\bar\eta(1)=\left\{\pm\left[\begin{array}{cc} 
		1 & 1 \\
	  0 & 1\end{array}\right]\right\}\mbox{ or }
		\left\{\pm\left[\begin{array}{cc} 
		1 & -1 \\
	  0 & 1\end{array}\right]\right\}.$$
		The expression
		$$\bar{g}_t=
		\left\{\pm\left[\begin{array}{cc} t & 0 \\
	  0 & t^{-1}\end{array}\right]\right\}$$
		defines a path in $\mathrm{PSL}(2,\Real)$
		parametrized by $t\in(0,1]$.
		Observe
		$$\bar{g}_t\bar\eta(h)\bar{g}_t^{-1}=\left\{\pm\left[\begin{array}{cc} 
		1 & ht^2 \\
	  0 & 1\end{array}\right]\right\}\mbox{ or }
		\left\{\pm\left[\begin{array}{cc} 
		1 & -ht^2 \\
	  0 & 1\end{array}\right]\right\},$$
		respectively,
		for all $h\in\Integral$ and $t\in(0,1]$.
		The path $\bar{g}_t$ lifts uniquely to be a path 
		$g_t \in\Sft\times_\Integral\Real$ parametrized by $(0,1]$,
		such that $g_1$ equals the identity.
		The elements $g_t$ satisfy the asserted properties,	as desired.
	\end{proof}

	\subsection{Commutator factorization}
	
	\begin{lemma}\label{comm_ell}
	Suppose $(\gamma_n)_{n\in\Natural}$ 
	is a sequence of elliptic elements in $\Sft$ 
	that converge to the identity element as $n$ tends to $\infty$.
	Then there exist sequences $(\alpha_n)_{n\in\Natural}$ and $(\beta_n)_{n\in\Natural}$
	in $\Sft$, both converging to the identity element as $n$ tends to $\infty$,
	such that the relation
	$$\gamma_n=\alpha_n\beta_n\alpha_n^{-1}\beta_n^{-1}$$
	holds for all but finitely many $n\in\Natural$.
	\end{lemma}

	\begin{proof}
	For any $\lambda\in(0,+\infty)$ and $\theta\in(0,\pi)$,
	we first factorize the matrix 
	$$C(\lambda,\theta)=\left[\begin{array}{cc} \cos(\theta)& \lambda\sin(\theta) \\ -\lambda^{-1}\sin(\theta)& \cos(\theta)\end{array}\right]
	\in\mathrm{SL}(2,\Real)$$
	as a commutator of two particularly constructed matrices in $\mathrm{SL}(2,\Real)$
	which depend continuously on $(\lambda,\theta)$.
	
	To this end, 	
	we construct the following matrices $A,B\in\mathrm{SL}(2,\Real)$,
	which depend continuously on the parameters $x,y\in[0,+\infty)$ and $\mu\in(0,+\infty)$:
	\begin{equation}\label{ell_AB}
	\begin{array}{cc}
	A=\left[\begin{array}{cc} \sqrt{1+y^2}+y\sqrt{1+x^2} & \mu xy \\
	-\mu^{-1} xy & \sqrt{1+y^2}-y\sqrt{1+x^2}\end{array}\right],
	&
	B=\left[\begin{array}{cc} \sqrt{1+x^2}& \mu x \\
	\mu^{-1} x & \sqrt{1+x^2}\end{array}\right].
	\end{array}
	\end{equation}
	Then the inverse matrices are easy to obtain, by switching diagonal entries and negating the off-diagonal ones:
	$$
	\begin{array}{cc}
	A^{-1}=\left[\begin{array}{cc} \sqrt{1+y^2}-y\sqrt{1+x^2} & -\mu xy \\
	\mu^{-1} xy & \sqrt{1+y^2}+y\sqrt{1+x^2}\end{array}\right],
	&
	B^{-1}=\left[\begin{array}{cc} \sqrt{1+x^2}& -\mu x \\
	-\mu^{-1} x & \sqrt{1+x^2}\end{array}\right].
	\end{array}$$
	Direct computation shows
	$$BA^{-1}B^{-1}=\left[\begin{array}{cc} \sqrt{1+y^2}-y\sqrt{1+x^2} & \mu xy \\
	-\mu^{-1} xy & \sqrt{1+y^2}+y\sqrt{1+x^2}\end{array}\right],$$
	and
	$$ABA^{-1}B^{-1}=\left[\begin{array}{cc} 1-2x^2y^2 & 2\mu xy\left(\sqrt{1+y^2}+y\sqrt{1+x^2}\right) \\
	-2\mu^{-1} xy\left(\sqrt{1+y^2}-y\sqrt{1+x^2}\right) & 1-2x^2y^2\end{array}\right].$$
	Therefore,
	given any parameters $(\lambda,\theta)\in(0,+\infty)\times (0,\pi)$,
	it is straightforward to verify the relation
	$$C(\lambda,\theta)=ABA^{-1}B^{-1}$$
	for any $x,y\in[0,+\infty)$ and $\mu\in(0,+\infty)$ 
	that satisfy the following conditions:
	\begin{equation}\label{ell_xy_mu}
	\begin{array}{ccc}
	xy=\sin(\theta/2),&
	\mu=\lambda\times \sqrt{\frac{\sqrt{1+y^2}-y\sqrt{1+x^2}}{\sqrt{1+y^2}+y\sqrt{1+x^2}}}.
	\end{array}
	\end{equation}
	
	We make the following particular assignments
	\begin{equation}\label{ell_xy_def}
	\begin{array}{ccc}
	x=\sqrt{2\tan(\theta/2)/\left(\lambda+\lambda^{-1}\right)}, & y=\sqrt{\left(\lambda+\lambda^{-1}\right)\sin(\theta)/4}.
	\end{array}	
	\end{equation}
	Then we obtain matrices $A,B\in\mathrm{SL}(2,\Real)$ 
	that depend continuously on the parameters $(\lambda,\theta)\in[0,+\infty)\times(0,\pi)$,
	by putting (\ref{ell_AB}), (\ref{ell_xy_mu}), and (\ref{ell_xy_def}) altogether.	
	
	Our particular factorization $C(\lambda,\theta)=ABA^{-1}B^{-1}$ satisfies
	$A\approx\mathbf{1}$ and $B\approx\mathbf{1}$ in $\mathrm{SL}(2,\Real)$
	if $C(\lambda,\theta)\approx\mathbf{1}$.
	In fact,
	if $C(\lambda,\theta)$ lies in a small neighborhood of the identity matrix $\mathbf{1}\in\mathrm{SL}(2,\Real)$,
	we have equivalently $\cos(\theta)\approx1$ and $\lambda^{\pm1} \sin(\theta)\approx 0$, with small error.
	Then we have $x\approx0$ and $y\approx 0$, and therefore, $\mu/\lambda\approx1$.
	This implies the following key point of our construction:
	$$\mu x=(\mu/\lambda)\times \sqrt{\frac{2\lambda\sin(\theta)}{(1+\lambda^{-2})(1+\cos(\theta))}}\approx 0,$$
	and
	$$\mu^{-1} x=(\mu/\lambda)^{-1}\times \sqrt{\frac{2\lambda^{-1}\sin(\theta)}{(1+\lambda^2)(1+\cos(\theta))}}\approx 0.$$
	We see $A\approx\mathbf{1}$ and $B\approx\mathbf{1}$ in $\mathrm{SL}(2,\Real)$
	from (\ref{ell_AB}), (\ref{ell_xy_mu}), and (\ref{ell_xy_def}).
	
	In general, any elliptic element in $\mathrm{PSL}(2,\Real)$ can be conjugated
	to $\{\pm C(\lambda,\theta)\}$ by some element in $\mathrm{SO}(2)/\{\pm\mathbf{1}\}$,
	for some parameters $(\lambda,\theta)\in(0,+\infty)\times(0,\pi)$.
	To see this, one may identify $\mathrm{PSL}(2,\Real)$ 
	with the fractional linear transformation group acting on the upper-half complex plane.
	In terms of hyperbolic geometry,
	$\{\pm C(\lambda,\theta)\}$ represents
	the rotation of angle $2\theta$ counterclockwise about the point $\lambda\mathbf{i}$. 
	The subgroup $\mathrm{SO}(2)/\{\pm\mathbf{1}\}$
	of $\mathrm{PSL}(2,\Real)$ consists of all the rotations about $\mathbf{i}$.
	Therefore, any elliptic element in $\mathrm{PSL}(2,\Real)$
	either fixes $\mathbf{i}$ already, or can be conjugated 
	by some rotation about $\mathbf{i}$ 
	to fix a point on the imaginary axis.
		
	Suppose $(\gamma_n)_{n\in\Natural}$ 
	is a sequence of elliptic elements in $\Sft$ 
	that converge to the identity element as $n$ tends to $\infty$.
	Denote by $\bar{\gamma}_n\in\mathrm{PSL}(2,\Real)$ the image of $\gamma_n$
	under the canonical covering projection $\Sft\to\mathrm{PSL}(2,\Real)$.
	Then we can find $S_n\in\mathrm{SO}(2)$ and $(\lambda_n,\theta_n)\in[1,+\infty)\times(0,\pi)$ such that
	$$\bar{\gamma}_n=\left\{\pm S_nC(\lambda_n,\theta_n)S_n^{-1}\right\},$$
	for all $n\in\Natural$.
	Let $A_n,B_n\in\mathrm{SL}(2,\Real)$ be the matrices as above, 
	such that $C(\lambda_n,\theta_n)=A_nB_nA_n^{-1}B_n^{-1}$.
	Then we obtain a commutator factorization
	$$\bar{\gamma}_n=\bar{\alpha}_n\bar{\beta}_n\bar{\alpha}_n^{-1}\bar{\beta}_n^{-1},$$
	where $\bar\alpha_n=\{\pm S_nA_nS_n^{-1}\}$ and $\bar\beta_n=\{\pm S_nB_nS_n^{-1}\}$
	are elements of $\mathrm{PSL}(2,\Real)$.	
	Because $\mathrm{SO}(2)$ is compact,
	the convergence of $(\gamma_n)_{n\in\Natural}$ implies
	$C(\lambda_n,\theta_n)\to \mathbf{1}$
	and hence,
	$A_n\to\mathbf{1}$ and $B_n\to\mathbf{1}$,
	as $n\to\infty$.
	It follows that for all but finitely many $n\in\Natural$,
	$\bar{\alpha}_n$, $\bar{\beta}_n$, and $\bar{\gamma}_n$ lie in some small neighborhood
	of $\{\pm\mathbf{1}\}\in\mathrm{PSL}(2,\Real)$
	which lifts homeomorphically to a small neighborhood of the identity element in $\Sft$.
	Then we obtain unique lifts
	$\alpha_n,\beta_n\in\Sft$ of $\bar{\alpha}_n,\bar{\beta}_n\in\mathrm{PSL}(2,\Real)$,
	which converge to the identity as $n$ tends to $\infty$,
	and which satisfy
	the relation $\gamma_n=\alpha_n\beta_n\alpha_n^{-1}\beta_n^{-1}$ for all but finitely many $n\in\Natural$,
	as asserted.
	\end{proof}

	\begin{lemma}\label{comm_hyp}
	Suppose $(\gamma_t)_{t\in[0,1)}$ 
	is a continuous family of hyperbolic elements in $\Sft$
	that converge to the identity element as $t$ tends to $1$.
	Then there exist continuous families $(\alpha_t)_{t\in[0,1)}$ and $(\beta_t)_{t\in[0,1)}$
	in $\Sft$, both converging to the identity element as $t$ tends to $1$,
	such that the relation
	$$\gamma_t=\alpha_t\beta_t\alpha_t^{-1}\beta_t^{-1}$$
	holds for all $t\in[0,1)$.
	\end{lemma}

\begin{proof}
	For any $\lambda\in(0,+\infty)$ and $\tau\in(0,+\infty)$,
	we first factorize the matrix 
	$$C(\lambda,\tau)=\left[\begin{array}{cc} 
	\cosh(\tau)& \lambda\sinh(\tau) \\ \lambda^{-1}\sinh(\tau)& \cosh(\tau)
	\end{array}\right]
	\in\mathrm{SL}(2,\Real)$$
	as a commutator of two particularly constructed matrices in $\mathrm{SL}(2,\Real)$
	which depends continuously on $(\lambda,\tau)$.
	
	We construct the following matrices $A,B\in\mathrm{SL}(2,\Real)$,
	which depend continuously on the parameters $x\in[0,1]$, $y\in[0,+\infty)$, and $\mu\in(0,+\infty)$:
	\begin{equation}\label{hyp_AB}
	\begin{array}{cc}
	A=\left[\begin{array}{cc} \sqrt{1+y^2}+y\sqrt{1-x^2} & \mu xy \\
	\mu^{-1} xy & \sqrt{1+y^2}-y\sqrt{1-x^2}\end{array}\right],
	&
	B=\left[\begin{array}{cc} \sqrt{1-x^2}& -\mu x \\
	\mu^{-1} x & \sqrt{1-x^2}\end{array}\right].
	\end{array}
	\end{equation}
	Then 
	$$ABA^{-1}B^{-1}=\left[\begin{array}{cc} 1+2x^2y^2 & 2\mu xy\left(\sqrt{1+y^2}+y\sqrt{1-x^2}\right) \\
	2\mu^{-1} xy\left(\sqrt{1+y^2}-y\sqrt{1-x^2}\right) & 1+2x^2y^2\end{array}\right].$$
	Given any parameters $(\lambda,\tau)\in(0,+\infty)\times (0,+\infty)$,
	the relation
	$$C(\lambda,\tau)=ABA^{-1}B^{-1}$$
	holds 
	for any $x\in(0,1]$ and $y\in(0,+\infty)$
	that satisfy the following conditions:
	\begin{equation}\label{hyp_xy_mu}
	\begin{array}{ccc}
	xy=\sinh(\tau/2),&
	\mu=\lambda\times \sqrt{\frac{\sqrt{1+y^2}-y\sqrt{1-x^2}}{\sqrt{1+y^2}+y\sqrt{1-x^2}}}.
	\end{array}
	\end{equation}
	
	We make the following particular assignments
	\begin{equation}\label{hyp_xy_def}
	\begin{array}{ccc}
	x=\sqrt{2\tanh(\tau/2)/(\lambda+\lambda^{-1})}, & y=\sqrt{(\lambda+\lambda^{-1})\sinh(\tau)/4}.
	\end{array}	
	\end{equation}
	Then we obtain matrices $A,B\in\mathrm{SL}(2,\Real)$ 
	that depend continously on the parameters $(\lambda,\tau)\in(0,+\infty)\times(0,+\infty)$,
	by putting (\ref{hyp_AB}), (\ref{hyp_xy_mu}), and (\ref{hyp_xy_def}) altogether.	
	A similar argument as before shows
	$A\approx\mathbf{1}$ and $B\approx\mathbf{1}$ in $\mathrm{SL}(2,\Real)$.
	if $C(\lambda,\tau)\approx\mathbf{1}$.
		
	In general, any hyperbolic element in $\mathrm{PSL}(2,\Real)$ can be conjugated
	to $\{\pm C(\lambda,\tau)\}$ by some unique element in $\mathrm{SO}(2)/\{\pm\mathbf{1}\}$,
	for some unique parameters $(\lambda,\theta)\in(0,+\infty)\times(0,+\infty)$.
	In fact, the conjugation element and the paremeters vary continuously
	as the hyperbolic element varies in $\mathrm{PSL}(2,\Real)$.
	To see this, one may again identify $\mathrm{PSL}(2,\Real)$ 
	with the fractional linear transformation group acting on the upper-half complex plane.
	In terms of hyperbolic geometry,
	$\{\pm C(\lambda,\theta)\}$ represents
	the hyperbolic translation 
	along the geodesic between the ideal endpoints $\{\pm\lambda\}$
	and of distance $2\tau$ toward $+\lambda$.
	The subgroup $\mathrm{SO}(2)/\{\pm\mathbf{1}\}$	
	of $\mathrm{PSL}(2,\Real)$ consists of all the rotations about $\mathbf{i}$.
	Therefore, any hyperbolic element in $\mathrm{PSL}(2,\Real)$
	can be conjugated by some unique rotation about $\mathbf{i}$ 
	to preserve a geodesic perpendicular to the positive imaginary axis,
	such that the ideal point $0$ is moved rightward.
	
	Suppose $(\gamma_t)_{t\in[0,1)}$ 
	is a continuous family of hyperbolic elements in $\Sft$ 
	that converge to the identity element as $t$ tends to $1$.
	Denote by $\bar{\gamma}_t\in\mathrm{PSL}(2,\Real)$ the image of $\gamma_t$
	under the canonical covering projection $\Sft\to\mathrm{PSL}(2,\Real)$.
	Then we can find continuous families 
	$\{\pm S_t\}\in\mathrm{SO}(2)/\{\pm\mathbf{1}\}$
	and 
	$(\lambda_t,\tau_t)\in(0,+\infty)\times(0,\pi)$ such that
	$$\bar{\gamma}_t=\left\{\pm S_tC(\lambda_t,\tau_t)S_t^{-1}\right\},$$
	for all $t\in[0,1)$.
	(One may actually require $S_t\in\mathrm{SO}(2)$
	to depend continuously on $t$, by path lifting.)
	Let $A_t,B_t\in\mathrm{SL}(2,\Real)$ be the matrices as above, 
	such that $C(\lambda_t,\tau_t)=A_tB_tA_t^{-1}B_t^{-1}$.
	Then we obtain a commutator factorization
	$$\bar{\gamma}_t=\bar{\alpha}_t\bar{\beta}_t\bar{\alpha}_t^{-1}\bar{\beta}_t^{-1},$$
	where $\bar\alpha_t=\{\pm S_tA_tS_t^{-1}\}$ and $\bar\beta_t=\{\pm S_tB_tS_t^{-1}\}$
	are elements of $\mathrm{PSL}(2,\Real)$.	
	Because $\mathrm{SO}(2)$ is compact,
	the convergence of $(\gamma_t)_{t\in[0,1)}$ implies
	that $\bar{\alpha}_t$, $\bar{\beta}_t$
	converges to $\{\pm\mathbf{1}\}\in\mathrm{PSL}(2,\Real)$
	as $t$ tends to $1$.
	Then we obtain unique lifts
	$\alpha_t,\beta_t\in\Sft$ of $\bar{\alpha}_t,\bar{\beta}_t\in\mathrm{PSL}(2,\Real)$,
	which converge to the identity as $t$ tends to $1$,
	and which satisfy
	the relation $\gamma_t=\alpha_t\beta_t\alpha_t^{-1}\beta_t^{-1}$ 
	for all $t\in[0,1)$, as asserted.
	\end{proof}

	\begin{lemma}\label{comm_par}
	Suppose $(\gamma_t)_{t\in[0,1)}$ 
	is a continuous family of parabolic elements in $\Sft$
	that converge to the identity element as $t$ tends to $1$.
	Then there exist continuous families $(\alpha_t)_{t\in[0,1)}$ and $(\beta_t)_{t\in[0,1)}$
	in $\Sft$, both converging to the identity element as $t$ tends to $1$,
	such that the relation
	$$\gamma_t=\alpha_t\beta_t\alpha_t^{-1}\beta_t^{-1}$$
	holds for all $t\in[0,1)$.
	\end{lemma}

\begin{proof}
	For any $u\in(0,+\infty)$, we factorize the matrices
	$$C(u)=\left[\begin{array}{cc} 
	1& u \\ 0& 1
	\end{array}\right]
	\in\mathrm{SL}(2,\Real)$$
	as a commutator
	$$C(u)=ABA^{-1}B^{-1},$$
	where
	\begin{equation}\label{par_AB}
	\begin{array}{cc}
	A=\left[\begin{array}{cc} 1 & \sqrt{u}+u \\
	0 & 1\end{array}\right],
	&
	B=\left[\begin{array}{cc} \frac{1}{\sqrt{1+\sqrt{u}}}& 0 \\
	0 & \sqrt{1+\sqrt{u}}\end{array}\right].
	\end{array}
	\end{equation}
	The matrices $A,B\in\mathrm{SL}(2,\Real)$
	depend continuously on $u\in(0,+\infty)$, 
	and converge to the identity matrix as $u$ tends to $0$.
	Observe $C(u)^{-1}$ factorizes as a commutator $BAB^{-1}A^{-1}$.	
	In general, any parabolic element in $\mathrm{PSL}(2,\Real)$ can be conjugated
	to $\{\pm C(u)\}$ or $\{\pm C(u)^{-1}\}$
	by some unique element in $\mathrm{SO}(2)/\{\pm\mathbf{1}\}$,
	for some unique parameter $u\in(0,+\infty)$.
	In fact, the conjugation element and the paremeter vary continuously
	as the parabolic element varies in $\mathrm{PSL}(2,\Real)$.
	On the upper-half complex plane,
	the conjugation and the parameter are determined
	by the requirement that the ideal fixed point of the parabolic element
	is conjugated to $\infty$.
			
	Suppose $(\gamma_t)_{t\in[0,1)}$ 
	is a continuous family of parabolic elements in $\Sft$ 
	that converge to the identity element as $t$ tends to $1$.
	Denote by $\bar{\gamma}_t\in\mathrm{PSL}(2,\Real)$ the image of $\gamma_t$
	under the canonical covering projection $\Sft\to\mathrm{PSL}(2,\Real)$.
	Then we can find continuous families 
	$\{\pm S_t\}\in\mathrm{SO}(2)/\{\pm\mathbf{1}\}$
	and 
	$u_t\in(0,+\infty)\times(0,\pi)$ such that
	$$\bar{\gamma}_t=\left\{\pm S_tC(u_t)S_t^{-1}\right\}\mbox{ or }
	\left\{\pm S_tC(u_t)^{-1}S_t^{-1}\right\}$$
	for all $t\in[0,1)$.
	Let $A_t,B_t\in\mathrm{SL}(2,\Real)$ be the matrices as above, 
	such that $C(u_t)=A_tB_tA_t^{-1}B_t^{-1}$.
	Possibly after switching $A_t$ and $B_t$,
	we obtain a commutator factorization
	$$\bar{\gamma}_t=\bar{\alpha}_t\bar{\beta}_t\bar{\alpha}_t^{-1}\bar{\beta}_t^{-1},$$
	where $\bar\alpha_t=\{\pm S_tA_tS_t^{-1}\}$ and $\bar\beta_t=\{\pm S_tB_tS_t^{-1}\}$
	are elements of $\mathrm{PSL}(2,\Real)$.	
	Then we obtain unique lifts
	$\alpha_t,\beta_t\in\Sft$ of $\bar{\alpha}_t,\bar{\beta}_t\in\mathrm{PSL}(2,\Real)$,
	which converge to the identity as $t$ tends to $1$,
	and which satisfy
	the relation $\gamma_t=\alpha_t\beta_t\alpha_t^{-1}\beta_t^{-1}$ 
	for all $t\in[0,1)$, as asserted.
	\end{proof}

\section{Reduction to the virtually central case with extra genera}\label{Sec-reduction_vc}

	
		Let $M$ be a formatted graph manifold. 
		For any ordinary fiber $f$ in a JSJ piece of $M$
		with a fibered collar neighborhood parametrized as $f\times D^2$,
		we can construct a new formatted graph manifold:
		$$M\#_{f}\left(f\times T^2\right)= 
		\left(M\setminus\mathrm{int}\left(f\times D^2\right)\right)\cup_{f\times \partial D^2}\left(f\times \left(D^2\# T^2\right)\right),$$
		where $D^2\#T^2$ stands for (a fixed model of) 
		a connected sum of a compact disk $D^2$ with a torus $T^2$.
		Namely, $M\#_f(f\times T^2)$ is obtained from $M$
		by replacing the interior of $f\times D^2$
		with the interior of $f\times (D^2\# T^2)$,
		retaining the topology near $f\times \partial D^2$.
		The construction comes naturally with a degree-one map
		\begin{equation}\label{def_horizontal_pinching_map}
		M\#_{f}\left(f\times T^2\right)\to M,
		\end{equation}
		which is the identity map restricted to $M\setminus\mathrm{int}(f\times D^2)$,
		and which is the product of the identity map $f\to f$ with 
		the pinching map $D^2\# T^2\to D^2$
		restricted to $f\times (D^2\# T^2)$.
		We refer to $M\#_f(f\times T^2)$ as the \emph{fiber connected sum} of $M$ and $f\times T^2$,
		and the map (\ref{def_horizontal_pinching_map})
		as the \emph{horizontally pinching map}.
		
		The JSJ graph of $M\#_f(f\times T^2)$ can be identified with the simplicial JSJ graph $(V,E)$ of $M$,
		and the horizontally pinching map induces the identity map between the JSJ graphs.
		If $f$ is an ordinary fiber $f_v$ of the JSJ piece corresponding to a vertex $v\in V$,
		we observe the simple relations
		\begin{equation}\label{pinch_e_chi}
		\begin{array}{cc}
		\opEuler_{M\#_f(f\times T^2)}=\opEuler_M, &
		\chi_{M\#_f(f\times T^2)}=\chi_M-2v^*,
		\end{array}
		\end{equation}
		(see Notation \ref{notation_xi_tau}).
	
		\begin{definition}\label{def_formatted_pinching}
			A map between formatted graph manifolds is called a \emph{formatted pinching map}
			if it admits a factorization as 
			the composition of finitely many maps between formatted graph manifolds,
			such that each of the factor maps is (formatted homeomorphically)
			conjugate to a horizontally pinching map of the form (\ref{def_horizontal_pinching_map}).
		\end{definition}
	
		\begin{remark}\label{remark_formatted_pinching}\
		\begin{enumerate}
		\item
		Let $M$ be a formatted graph manifold with a simplicial JSJ graph $(V,E)$.
		For any union of finitely many mutually disjoint of ordinary fibers 
		$\mathcal{F}$ in JSJ pieces of $M$,
		we can construct 
		the simultaneous fiber connected sum
		$M\#_{\mathcal{F}}(\mathcal{F}\times T^2)$,
		namely,
		$(M\setminus\mathrm{int}(\mathcal{F}\times D^2))
		\cup_{\mathcal{F}\times \partial D^2}(\mathcal{F}\times(D^2\# T^2))$.
		So there is a simultaneous horizontally pinching map
		$M\#_{\mathcal{F}}(\mathcal{F}\times T^2)\to M$.
		\item 
		In general, every formatted pinching map $M'\to M$
		can be written as a composite map $M'\to M\#_{\mathcal{F}}(\mathcal{F}\times T^2)\to M$,
		where the first factor is a formatted homeomorphism and the second factor is a simultaneous horizontally pinching map.
		It is instructive to look at the composition of two horizontally pinching maps,
		and proceed by induction.
		\item
		One might also be interested in
		the fiber connected sum $M\#_f (f\times T^2)$ where $f$ 
		is an exceptional fiber in a JSJ piece.
		It will give rise to a degree--$1$ map $M\#_f (f\times T^2)\to M$ defined similarly.
		In that case, however,
		the JSJ decomposition of $M\#_f (f\times T^2)$ will have a new JSJ torus $f\times \partial D^2$
		and a new JSJ piece	$f\times\mathrm{int}(D^2\# T^2)$.
		\end{enumerate}
		\end{remark}

	\subsection{Local reduction}

	\begin{theorem}\label{local_reduction_vc}
		Let $M'\to M$ be a formatted pinching map between 
		formatted graph manifolds
		with identified simplicial JSJ graphs $(V,E)$.
		Let $\rho\colon \pi_1(M)\to\Sft\times_\Integral\Real$ 
		be a representation.
		Suppose that
		$C=(V_C,E_C)$ be a maximal connected subgraph of $(V,E)$	over which $\rho$ is abelian.
		Adopt Notations \ref{notation_xi_tau} and \ref{notation_delta}.
		Suppose that the inequality
		$$\chi_{M'}(v)\leq\chi_M(v)-2\cdot\delta_\rho(v;C)$$
		holds for all $v\in V$.
		
		Then, for any constant $\epsilon>0$,
		there exists a representation $\rho'\colon\pi_1(M')\to \Sft\times_\Integral\Real$
		which satisfies all the following properties:
		\begin{itemize}
		\item 
		In the space of representations $\mathcal{R}(\pi_1(M'),\Sft\times_\Integral\Real)$,
		$\rho'$ is path-connected with the pull-back representation of $\rho$.
		\item
		For any $v\in V$,
		$$\left|\xi_{\rho'}(v)-\xi_{\rho}(v)\right|<\epsilon.$$
		\item
		For any $\{v,w\}\in E$,
		$\tau_{\rho'}(v,w)$ equals $\tau_{\rho}(v,w)$ 
		if $\tau_{\rho}(v,w)$ is finite.
		\item
		For any $v\in V_C$, $\delta_{\rho'}(v)$ equals $0$.
		\end{itemize}
		In fact, one may require, furthermore,
		$\xi_{\rho'}(v)=\xi_\rho(v)$ if $v$ is not a vertex of $C$,
		and $\tau_{\rho'}(v,w)=\tau_{\rho}(v,w)$ if neither $v$ nor $w$ are vertices of $C$.		
	\end{theorem}
	
	The rest of this subsection is devoted to the proof of Theorem \ref{local_reduction_vc}.

	\subsubsection{The standard local model}
		Let $M$ be any formatted graph manifold with a simplicial JSJ graph $(V,E)$,
		and $\rho\colon \pi_1(M)\to\Sft\times_\Integral\Real$ any representation.
		Let $C=(V_C,E_C)$ be a maximal connected subgraph of $(V,E)$ over which $\rho$ is abelian type.
		Denote by $M_C$ the block submanifold of $M$ over $C$, namely,
		$$M_C=\bigcup_{w\in V_C} J_w\cup \bigcup_{\{w,u\}\in E_C} T_{w,u}.$$				
		Denote by	$\Delta_\rho(C)$
		the set of directed edges	which depart from $V_C$ and arrive into $V\setminus V_C$,
		and which fail to be virtually central with respect to $\rho$,
		namely,
		$$\Delta_\rho(C)=\left\{(w,v)\in V_C\times(V\setminus V_C)\colon \{v,w\}\in E\mbox{ and }\tau_\rho(v,w)=\infty\right\}.$$		
		These directed edges correspond to a subset of outward oriented JSJ tori 
		on the boundary of the closure of $M_C$.
		We construct a particular formatted graph manifold $M^\dagger=M^\dagger(\Delta_\rho(C))$,
		together with a formatted pinching map 
		\begin{equation}\label{blowup}
		M^\dagger\to M,
		\end{equation}
		such that the equality
		\begin{equation}\label{standard_equal}
		\chi_{M^\dagger}(v)=\chi_M(v)-2\cdot\delta_\rho(v,C)
		\end{equation}
		holds for all $v\in V$.
		The construction is as follows.
		
		For any $(w,v)\in \Delta_\rho(C)$,
		take a collar neighborhood of $T_{v,w}$ in $\mathrm{clos}(J_v)$, 
		parametrized as $T_{v,w}\times[0,1]$, 
		where $T_{v,w}\times\{0\}$ is the JSJ torus $T_{v,w}$.
		We require that these $T_{v,w}\times[0,1]$ are mutually disjoint in $M$.
		Take an ordinary fiber $f_v$ of $J_v$ in $T_{v,w}\times(0,1)$.
		We obtain a compact submanifold of $M$ with boundary:
		$$W_C=\mathrm{clos}(M_C)\cup\bigcup_{(w,v)\in \Delta_\rho(C)} T_{v,w}\times[0,1].$$
		Construct a simultaneous fiber connected sum of $W_C$ with trivialized circle bundles over tori:
		$$W^\dagger_C=\mathrm{clos}(M_C)\cup\bigcup_{(w,v)\in \Delta_\rho(C)} \left(T_{v,w}\times[0,1]\right)\#_{f_v}\left(f_v\times T^2\right).$$
		The formatted graph manifold $M^\dagger$ is constructed as
		$$M^\dagger=\left(M\setminus\mathrm{int}(W_C)\right)\cup_{\partial W_C} W^\dagger_C.$$
		The formatted pinching map $M^\dagger\to M$
		is defined as the identity map on
		$M\setminus\mathrm{int}(W_C)$
		together with the simultaneous horizontal pinching map $W^\dagger_C\to W_C$.
		
		We introduce some additional notations to describe the structure of $M^\dagger$.
		The collection of tori $\partial W^\dagger_C$
		decomposes $M^\dagger$ into the interior of $W^\dagger_C$
		together with the connected components $X_q$ of $M\setminus W_C$,
		indexed by $q\in\pi_0(M\setminus W_C)$.
		We denote by $X_{q(v)}$ the component that contains $J_v$.
		The collection of tori $T_{v,w}\times\{0\}$ decomposes 
		the interior of $W^\dagger_C$ into $M_C$ 
		together with the interiors of the fiber connected sums
		$$W_{w,v}=\left(T_{v,w}\times[0,1]\right)\#_{f_v}\left(f_v\times T^2\right).$$
		For each $(w,v)\in\Delta_\rho(C)$,
		fix an oriented slope $c_{v,w}$ on $T_{v,w}$ 
		such that $T_{v,w}$ is parametrized as the product torus $f_v\times c_{v,w}$.
		Then $W_{w,v}$ can be parametrized as a product
		$$W_{w,v}=f_v\times\Omega_{w,v},$$
		where $\Omega_{w,v}$ denotes the connected sum of the annulus $c_{v,w}\times[0,1]$
		with the torus $T^2$.
		
		The following commutative diagram summarizes our above construction 
		about any $(v,w)\in\Delta_\rho(C)$.
		Each arrow indicates an inclusion, and each equality symbol indicates a parametrization:
		\begin{equation}\label{diagram_decompositions}
		\xymatrix{
		f_v\times c_{v,w}\times\{0\} \ar[rd] \ar@{=}[r] & T_{v,w}\times\{0\} \ar[r] \ar[rd] & \mathrm{clos}\left(M_C\right) \ar[rd] & & \\
		& f_v\times \Omega_{w,v} \ar@{=}[r] & W^\dagger_{v,w} \ar[r] & W_C	\ar[r] & M^\dagger\\
		f_v\times c_{v,w}\times\{1\} \ar[ru] \ar@{=}[r] & T_{v,w}\times\{1\} \ar[rr] \ar[ru] & & \mathrm{clos}\left(X_{q(v)}\right) \ar[ru] & 
		}
		\end{equation}
		
		\subsubsection{Associated graph-of-groups decompositions}		
		The structure of $\pi_1(M^\dagger)$ can be described with graph-of-groups decompositions.
		We briefly recall the terminology in group theory for the reader's reference.
		Then we point out necessary choices and elaborate the decomposition.
		
		Recall that a \emph{graph of groups} $\mathcal{G}$ refers to a collection of data as follows:
		A connected graph $\Lambda$, regarded as a connected finite cell $1$--complex;
		a group $G_v$ for every vertex ($0$--cell) $v$ of $\Lambda$;
		a group $G_e$ for every edge ($1$--cell) $e$ of $\Lambda$;
		and an injective homomorphism $i_\delta\colon G_e\to G_v$
		for every end $\delta$ of an edge $e$ attached to a vertex $v$.
		If $\Lambda$ contains a unqiue vertex $v$ and no edges, 
		the fundamental group $\pi_1(\mathcal{G})$ is defined as $G_v$.
		If $\Lambda$ contains at least one edge $e$, 
		$\pi_1(\mathcal{G})$ is defined recursively using free amalgamations and HNN extensions.
		To be precise, choose an orientation of $e$ to distinguish its positive and negative ends $e^{\pm}$.
		When $e$ is separating in $\Lambda$, 
		$\pi_1(\mathcal{G})$ is defined as the free amalgamation
		$$\pi_1\left(\mathcal{G}\right)=
		\left(\pi_1\left(\mathcal{G}_{(\Lambda\setminus e)^-}\right)*\pi_1\left(\mathcal{G}_{(\Lambda\setminus e)^+}\right)\right)
		/
		\left\langle\!\left\langle i_{e^-}(g)^{-1}i_{e^+}(g)\colon g\in G_e \right\rangle\!\right\rangle,$$
		where	$\mathcal{G}_{(\Lambda\setminus e)^\pm}$ stands for
		the subgraph of groups over the connected component $(\Lambda\setminus e)^\pm$ of $\Lambda\setminus e$
		attached on the end $e^\pm$.
		When $e$ is nonseparating in $\Lambda$, 
		$\pi_1(\mathcal{G})$ is defined as the HNN extension
		$$\pi_1\left(\mathcal{G}\right)=
		\left(\pi_1\left(\mathcal{G}_{\Lambda\setminus e}\right)*\langle t_e\rangle\right)
		/
		\left\langle\!\left\langle i_{e^-}(g)^{-1}t_ei_{e^+}(g)t_e^{-1}\colon g\in G_e \right\rangle\!\right\rangle,$$
		where	$\mathcal{G}_{\Lambda\setminus e}$ stands for
		the subgraph of groups over the connected subgraph $\Lambda\setminus e$,
		and where $t_e$ is a distinguished free letter called the \emph{stable letter}.
		Hence the group $\pi_1(\mathcal{G})$ admits a presentation
		whose generators are generators of the vertex groups together with
		$b_1(\Lambda)$ stable letters.
		The relators of the presentation are either relators of the vertex groups
		or the relators arising from the free amalgamations and the HNN extensions,
		which correspond to generators of the edge groups.
		For different choices of the edge $e$ and its orientation,
		the resulting fundamental group $\pi_1(\mathcal{G})$ from the reduction procedure
		is actually unique up to canonical isomorphisms.
		A \emph{graph-of-groups decomposition} of a group $G$ refers to a graph of groups $\mathcal{G}$
		together with an isomorphism $G\cong \pi_1(\mathcal{G})$.		
		
		Suppose that we have chosen	basepoints for all the spaces in the diagram (\ref{diagram_decompositions}),
		and have chosen paths connecting from the included basepoints to the basepoints of the target spaces 
		for all the arrows thereof.
		Then the arrows induce injective homomorphisms between the fundamental groups
		of the pointed subspaces.
		One may adjust the homomorphisms with conjugations in the target groups
		by adjusting of the paths relative to the basepoints.		
		After fixing such choices, 
		we will obtain a graph-of-groups decomposition of $\pi_1(M^\dagger)$,
		whose vertex groups are $\pi_1(W^\dagger_C)$ and the subgroups $\pi_1(X_q)$,
		and whose edge groups are the subgroups $\pi_1(T)$ 
		corresponding to the boundary components $T$ of $W^\dagger_C$.
		We will also obtain a graph-of-groups decomposition of $\pi_1(W^\dagger_C)$,
		whose vertex groups are $\pi_1(M_C)$ and the subgroups $\pi_1(W^\dagger_{v,w})$,
		and whose edge groups are the subgroups $\pi_1(T_{v,w}\times\{0\})$.
		
		We choose basepoints and paths as above for objects in the diagram (\ref{diagram_decompositions}),
		and then identify $\pi_1(M^\dagger)$ and $\pi_1(W^\dagger_C)$ with their graph-of-groups decompositions.
		We actually adjust the paths to make the homomorphisms 
		$\pi_1(T_{v,w}\times\{0\})\to\pi_1(W^\dagger_{w,v})$
		and $\pi_1(T_{v,w}\times\{1\}\to\pi_1(W^\dagger_{w,v})$
		look nicer, as follows.
		Fix a presentation of $\pi_1(\Omega_{w,v})$
		with four generators $x_{v,w},y_{v,w},c_{0,v,w},c_{1,v,w}$
		and one relation
		\begin{equation}\label{relation_Omega}
		x_{v,w}y_{v,w}x_{v,w}^{-1}y_{v,w}^{-1}=c_{1,v,w}c_{0,v,w}^{-1}.
		\end{equation}
		We require that the conjugation classes of $c_{0,v,w}$ and $c_{1,v,w}$ in $\pi_1(\Omega_{w,v})$
		represent	the free homotopy classes of the loops $c_{v,w}\times\{0\}$ and $c_{v,w}\times\{1\}$, 
		in $\Omega_{w,v}$ respectively.
		(The presentation follows immediately from the connected sum decomposition of $\Omega_{w,v}$
		and the van Kampen theorem.)
		This yields a presentation of
		$\pi_1(W^\dagger_{w,v})$ with five generators $f_v,x_{v,w},y_{v,w},c_{0,v,w},c_{1,v,w}$,
		and the five relations,
		namely,
		the relation (\ref{relation_Omega})
		and another four relations saying that $f_v$ commutes with 
		any other generators.
		We also require that the conjugation class of $f_v$ in $\pi_1(W^\dagger_{w,v})$
		represents the free homotopy classes of the ordinary fiber $f_v$.
		After suitable adjustment of the chosen paths,
		we assume that the image of 
		$\pi_1(T_{v,w}\times\{0\})\to \pi_1(W^\dagger_{w,v})$
		is the free abelian subgroup generated by $c_{0,v,w}$ and $f_v$.
		We also assume
		that the image of $\pi_1(T_{v,w}\times\{1\})\to \pi_1(W^\dagger_{w,v})$ is
		the free abelian subgroup generated by $c_{1,v,w}$ and $f_v$.
		
		\subsubsection{Reductions for the standard local model}
		
		\begin{lemma}\label{ell_reduction_vc}
		The statement of Theorem \ref{local_reduction_vc} holds true if 
		$M'$ is the standard local model $M^\dagger$
		and if $\rho$ is of elliptic type over $C$.
		\end{lemma}
		
		\begin{proof}
		Let $M$, $(V,E)$, $\rho$, and $C=(V_C,E_C)$ be
		as assumed in Theorem \ref{local_reduction_vc}.
		Suppose that $\rho$ is of elliptic type over $C$.
		Given any constant $\epsilon>0$,
		we show that a representation $\rho'$ of $\pi_1(M^\dagger)$ as asserted
		can be obtained by modifying the pull-back representation	$\rho^\dagger_{\mathtt{pb}}$ of $\rho$.
		Our modification is supported at the vertex subgroup $\pi_1(W^\dagger_C)$ of $\pi_1(M^\dagger)$,
		in the sense of the following description:
		On any vertex group $\pi_1(X_q)$ or any stable letter,
		$\rho'$ is defined as the restriction of $\rho^\dagger_{\mathtt{pb}}$.
		Meanwhile, 
		$\rho'$ is defined on the vertex group $\pi_1(W^\dagger_C)$,
		and equals $\rho^\dagger_{\mathtt{pb}}$ on any incoming edge group,
		namely, the image of any $\pi_1(T_{v,w}\times\{1\})\to \pi_1(W^\dagger_C)$.
		In fact, we construct a sequence of supported modifications
		$\{\rho^\dagger_n\}_{n\in\Natural}$,
		which converges to $\rho^\dagger_{\mathtt{pb}}$	as $n$ tends to $\infty$. 
		For all sufficiently large $n$, 
		we show that $\rho'$ can be taken as $\rho^\dagger_n$.		
		
		The restriction of $\rho^\dagger_{\mathtt{pb}}$ to $\pi_1(\mathrm{clos}(M_C))$
		equals $\rho$, and factors through an elliptic-type representation
		$\eta\colon H\to \Sft\times_\Integral\Real$,
		where $H$ stands for the abelianizaiton of $\pi_1(\mathrm{clos}(M_C))$.
		We apply Lemma \ref{ab_rep_ell},
		and obtain a sequence of perturbed elliptic-type representations
		$\eta_n\colon H\to \Sft\times_\Integral\Real$, indexed by $n\in\Natural$.
		Denote by $\rho^\dagger_n\colon \pi_1(\mathrm{clos}(M_C))\to \Sft\times_\Integral\Real$
		the pull-back representations of $\eta_n$, for all $n\in\Natural$.
		For any $(w,v)\in\Delta_\rho(C)$,
		the maximality of $C$ implies that $\rho$ is of central type at $v$,
		so $\rho^\dagger_{\mathtt{pb}}(f_v)$ is 
		by definition the central element
		$\xi_\rho(v)\in\Real$ of $\Sft\times_\Integral\Real$.
		Then the asserted properties of $\eta_n$ imply that
		$\rho^\dagger_n(f_v)=\xi_\rho(v)$ is central and constant for all $n\in\Natural$.
		We also observe $\rho^\dagger_{\mathtt{pb}}(c_{1,v,w})=\rho^\dagger_{\mathtt{pb}}(c_{0,v,w})$.
		Then the asserted properties of $\eta_n$ imply 
		$\rho^\dagger_{\mathtt{pb}}(c_{1,v,w})\rho^\dagger_n(c_{0,v,w})^{-1}\in\Sft$ for all $n\in\Natural$,
		and $\lim_{n\to\infty}\rho^\dagger_n(c_{0,v,w})=\rho^\dagger_{\mathtt{pb}}(c_{1,v,w})$.
		We apply Lemma \ref{comm_ell} to 
		$\gamma_n=\rho^\dagger_{\mathtt{pb}}(c_{1,v,w})\rho^\dagger_n(c_{0,v,w})^{-1}$,
		and obtain some elements $\alpha_n,\beta_n\in\Sft$,
		for all but finitely many $n\in\Natural$.
		Remember $\gamma_n=\alpha_n\beta_n\alpha_n^{-1}\beta_n^{-1}$
		and $\lim_{n\to\infty}\alpha_n=\lim_{n\to\infty}\beta_n=\mathrm{id}$. 
		Then we can extend the restriction of $\rho^\dagger_n$ to $\pi_1(T_{v,w}\times\{0\})$
		to be a representation of $\pi_1(W^\dagger_{w,v})$,
		defining
		$$\begin{array}{ccc}
		\rho^\dagger_n(x_{v,w})=\alpha_n, & \rho^\dagger_n(y_{v,w})=\beta_n, & \rho^\dagger_n(c_{1,v,w})=\rho^\dagger_{\mathtt{pb}}(c_{1,v,w}),
		\end{array}$$
		for all but finitely many $n\in\Natural$.
		Note that the relation (\ref{relation_Omega}) is preserved under $\rho^\dagger_n$,
		namely,
		$\rho^\dagger(x_{v,w})\rho^\dagger(y_{v,w})\rho^\dagger(x_{v,w})^{-1}\rho^\dagger(x_{v,w})^{-1}
		=\alpha_n\beta_n\alpha_n^{-1}\beta_n^{-1}=\gamma_n
		=\rho^\dagger_n(c_{1,v,w})\rho^\dagger_n(c_{0,v,w})^{-1}$;
		the commutativity relations of $f_v$	with $x_{v,w}$, $y_{v,w}$, $c_{0,v,w}$, and $c_{1,v,w}$ are
		also preserved under $\rho^\dagger_n$, since $\rho^\dagger_n(f_v)=\xi_\rho(v)$ is central.
		Perform such extension for all $(w,v)\in\Delta_\rho(C)$.
		Then, for all but finitely many $n\in\Natural$,
		we obtain representations $\rho^\dagger_n\colon \pi_1(W^\dagger_C)\to \Sft\times_\Integral\Real$,
		which all coincide with $\rho^\dagger_{\mathtt{pb}}$ on the incoming edge groups $\pi_1(T_{v,w}\times\{1\})$.
		Therefore, 
		they extend uniquely to be representations $\rho^\dagger_n\colon \pi_1(M^\dagger)\to \Sft\times_\Integral\Real$,
		which all coincide with $\rho^\dagger_{\mathtt{pb}}$ on any vertex groups $\pi_1(X_q)$ and any stable letter.
		
		Our construction guarantees $\lim_{n\to\infty}\rho^\dagger_n=\rho^\dagger_{\mathtt{pb}}$ 
		in $\mathcal{R}(\pi_1(M^\dagger),\Sft\times_\Integral\Real)$.
		By Proposition \ref{lifting_representations},
		$\mathcal{R}(\pi_1(M^\dagger),\Sft\times_\Integral\Real)$ is locally path-connected.
		Then $\rho^\dagger_n$ lies in the path-connected component of $\rho^\dagger_{\mathtt{pb}}$
		when $n$ is sufficiently large.
		For any $v\in V$,
		if $v\not\in V_C$,
		the supported modification implies $\xi_{\rho^\dagger_n}(v)=\xi_\rho(v)$;
		otherwise,
		the continuity of the essential winding number implies
		$\lim_{n\to\infty}\xi_{\rho^\dagger_n}(v)=\xi_\rho(v)$.		
		For any $\{v,w\}\in E$,
		if $\{v,w\}\cap V_C=\emptyset$,
		the supported modification implies $\tau_{\rho^\dagger_n}(v,w)=\tau_\rho(v,w)$;
		otherwise,		
		the asserted properties of $\eta_n$ implies	
		that $\tau_{\rho^\dagger_n}(v,w)$ is finite, 
		and
		$\lim_{n\to\infty} \tau_{\rho^{\dagger}_n}(v,w)=\tau_\rho(v,w)$,
		so for any sufficiently large $n$,
		$\tau_{\rho^\dagger_n}(v,w)$ equals $\tau_{\rho}(v,w)$ if $\tau_\rho(v,w)$ is finite.
		Therefore, for the given constant $\epsilon>0$, 
		we take $\rho'$ to be $\rho^\dagger_n$ for some sufficiently large $n$,
		and $\rho'$ is as desired.
		\end{proof}
		
		\begin{lemma}\label{hyp_reduction_vc}
		The statement of Theorem \ref{local_reduction_vc} holds true if 
		$M'$ is the standard local model $M^\dagger$
		and if $\rho$ is of hyperbolic type over $C$.
		\end{lemma}
		
		\begin{proof}
		Suppose that $\rho$ is of hyperbolic type over $C$.
		We follow a similar procedure as in the elliptic case (Lemma \ref{ell_reduction_vc}),
		but use a path of supported modifications of $\rho^\dagger_{\mathtt{pb}}$ instead of a sequence.
		To be precise,
		we apply Lemma \ref{ab_rep_hyp} to obtain a path of hyperbolic-type representations
		$\eta_t\colon H\to \Sft\times_\Integral\Real$, parametrized by $t\in(0,1]$.
		Extend the path continuously to $t=0$ by defining $\eta_0(h)=\winding(\eta(h))$ 
		for all $h\in H$.
		Denote by $\rho^\dagger_t\colon \pi_1(\mathrm{clos}(M_C))\to \Sft\times_\Integral\Real$
		the pull-back representations of $\eta_t$, for all $t\in[0,1]$.
		In particular, $\rho^\dagger_t$ is 
		the restriction of $\rho^\dagger_{\mathtt{pb}}$ to $\pi_1(\mathrm{clos}(M_C))$ at $t=1$.
		For any $(w,v)\in\Delta_\rho(C)$,
		we apply Lemma \ref{comm_hyp} to $\gamma_t=\rho^\dagger_{\mathtt{pb}}(c_{1,v,w})\rho^\dagger_t(c_{0,v,w})^{-1}$,
		and obtain some elements $\alpha_t,\beta_t\in\Sft$, for $t\in[0,1)$.
		Remember $\gamma_t=\alpha_t\beta_t\alpha_t^{-1}\beta_t^{-1}$
		and $\lim_{t\to1}\alpha_t=\lim_{t\to1}\beta_t=\mathrm{id}$,
		so we define $\alpha_0=\beta_0=\mathrm{id}$ at $t=0$.
		Then we can extend the restriction of $\rho^\dagger_t$ to $\pi_1(T_{v,w}\times\{0\})$
		to be a representation of $\pi_1(W^\dagger_{w,v})$,
		defining
		$$\begin{array}{ccc}
		\rho^\dagger_t(x_{v,w})=\alpha_t, & \rho^\dagger_t(y_{v,w})=\beta_t, & \rho^\dagger_t(c_{1,v,w})=\rho^\dagger_{\mathtt{pb}}(c_{1,v,w}),
		\end{array}$$
		for all $[0,1]$.
		Extend this way for all $\pi_1(W^\dagger_{w,v})$,
		and extend further by $\rho^\dagger_{\mathtt{pb}}$ on all $\pi_1(X_q)$ and stable letters.
		Then we obtain representations $\rho^\dagger_t\colon \pi_1(M^\dagger)\to \Sft\times_\Integral\Real$,
		which varies continuously for all $t\in[0,1]$,
		and which equals $\rho_{\mathtt{pb}}$ at $t=1$.
		For any $v\in V$, $\xi_{\rho^\dagger_t}(v)=\xi_\rho(v)$ holds for all $v\in V$ at $t=0$.
		(The constructions in Lemmas \ref{ab_rep_hyp} and \ref{comm_hyp} actually 
		make $\xi_{\rho^\dagger_t}=\xi_\rho(v)$ for all $t\in[0,1]$.)		
		For any $\{v,w\}\in E$,
		if $\{v,w\}\cap V_C=\emptyset$,
		$\tau_{\rho^\dagger_t}(v,w)=\tau_\rho(v,w)$ holds for all $t\in[0,1]$;
		otherwise,
		$\tau_{\rho^\dagger_t}(v,w)=1$ holds at $t=0$,
		while $\tau_\rho(v,w)$ is either $1$ or $\infty$.
		We take $\rho'$ to be $\rho^\dagger_0$, as desired.
		\end{proof}
		
		\begin{lemma}\label{par_reduction_vc}
		The statement of Theorem \ref{local_reduction_vc} holds true if 
		$M'$ is the standard local model $M^\dagger$
		and if $\rho$ is of parabolic type over $C$.
		\end{lemma}
		
		\begin{proof}
		Suppose that $\rho$ is of parabolic type over $C$.
		This case is almost completely the same as the hyperbolic case (Lemma \ref{hyp_reduction_vc}).
		Note that if we set $\eta_t=g_t\eta g_t^{-1}$ in the conclusion of Lemma \ref{ab_rep_par},
		then $\eta_t$ satisfies exactly the same properties as listed in Lemma \ref{ab_rep_hyp}.
		Therefore, we simply repeat the argument of the hyperbolic case,
		applying Lemmas \ref{ab_rep_par} and \ref{comm_par}
		instead of Lemmas \ref{ab_rep_hyp} and \ref{comm_hyp}.
		Then we obtain a path of representations 
		$\rho^\dagger_t\colon \pi_1(M^\dagger)\to \Sft\times_\Integral\Real$,
		parametrized by $t\in[0,1]$, which equals $\rho_{\mathtt{pb}}$ at $t=1$,
		and which is central and abelian over $C$ at $t=0$.
		Again $\rho'$ can be taken as $\rho^\dagger_0$ as desired.
		\end{proof}
		
		\subsubsection{Local reduction in general}
		We finish the proof Theorem \ref{local_reduction_vc} as follows.
		Let $M$ be any formatted graph manifold with a simplicial JSJ graph $(V,E)$,
		and $\rho\colon \pi_1(M)\to\Sft\times_\Integral\Real$ any representation.
		Let $C=(V_C,E_C)$ be a maximal connected subgraph of $(V,E)$ 
		over which $\rho$ is noncentral abelian type.
				
		Suppose that $M'\to M$ is a formatted pinching map of a formatted graph manifold $M'$,
		such that the inequality
		$$\chi_{M'}(v)\leq\chi_M(v)-2\cdot\delta_\rho(v;C)$$
		holds for all $v\in V$.
		It follows from (\ref{standard_equal})
		that $\chi_{M'}(v)\leq\chi_{M^\dagger}(v)$ holds for all $v\in V$.
		Therefore,
		$M'\to M$ admits a factorization into a composition of formatted pinching maps
		$M'\to M^\dagger\to M$.
		Given any constant $\epsilon>0$,
		there is some representation $\rho^\dagger$ of $\pi_1(M^\dagger)$
		that satisfies the asserted properties for $M^\dagger\to M$,
		(Lemmas \ref{ell_reduction_vc}, \ref{hyp_reduction_vc}, and \ref{par_reduction_vc}).
		We take $\rho'$ to be the pull-back of $\rho^\dagger$ to $\pi_1(M')$.
		Then $\rho'$ also satisfies the properties, as desired.
		
		This completes the proof of Theorem \ref{local_reduction_vc}.
	
	\subsection{Global reduction}
	
	\begin{theorem}\label{reduction_vc}
		Let $M'\to M$ be a formatted pinching map between 
		formatted graph manifolds
		with identified simplicial JSJ graphs $(V,E)$.
		Adopt Notations \ref{notation_xi_tau} and \ref{notation_delta}.
		Suppose that the inequality
		$$\chi_{M'}(v)\leq\chi_{M}(v)-2\cdot\delta_\rho(v)$$
		holds for all vertices $v\in V$.
		
		Then, 
		for and any constant $\epsilon>0$ and
		any representation $\rho\colon \pi_1(M)\to\Sft\times_{\Integral}\Real$,
		there exists a representation
		$\rho'\colon \pi_1(M')\to\Sft\times_{\Integral}\Real$,
		such that the following conditions are all satisfied:
		\begin{itemize}
		\item 
		In the space of representations $\mathcal{R}(\pi_1(M'),\Sft\times_\Integral\Real)$,
		$\rho'$ is path-connected with the pull-back representation of $\rho$.
		\item
		For all $v\in V$,
		$$\left|\xi_{\rho'}(v)-\xi_{\rho}(v)\right|<\epsilon.$$
		\item 
		For all $\{v,w\}\in E$,
		$\tau_{\rho'}(v,w)$	equals $\tau_\rho(v,w)$ if $\tau_\rho(v,w)$ is finite.
		\item
		For all $v\in V$, $\delta_{\rho'}(v)$ equals $0$.
		Hence $\tau_{\rho'}(v,w)$ is finite for all $\{v,w\}\in E$.
		\end{itemize}
	\end{theorem}

	\begin{proof}
		Set $(M'_0,\rho'_0)=(M,\rho)$.
		Given any constant $\epsilon>0$,
		we construct $(M'_n,\rho'_n)$ and $M'_n\to M'_{n-1}$ inductively, 
		for all $n\in\Natural$ applicable, as follows.
				
		Suppose that $(M'_{n-1},\rho'_{n-1})$ has been constructed,
		where $M'_{n-1}$ is a formatted graph manifold with a simplicial JSJ graph identified with $(V,E)$,
		and where
		$\rho'_{n-1}\colon\pi_1(M'_{n-1})\to \Sft\times_\Integral\Real$
		is a representation with $\tau_{\rho'_{n-1}}(v_{n-1},w_{n-1})=\infty$ for some $\{v_{n-1},w_{n-1}\}\in E$.
		Take a maximal connected subgraphs $C_{n-1}=(V_{n-1},E_{n-1})$ of $(V,E)$ with $\{v_{n-1},w_{n-1}\}\in E_{n-1}$,
		such that $\rho'_{n-1}$ is of noncentral abelian type on $C_{n-1}$.
		Construct a formatted graph manifold $M'_n$ and a formatted pinching map
		$M'_n\to M'_{n-1}$, such that
		$$\chi_{M'_n}(v)=\chi_{M'_{n-1}}(v)-2\cdot\delta_\rho(v;C_{n-1}),$$
		for all $v\in V$.
		For example, $M'_n$ can be obtained as a suitable simultaneous fiber connected sum of $M'_{n-1}$,
		(Remark \ref{remark_formatted_pinching}).
		Apply Theorem \ref{local_reduction_vc}
		to construct a representation $\rho'_n$ of $\pi_1(M'_n)$ with $\tau_{\rho'_n}(v_{n-1},w_{n-1})<\infty$.
		Moreover, $\rho'_n$ lies
		in the path-connected component of the pull-back of $\rho'_{n-1}$.
		The construction guarantees
		$|\xi_{\rho'_n}(v)-\xi_{\rho'_{n-1}}(v)|<2^{-n}\epsilon$ for all $v\in V$.
		It also guarantees
		$\tau_{\rho'_n}(v,w)=\tau_{\rho'_{n-1}}(v,w)$
		for all $\{v,w\}\in E$ 
		unless $\{v,w\}\cap V_{n-1}\neq\emptyset$ and $\tau_\rho(v,w)=\infty$.
		
		The above construction terminates at some step $n$ with $\tau_{\rho'_n}(v,w)$ finite
		for all $\{v,w\}\in E$. For this $M'_n$, 
		we obtain a formatted pinching map $M'_n\to M$ 
		as the composite map $M'_n\to M'_{n-1}\to\cdots \to M'_0$.
		Observe
		$$\chi_{M'_n}(v)=\chi_M(v)-2\cdot\mathrm{valence}_{(V,E)}(v)$$
		for all $v\in V$.
		We estimate
		$$\left|\xi_{\rho'_n}(v)-\xi_\rho(v)\right|
		\leq\sum_{l=1}^{n} \left|\xi_{\rho'_l}(v)-\xi_{\rho'_{l-1}}(v)\right|
		<\sum_{l=1}^{n} 2^{-l}\epsilon
		<\epsilon.$$
		Moreover, any finite $\tau_{\rho}(v,w)$ remains unchanged throughout the construction,
		so $\tau_{\rho'_n}(v,w)$ equals $\tau_\rho(v,w)$ if $\tau_\rho(v,w)$ is finite.
		
		By the assumption about $M'$,
		we have $\chi_{M'}(v)\leq \chi_{M'_n}(v)$ for all $v\in V$,
		so the formatted pinching map $M'\to M$ factorizes as a composition of formatted pinching maps
		$M'\to M'_n\to M$.
		Take $\rho'$ to be the pull-back of $\rho'_n$ to $\pi_1(M')$.
		It follows that $\rho'$ satisfies the properties as asserted.
		\end{proof}
		
\section{Seifert representations for graph manifolds: the proof}\label{Sec-proof_volume_general_GM}
	In this section, we prove Theorem \ref{volume_general_GM}.
	Let $M$ be a formatted graph manifold with a simplicial JSJ graph $(V,E)$,
	and $\rho\colon\pi_1(M)\to \Sft\times_\Integral\Real$ be a representation. 
	
	For each vertex $v\in V$, 
	we take $\delta_\rho(v)$ mutually disjoint ordinary fibers
	in the corresponding JSJ piece $J_v$,
	(see Notation \ref{notation_delta}).
	Construct a formatted graph manifold $M'$
	as the simultaneous fiber connected sum of $M$ with trivial circle bundles over tori,
	with respect to the union of all these fibers, (see Remark \ref{remark_formatted_pinching}).
	Denote by $M'$ the resulting formatted graph manifold and
	$M'\to M$ the associated formatted pinching map,
	namely, the simultaneous horizontally pinching of $M'$ onto $M$.
	We obtain the relation
	$$\chi_{M'}(v)=\chi_{M}(v)-2\cdot\delta_\rho(v)$$
	for all $v\in V$.
	Moreover, we obtain the relation 
	$$\opEuler_{M'}=\opEuler_M$$
	in $\mathrm{End}(\Real^V)$.
	These obvious relations are the simultaneous version of (\ref{pinch_e_chi}).
	
	We apply Theorem \ref{reduction_vc}
	to obtain a sequence representation $\rho'_n$ of $\pi_1(M')$, index by $n\in\Natural$,
	which are all path-connected to the pull-back representation of $\rho$,
	in the space of representations
	$\mathcal{R}(\pi_1(M'),\Sft\times_\Integral\Real)$.
	We may require
	$$\lim_{n\to\infty} \xi_{\rho'_n}(v)=\xi_\rho(v)$$
	for all $v\in V$
	and 
	$$\tau_{\rho'_n}(v,w)=\tau_\rho(v,w)$$ for all $\{v,w\}\in E$
	with $\tau_\rho(v,w)$ finite.
	Moreover,
	we may require $\tau_{\rho'_n}(v,w)$ to be finite for all $\{v,w\}\in E$ and all $n\in\Natural$.
	
	Since $\rho'_n$ are all path-connected with the pull-back of $\rho$,
	it follows that
	$$\mathrm{vol}_{\Sft\times_\Integral\Real}\left(M,\rho\right)=
	\mathrm{vol}_{\Sft\times_\Integral\Real}\left(M',\rho'_n\right).$$
	By Theorem \ref{volume_virtually_central_GM},
	we obtain
	$$\mathrm{vol}_{\Sft\times_\Integral\Real}\left(M',\rho'_n\right)=
	4\pi^2\cdot\left(\xi_{\rho'_n},\opEuler_{M'}\xi_{\rho'_n}\right).$$
	Passing to the limit as $n$ tends to $\infty$, 
	we obtain the volume formula
	$$\mathrm{vol}_{\Sft\times_\Integral\Real}(M,\rho)=
	4\pi^2\cdot\left(\xi_{\rho},\opEuler_M\xi_{\rho}\right),$$
	as asserted in Theorem \ref{volume_general_GM}.
	
	Moreover, the asserted rationality in Theorem \ref{volume_general_GM} about $(M,\rho)$ 
	follows from the rationality about $(M',\rho')$.
	After all, 
	neither the Euler operator nor the volume has changed.
	
	It remains to establish the asserted upper bound of 
	$|(\opEuler_M\xi_\rho)(v)|$ for all $v\in V$.
	At this point, 
	we notice that the estimate for $|(\opEuler_{M'}\xi_{\rho'_n})(v)|$
	from Theorem \ref{volume_virtually_central_GM} 
	only implies a weaker upper bound:
	\begin{equation}\label{estimate_with_delta}
	|(\opEuler_{M}\xi_{\rho})(v)|
	\leq \max\left\{0,-\chi_M(v)+2\cdot\delta_\rho(v)-\sum_{\{v,w\}\in E}\frac{1}{\tau_{\rho}(v,w)}\right\}.
	\end{equation}
	To get rid of the unwanted term with $\delta_\rho(v)$,
	we again appeal to the covering trick, as we did in Subsection \ref{Subsec-covering_trick}.
	This is done as follows.
	
	Denote by $\bar\rho\colon \pi_1(M)\to \mathrm{PSL}(2,\Real)$
	the induced representation of $\rho$.
	Since $\bar\rho(\pi_1(M))$ is a finitely generated subgroup of a linear group,
	it is residually finite.
	Recall that $\tau_\rho(v,w)$ is the order of $\bar{\rho}(\pi_1(T_{v,w}))$,
	for any $\{v,w\}\in E$.
	Therefore, for any positive integer $D\in\Natural$, 
	we can construct some quotient homomorphism 
	$\bar\rho(\pi_1(M))\to \Gamma$ onto a finite group $\Gamma$ (depending on $D$),
	with the following properties:
	For any $\{v,w\}\in E$, if $\tau_\rho(v,w)$ is finite, 
	then $\bar{\rho}(\pi_1(T_{v,w}))$ injects $\Gamma$;
	otherwise, $\bar{\rho}(\pi_1(T_{v,w}))$ 
	surjects a subgroup of $\Gamma$	of order at least $D$.
	Take $M^*$ to be the regular finite cover of $M$ that corresponds
	to the kernel of the composite homomorphism
	$\pi_1(M)\to \bar\rho(\pi_1(M))\to \Gamma$.
	Take a characteristic finite cover of the JSJ graph of $M^*$
	which is a simplicial graph.
	Take $M''$ to be the pull-back cover of $M^*$.
	Then $M''$ is a characteristic finite cover of $M^*$,
	and therefore, a regular finite cover of $M$.
	We furnish $M''$ with the formatted graph manifold structure
	that lifts from $M$.
	
	The construction makes sure that 
	$M''\to M$
	is a formatted covering projection between
	formatted graph manifolds.	
	Denote by $(V'',E'')$ the JSJ graph of $M''$.
	Denote by $\rho''\colon\pi_1(M'')\to\Sft\times_\Integral\Real$ 
	the pull-back representation of $\rho$.
	For any covering pair of JSJ pieces $J''_{v''}\to J_v$,
	the same computations as in Subsection \ref{Subsec-covering_trick} work,
	yielding the formula
	\begin{equation}\label{tmp_delta_LHS}
	\left|\left(\opEuler_{M''}\xi_{\rho''}\right)(v'')\right|=
	\frac{[J''_{v''}:J_{v}]}{[f''_{v''}:f_{v}]}\times 
	\left|\left(\opEuler_{M}\xi_{\rho}\right)(v)\right|,
	\end{equation}
	and the comparison
	\begin{equation}\label{tmp_delta_RHSa}
	-\chi_{M''}(v'')-\sum_{\{v'',w''\}\in E''}\frac{1}{\tau_{\rho''}(v'',w'')}
	\,\leq\,
	\frac{[J''_{v''}:J_v]}{[f''_{v''}:f_v]}\times \left(-\chi_M(v)-\sum_{\{v,w\}\in E}\frac{1}{\tau_{\rho}(v,w)}\right).
	\end{equation}
	Note that terms with $\tau_{\rho'}(v',w')=\tau_{\rho}(v,w)=\infty$ have no contribution on both sides of the inequality.
	
	To compare $\delta_{\rho''}(v'')$ and $\delta_\rho(v)$, we divide into two cases according to the type of $\rho$ at $v$.
	If $\rho$ is abelian at $v$, $\rho''$ must also be abelian at $v''$,
	then $\delta_{\rho''}(v'')=\delta_{\rho}(v)=0$.
	If $\rho$ is nonabelian at $v$, 
	then for any JSJ torus $T_{v,w}$ of $M$ with $\tau_{\rho}(v,w)=\infty$,
	the number of JSJ tori $T''_{v'',w''}$ in $M''$ that cover $T_{v,w}$
	is at most $[J''_{v''}:J_v]/[f''_{v''}:f_v]$ divided by $D$.
	In fact, the precise number should be $[J''_{v''}:J_v]/[T_{v'',w''}:T_{v,w}]$,
	but we observe $[T_{v'',w''}:T_{v,w}]\geq D$ from the construction,
	and observe $[f''_{v''}:f_v]=1$, 
	since $\bar\rho(f_v)$ has to be trivial	in the nonabelian case.
	In both cases, we reach the following comparison:
	\begin{equation}\label{tmp_delta_RHSb}
	\delta_{\rho''}(v'')\leq \frac{[J''_{v''}:J_v]}{[f''_{v''}:f_v]}\times \frac{1}{D}\times \delta_{\rho}(v).
	\end{equation}
	
	We apply the weaker upper bound (\ref{estimate_with_delta}) to $(M'',\rho'')$. Then:
	$$
	\left|(\opEuler_{M''}\xi_{\rho''})(v'')\right|
	\leq \max\left\{0,-\chi_{M''}(v'')+2\cdot\delta_{\rho''}(v'')-\sum_{\{v'',w''\}\in E''}\frac{1}{\tau_{\rho''}(v'',w'')}\right\}.$$
	Together with (\ref{tmp_delta_LHS}), (\ref{tmp_delta_RHSa}), and (\ref{tmp_delta_RHSb}), 
	this yields:
	$$|(\opEuler_{M}\xi_{\rho})(v)|
	\leq \max\left\{0,-\chi_{M}(v)+\frac{2\cdot\delta_{\rho}(v)}{D}-\sum_{\{v,w\}\in E}\frac{1}{\tau_{\rho}(v,w)}\right\}.$$
	Note that $D$ is an arbitrary positive integer, and $\delta_{\rho}(v)$ is at most the number of edges in $(V,E)$.
	Therefore, we obtain the upper bound
	$$|(\opEuler_{M}\xi_{\rho})(v)|
	\leq \max\left\{0,-\chi_{M}(v)-\sum_{\{v,w\}\in E}\frac{1}{\tau_{\rho}(v,w)}\right\},$$
	as desired.
	
	This completes the proof of Theorem \ref{volume_general_GM}.

\section{Application to strictly diagonally dominant graph manifolds}\label{Sec-strictly_diagonally_dominant}
	As our first application of Theorem \ref{volume_general_GM},
	we exhibit a class of graph manifolds 
	whose covering Seifert volume (see (\ref{def_CSV})) can be effectively bounded.
	Example \ref{CSV_Mgabcd} follows immediately as a special case.
	
	\begin{theorem}\label{bound_diagonally_dominant}
		Let $M$ be a formatted graph manifold with a simplicial JSJ graph $(V,E)$.
		Adopt Notation \ref{notation_xkb}.
		Suppose that the Euler operator $\opEuler_M$ is strictly diagonally dominant,
		namely, the inequality
			$$|k_v|> \sum_{\{v,w\}\in E} |1/b_{v,w}|$$
		holds for all vertices $v\in V$.
		Then 
		$$\mathrm{CSV}(M)\leq\sum_{v\in V} \frac{4\pi^2\chi_v^2}{|k_v|-\sum_{\{v,w\}\in E}\left|1/b_{v,w}\right|}.$$
	\end{theorem}
	
	To prove Theorem \ref{bound_diagonally_dominant},
	we make use of the well-known Gershgorin circle theorem in matrix analysis:
	
	\begin{theorem}[{\cite[Chapter 6, Theorem 6.1.1]{HJ-book}}]\label{Gershgorin_circle_theorem}
		Let $A=(a_{i,j})_{n\times n}$ be a square matrix of size $n$ with entries in $\Complex$.
		Then the eigenvalues of $A$ all lie in the union of disks $D_1\cup D_2 \cup \cdots\cup D_n$,
		where 
		$$D_i=\left\{z\in\Complex\colon |z-a_{i,i}|\leq \sum_{j\neq i} |a_{i,j}|\right\},$$
		for $i=1,2,\cdots,n$.
	\end{theorem}
	
	\begin{lemma}\label{weak_bound_diagonally_dominant}
		Under the assumptions of Theorem \ref{bound_diagonally_dominant},
		the Seifert volume of $M$ satisfies the inequality
		$$\mathrm{SV}(M)\leq 
		\frac{4\pi^2\cdot \sum_{v\in V} \chi_v^2}{\min\left\{|k_v|-\sum_{\{v,w\}\in E}\left|1/b_{v,w}\right|\colon v\in V\right\}}.$$		
	\end{lemma}
	
	\begin{proof}
		The operator $\opEuler_M$ has a symmetric square matrix over the standard basis of $\Real^V$.
		Theorem \ref{Gershgorin_circle_theorem} implies that its eigenvalues $\lambda$ are all nonzero, and indeed,
		$$\left|\lambda\right|\geq \min\left\{|k_v|-\sum_{\{v,w\}\in E}\left|1/b_{v,w}\right|\colon v\in V\right\}.$$
		In particular, $\opEuler_M$ is invertible.
		Note that the operator norm of $\opEuler_M^{-1}$ on the inner product vector space $\Real^V$
		is the largest $|1/\lambda|$.
		For any vector $\xi\in\Real^V$ with $|(\opEuler_M\xi)(v)|\leq -\chi_v$ for all $v\in V$,
		we estimate
		$$\left(\xi,\opEuler_M\xi\right)\leq \left\|\opEuler_M^{-1}\right\|_{\texttt{op}}\times \left\|\opEuler_M\xi\right\|^2_{\ell^2}
		\leq\frac{\sum_{v\in V}\chi_v^2}{\min\left\{|k_v|-\sum_{\{v,w\}\in E}\left|1/b_{v,w}\right|\colon v\in V\right\}}.$$
		Then by Theorem \ref{volume_general_GM} we obtain the asserted inequality.
	\end{proof}
	
	\begin{proof}[{Proof of Theorem \ref{bound_diagonally_dominant}}]
	Suppose that $M^*\to M$ is any given finite cover of $M$.
	Then Lemma \ref{degree_typed_cover} applies and yields some positive integers $m^*$ and $r^*$.
	Take $m$ to be $m^*$, and for any vertex $v\in V$, take $r_v$ to be some positive integral multiple of $r^*$,
	such that the product
	$$C=r_v\times\left(|k_v|-\sum_{\{v,w\}\in E}\left|1/b_{v,w}\right|\right)$$
	becomes constant independent of $v$.
	Let $M'$ be a formatted cover of $M$ as guaranteed in Lemma \ref{degree_typed_cover},
	with respect to $m$ and these $r_v$.
	We observe
	$$|k_{v'}|-\sum_{\{v',w'\}\in E'}\left|1/b_{v',w'}\right|
	=r_v\times\left(|k_{v}|-\sum_{\{v,w\}\in E}|1/b_{v,w}|\right)=C,$$
	for any covering pair of JSJ pieces $J'_{v'}\to J_v$,
	by Proposition \ref{virtual_xkb}.
	Lemma \ref{weak_bound_diagonally_dominant} applies to $M'$, and yields
	\begin{eqnarray*}
	\mathrm{SV}(M')&\leq& \frac{4\pi^2\cdot \sum_{v'\in V'} \chi_{v'}^2}{C}\\
	&=&\sum_{v\in V} \frac{[M':M]}{m^2r_v}\times \frac{m^2r_v^2\times 4\pi^2\chi_v^2}{r_v\times \left(|k_{v}|-\sum_{\{v,w\}\in E}|1/b_{v,w}|\right)}\\
	&=&[M':M]\times \sum_{v\in V} \frac{4\pi^2\chi_v^2}{|k_v|-\sum_{\{v,w\}\in E}\left|1/b_{v,w}\right|}.
	\end{eqnarray*}
	Since $M'\to M$ factors through the given cover $M^*$, we obtain
	$$
	\frac{\mathrm{SV}(M^*)}{[M^*:M]}\leq \frac{\mathrm{SV}(M')}{[M':M]}
	\leq\sum_{v\in V} \frac{4\pi^2\chi_v^2}{|k_v|-\sum_{\{v,w\}\in E}\left|1/b_{v,w}\right|}.$$
	Since the finite cover $M^*\to M$ is arbitrary,
	we obtain the estimate for $\mathrm{CSV}(M)$
	as asserted in Theorem \ref{bound_diagonally_dominant}.
	\end{proof}

\section{Virtual existence of generic volume values}\label{Sec-virtual_existence}
	Provided Theorem \ref{volume_general_GM}, 
	it becomes suitable to ask which values in $[0,+\infty)$ actually arise
	as volume of $\Sft\times_\Integral\Real$--representations for a given graph manifold.
	In this section,
	we offer a virtual and generic answer to this question,
	as Theorem \ref{volume_existence_general_GM}.
	It is virtual because of the finite cover in the conclusion.
	It is generic because of the strict inequality in the hypothesis.

	\begin{theorem}\label{volume_existence_general_GM}
		Let $M$ be a formatted graph manifold 
		with a simplicial JSJ graph $(V,E)$.
		Adopt Notation \ref{notation_xi_tau}.
		Let $\xi\in\Rational^V$ be any vector such that 
		the following inequality holds for all $v\in V$:
		$$|(\opEuler_M\xi)(v)|<-\chi_M(v).$$
		Then, 
		there exist a formatted finite cover $M'\to M$ 
		and a representation
		$\rho'\colon\pi_1(M')\to \Sft\times_\Integral\Real$
		with the property
		$$\frac{\mathrm{vol}_{\Sft\times_\Integral\Real}\left(M',\rho'\right)}{[M':M]}=4\pi^2\cdot(\xi,\opEuler_M\xi).$$
	\end{theorem}
	
	The rest of this section is devoted to the proof of Theorem \ref{volume_existence_general_GM}.

	\begin{lemma}\label{volume_existence_special}
		Let $M$ be a formatted graph manifold 
		with a simplicial JSJ graph $(V,E)$.
		Adopt Notation \ref{notation_xi_tau}.
		Suppose that there are no exceptional fibers in any JSJ piece of $M$.
		Suppose that $\xi\in\Integral^V$ is a vector,
		such that 
		$b_{v,w}$ divides $\xi(v)$ for all $v\in V$ and $\{v,w\}\in E$,
		and such that
		$$|(\opEuler_M\xi)(v)|\leq-\chi_M(v)-\mathrm{valence}_{(V,E)}(v)$$
		holds for all $v\in V$.
		Then, there exists a representation
		$\rho\colon\pi_1(M)\to \Sft$,
		such that $\rho(f_v)$ lies in the center $\Integral$ and equals $\xi(v)$,
		for all $v\in V$.
	\end{lemma}
	
	\begin{proof}
		We choose auxiliary basepoints of the JSJ pieces and the JSJ tori, 
		and choose auxiliary paths to connect them,
		so the fundamental group $\pi_1(M)$ 
		decomposes into a graph of groups accordingly.
		We construct representations $\pi_1(J_v)\to \Sft$ for all $v\in V$
		such that the restricted representations to $\pi_1(T_{v,w})$ has image in the center $\Integral$.
		We make sure that the restricted representations $\pi_1(T_{v,w})\to\Integral$
		and $\pi_1(T_{w,v})\to \Integral$ are equal for each $\{v,w\}\in E$.
		Then these representations extends to a representation of $\pi_1(M)$,
		by any arbitrary assignments for the stable letters.
				
		Since there are no exceptional fibers in any JSJ piece $J_v$ of $M$,
		we can identify $\mathrm{clos}(J_v)$ as a product $f_v\times\Sigma_v$,
		where $\Sigma_v$ is a compact oriented surface.
		The components of $\partial \Sigma_v$ naturally correspond to the edges $\{v,w\}$
		incident to $v$. With the induced orientations, they give rise to oriented slopes $s_{v,w}$ on 
		the corresponding JSJ tori $T_{v,w}$.
		Therefore, we obtain a Waldhausen basis $[f_v],[s_{v,w}]$ for $H_1(T_{v,w};\Real)$,
		which actually generates the integral lattice $H_1(T_{v,w};\Integral)$.
		The equation $[f_w]=a_{v,w}[f_v]+b_{v,w}[s_{v,w}]$ on $H_1(T_{v,w};\Integral)$
		determines a unique integer $a_{v,w}\in\Integral$.
		We obtain a unique homomorphism $\eta_{v,w}\colon H_1(T_{v,w};\Integral)\to \Integral$
		which sends $[f_v]$ to $\xi(v)$ and $[f_w]$ to $\xi(w)$.
		In fact, $[s_{v,w}]$ must go to $(\xi(w)-a_{v,w}\xi(v))/b_{v,w}$,
		which lies in $\Integral$ because $b_{v,w}$ divides both $\xi(v)$ and $\xi(w)$.
		We observe the relation
		$$(\opEuler_M\xi)(v)=k_v\xi(v)-\sum_{\{v,w\}\in E}\frac{1}{b_{v,w}}\times\xi(w)
		=(-1)\cdot\sum_{\{v,w\}\in E}\eta_{v,w}\left([s_{v,w}]\right).$$
		
		To simplify notations, 
		we focus on a vertex $v$,
		and suppose that $\Sigma_v$ is of genus $g$ with $n$ boundary components.
		We enumerate the vertices incident to $v$ as $w_1,w_2,\cdots,w_n$.
		The fundamental group $\pi_1(\Sigma_v)$ admits a well-known presentation of $2g+n$ generators
		$x_1,y_1,x_2,y_2,\cdots,x_g,y_g,s_1,\cdots,s_n$,
		with a single relation
		$$[x_1,y_1][x_2,y_2]\cdots[x_g,y_g]=s_1s_2\cdots s_n,$$
		where $[x_i,y_i]$ stands for the commutator $x_iy_ix_i^{-1}y_i^{-1}$.
		Then $\pi_1(J_v)$ is generated as a direct product of $\pi_1(\Sigma_v)$ and 
		the infinite cyclic group $\langle f_v\rangle$.
		We assume that the conjugacy class of $s_j$ corresponds to the free homotopy class
		of the oriented slope $s_{v,w_j}$. 
		To define a homomorphism $\rho_v\colon \pi_1(J_v)\to \Sft$,
		we first assign $\rho_v(f_v)$ to be the central element $\xi(v)\in\Integral$,
		and $\rho_v(s_j)$ to be $\eta_{v,w_j}([s_{v,w_j}])\in\Integral$,
		for $j=1,2,\cdots,n$.
		Therefore, the assumption just says
		$$\left|\rho_v(s_1s_2\cdots s_n)\right|\leq 2g-2.$$
		In this case, $\rho_v(s_1s_2\cdots s_n)$ can be written as a product of $g$
		commutators $[\alpha_1,\beta_1][\alpha_2,\beta_2]\cdots[\alpha_g,\beta_g]$,
		where $\alpha_1,\beta_1,\alpha_2,\beta_2,\cdots,\alpha_g,\beta_g$ are some elements of $\Sft$.
		In fact, this is a special case of \cite[Theorems 2.3 and 4.1]{EHN}.
		Alternatively, 
		one may apply Theorem \ref{volume_central_bundle} to 
		an oriented circle bundles $N$ 
		over closed oriented surfaces of genus $g$ with Euler number $1$:
		Since $\pi_1(N)$ can be generated with $2g+1$ generators $X_1,Y_1,X_2,Y_2,\cdots,X_g,Y_g,Z$,
		where $Z=[X_1,Y_1][X_2,Y_2]\cdots[X_g,Y_g]$ is central,
		$H^2(N;\Integral)$ has no torsion,
		and there is a representation $\phi\colon \pi_1(N)\to\Sft$ which sends $f$ to 
		our $\rho_v(s_1s_2\cdots s_n)\in\Integral$.
		Hence $\alpha_i$ and $\beta_i$ can be taken as $\phi(X_i)$ and $\phi(Y_i)$, respectively.
		(Yet another manual proof can be derived from
		the explicit constructions in Lemmas \ref{comm_ell}, \ref{comm_hyp}, and \ref{comm_par}.)
		Anyways, taking a commutator factorization of $\rho(s_1s_2\cdots s_n)$ of length $g$ as above,
		we assign $\rho_v(x_i)=\alpha_i$ and $\rho_v(y_i)=\beta_i$.
		This defines a homomorphism $\rho_v$ of $\pi_1(J_v)$,
		whose restriction to the subgroup $\pi_1(T_{v,w_j})=\langle s_j,f_v\rangle$
		is the pull-back of $\eta_{v,w_j}$ via the abelianization isomorphism 
		$\pi_1(T_{v,w_j})\cong H_1(T_{v,w_j};\Integral)$.
		
		Perform the above construction for all $v\in V$,
		and obtain representations $\rho_v\colon\pi_1(J_v)\to \Sft$.
		Then the restricted representations of $\rho_v$ and $\rho_w$ 
		to $\pi_1(T_{v,w})=\pi_1(T_{w,v})$ are obviously the same for any $\{v,w\}\in E$.
		As explained, we can assemble these $\rho_v$ and obtain a representation
		$\rho\colon\pi_1(M)\to \Sft$ as desired.
	\end{proof}

	\begin{lemma}\label{degree_typed_cover}
		Let $M$ be a formatted graph manifold with a simplicial JSJ graph $(V,E)$.
		Adopt Notation \ref{notation_xi_tau}.
		Suppose $\chi_M(v)<0$ holds for all $v\in V$.
		Then, for
		any formatted covering projection $M^*\to M$,
		there exist positive integers $m^*$ and $r^*$ with the following property:
		
		For any positive integer $m$ divisible by $m^*$, and for any positive integers $r_v$ divisible by $r^*$,
		indexed by $v\in V$,
		there exist
		a formatted graph manifold $M'$, and
		a formatted covering projection $M'\to M$,
		such that the following conditions are all satisfied:
		\begin{itemize}
		\item Every JSJ torus $T'_{v',w'}$ of $M'$ projects a JSJ torus $T_{v,w}$ of $M$
		as a characteristic cover of degree $m^2$.
		\item Every JSJ piece $J'_{v'}$ of $M'$ projects a JSJ piece $J_v$ of $M$
		as a cover of degree $m^2r_v$.
		\item There are no exceptional fibers in any JSJ pieces of $M'$.
		\item The projection $M'\to M$ factors through $M^*$.		
		\end{itemize}
	\end{lemma}
	
	\begin{proof}
		Since every JSJ piece $J_v$ of $M$ fibers over a $2$--orbifold of negative Euler characteristic,
		we may replace $M^*$ with some finite cover,
		and assume that every JSJ piece of $M^*$ 
		is homeomorphic to the product of a circle and a surface of positive genus 
		and with an even number of punctures.
		For example, this can be done immediately by \cite[Proposition 4.2]{DLW-cs_sep_vol}, 
		(applied to $M^*$ and a collection of product covers of the JSJ pieces as described).
		Moreover, we may replace $M^*$ with some finite cover,
		and assume that $M^*$ is regular over $M$,
		and that the JSJ tori of $M^*$ are all characteristic over 
		JSJ tori of $M$ of the equal degree.
		For example, this is a special case of \cite[Corollary 4.5]{DLW-cs_sep_vol},
		(applied to $M^*\to M$ as $N\to M$ thereof).
		
		With these simplification assumptions,
		we take $m^*$ and $r^*$ according to the following conditions:
		The product $m^*\times m^*$ equals the covering degree 
		$[T^*_{v^*,w^*}:T_{v,w}]$, for some (hence any) covering pair of JSJ tori in $M^*$ and $M$;
		and the product $m^*\times m^*r^*$ equals the least common multiple
		of $[J^*_{v^*}:J_v]$, ranging over all the covering pairs of JSJ pieces in $M^*$ and $M$.
		
		For any positive integers $m$ an $l$, we observe the following simple constructions:
		Given any orientable compact surface $\Sigma$ of positive genus
		with an even number of boundary components,
		we can always construct a finite cover $\Sigma'$ of $\Sigma$,
		such that every component of $\partial \Sigma'$ covers a component $\partial \Sigma$
		of degree $l$, and $\Sigma'$ covers $\Sigma$ of degree $s$.
		For example, 
		one may first take an $l$--cyclic cover of $\Sigma$
		dual to a union of disjoint simple arcs that intersects every component of $\partial \Sigma$
		in exactly one point,
		and then take an $s$--cyclic cover of that cover
		dual to a nonseparating simple closed curve.
		It follows that the product of a circle with $\Sigma$
		always admits a finite cover of degree $l^2s$,
		such that the boundary components projects as characteristic covers of degree $l^2$.
		
		Suppose that $m$ and $r_v$, indexed by $v\in V$,
		are any given positive integers.
		For each JSJ piece $J^*_{v^*}$ of $M^*$,
		the above construction
		yields a finite cover $\mathrm{clos}(J'_{v'})\to \mathrm{clos}(J^*_{v^*})$,
		which has degree $(r_v/r^*)\times[J^*_{v^*}:J_v]$,
		and which is characteristic of degree $(m/m^*)^2$ restricted to each boundary component.
		Take $D$ to be a common multiple for all $[J'_{v'}:J^*_{v^*}]$.
		Take $D/[J'_{v'}:J^*_{v^*}]$ copies of each $\mathrm{clos}(J'_{v'})$.
		We may obtain a finite cover $M'$ over $M^*$ by gluing up these copies,
		identifying the common boundary components via covering transformations over 
		the JSJ tori of $M^*$.
		Possibly after discarding extra connected components, we may require that $M'$ is connected.
		Possibly after passing to a pull-back cover induced by a cover of the JSJ graph, 
		we may require that $M'$ has a simplicial JSJ graph.
		It is straighforward to verify 
		that the composite covering projection $M'\to M^*\to M$ is as desired.
	\end{proof}
		
	We are ready to prove Theorem \ref{volume_existence_general_GM} 
	using Lemmas \ref{volume_existence_special} and \ref{degree_typed_cover}.
	
	Let $M$ be a formatted graph manifold with a simplicial graph $(V,E)$.
	Note that the assumption in Theorem \ref{volume_existence_general_GM} about $\xi\in\Rational^V$ 
	implies $\chi_M(v)<0$ for all $v\in V$.
	Let $m^*$ and $r^*$ be the positive integers as provided by Lemma \ref{degree_typed_cover},
	with respect to the trivial cover $M^*=M$.
	For all $v\in V$, we simply take $r_v$ to be $r^*$. 
	On the other hand,
	we take $m$ to be a sufficiently large positive integral multiple of $m^*$.
	To be precise,
	we require 
	$$m\cdot\xi(v)/b_{v,w}\in\Integral$$
	for all $v\in V$ and $\{v,w\}\in E$,
	and moreover, 
	we require the inequality
	$$\left|(\opEuler_M\xi)(v)\right|\leq -\chi_M(v)-\frac{1}{m}\times\mathrm{valence}_{(V,E)}(v)$$
	for all $v\in V$.
	By Lemma \ref{degree_typed_cover},
	there is a finite cover $M'\to M$ satisfying the listed conditions thereof.
	
	In particular, 
	we have $[J'_{v'}:J_v]=m^2r^*$ and $[f'_{v'}:f_v]=m$ 
	for any covering pair of JSJ pieces $J'_{v'}\to J_v$,
	and $[T'_{v',w'}:T_{v,w}]=m^2$ for
	for any covering pairs of JSJ tori $T'_{v',w'}\to T_{v,w}$.
	So we obtain
	$$\begin{array}{ccc}
	\chi_{M'}(v')=mr^*\times\chi_{M}(v), &
	k_{v'}=r^*\times k_v, &
	\mathrm{valence}_{(V',E')}(v')=r^*\times\mathrm{valence}_{(V,E)},
	\end{array}$$	
	and
	$$b_{v',w'}=b_{v,w},$$
	(see Proposition \ref{virtual_xkb}).
	
	Let $\xi'\in\Integral^{V'}$ be the vector defined as
	$$\xi'(v')=m\times \xi(v)$$
	for any $v'\in V'$, lying over a vertex $v\in V$.
	Hence
	$$k_{v'}\xi'(v')-\sum_{\{v',w'\}\in E'} \frac{\xi'(w')}{b_{v',w'}}
	=mr^*\times k_v\xi(v)-\sum_{\{v,w\}\in E} r^*\times\frac{m\times\xi(w)}{b_{v,w}},$$
	or equivalently,
	$$\left(\opEuler_{M'}\xi'\right)(v')=mr^*\times(\opEuler_{M}\xi)(v).$$
	It follows that
	$$\left|\left(\opEuler_{M'}\xi'\right)(v')\right|\leq -\chi_{M'}(v')-\mathrm{valence}_{(V',E')}(v')$$
	holds for all $v'\in V'$.
	Therefore,
	Lemma \ref{volume_existence_special} applies to $M'$ and $\xi'$.
	We obtain some representation $\rho'\colon\pi_1(M')\to\Sft$,
	such that
	$$\mathrm{vol}_{\Sft\times_\Integral\Real}\left(M',\rho'\right)
	=4\pi^2\cdot\left(\xi',\opEuler_{M'}\xi'\right)
	=4\pi^2\cdot\left(\xi,\opEuler_{M}\xi\right)\times[M':M].$$
	In other words, $(M',\rho')$ is as desired.
	
	This completes the proof of Theorem \ref{volume_existence_general_GM}.

\section{Application to constant cyclic graph manifolds}\label{Sec-cyclic}
	We apply Theorems \ref{volume_general_GM} and \ref{volume_existence_general_GM}
	to a family of graph manifolds, 
	and determine their covering Seifert volume (see (\ref{def_CSV})).
	As a quick consequence,
	we also obtain Example \ref{CSV_Mgm} (see Remark \ref{cyclic_k_b_remark}).
			
	\begin{theorem}\label{cyclic_k_b}
		Let $M$ be a formatted graph manifold with a cyclic JSJ graph of $n$ vertices,
		$n\geq3$.	Adopt Notation \ref{notation_xkb}.
		Enumerate the vertices as $j$	and the edges as $\{j,j+1\}$, 
		where $j$ ranges over $\Integral/n\Integral$.
		Suppose there are rational numbers $\chi<0$, $b\neq0$, and $k$, such that
		$$\begin{array}{ccc}
		\chi_{j}=\chi, &b_{j,j+1}=b, &k_{j}=k,
		\end{array}$$
		for all $j\in\Integral/n\Integral$.
		Then
		$$\mathrm{CSV}(M)=\begin{cases}4n\pi^2\chi^2/(|k|-|2/b|)&\mbox{if }|k|>|2/b|\\+\infty&\mbox{if }|k|\leq|2/b|\end{cases}$$
	\end{theorem}
	
	\begin{remark}\label{cyclic_k_b_remark}
		We derive Example \ref{CSV_Mgm} as follows.
		Since $g>0$, any $M(g;m,1,m^2-1,m)$ admits a $2$--fold cyclic cover $M'$
		which restricts to a connected $2$--fold cover to each of the JSJ pieces $J_\pm$.
		Take $M''$ to be the $2$--fold cover of $M'$ induced by the $2$--fold cover of its JSJ graph.
		Then $M''$ will satisfy the hypothesis of Theorem \ref{cyclic_k_b},
		with $\chi=2(1-2g)$, $b=1$, $k=2m$, and $n=4$.
		One may compute $\mathrm{CSV}(M(g;m,1,m^2-1,m))$ as $\mathrm{CSV}(M'')/4$.
	\end{remark}
	
	Below we divide into two cases:
	The finite case is done in Lemma \ref{cyclic_k_b_finite},
	and the infinite case in Lemma \ref{cyclic_k_b_infinite}.
	Together they make a complete proof of Theorem \ref{cyclic_k_b}.

	\begin{lemma}\label{cyclic_k_b_finite}
		Adopt the notations of Theorem \ref{cyclic_k_b}.
		If $|k|>|2/b|$, then $\mathrm{CSV}(M)=4n\pi^2\chi^2/(|k|-|2/b|)$.
	\end{lemma}
	
	\begin{proof}
		Without loss of generality, we assume the number of vertices $n$ to be even.
		In fact, for any odd $n$, 
		we may work with the $2$--fold cyclic cover $M'$ of $M$ induced by 
		that of the JSJ graph, since $\mathrm{CSV}(M')=2\cdot\mathrm{CSV}(M)$ and $n'=2n$. 
		We actually make use of the parity assumption only when $k,b$ have opposite signs.		
	
		Under the assumption $|k|>|2/b|$,
		the Euler operator $\opEuler_M$ is strongly diagonally dominant.
		So we apply Theorem \ref{bound_diagonally_dominant} to obtain
		$\mathrm{CSV}(M)\leq 4n\pi^2\chi^2/(|k|-|2/b|)$.
		On the other hand, suppose first that $k$ and $b$ have the same sign.
		Note that the Euler operator $\opEuler_M\in\mathrm{End}_\Real(\Real^{\Integral/n\Integral})$ 
		can be explicitly written down, as
		$$(\opEuler_M\eta)(j)=k\cdot\eta(j)-\frac{\eta(j-1)+\eta(j+1)}{b},$$
		for all $\eta\in\Real^{\Integral/n\Integral}$ and $j\in\Integral/n\Integral$.
		For any $y\in\Rational$ and $|y|<|\chi|/|k-2/b|$,
		we construct $\eta\in\Rational^{\Integral/n\Integral}$ such that
		$\eta(j)=y$ for all $j\in\Integral/n\Integral$.
		Then Theorem \ref{volume_existence_general_GM} applies to this case,
		so we obtain $\mathrm{CSV}(M)\geq 4n\pi^2y^2\cdot|k-2/b|$.
		The right-hand side converges to to $4n\pi^2\chi^2/(|k|-|2/b|)$
		as $y$ tends to $|\chi|/|k-2/b|$.
		Therefore, we obtain $\mathrm{CSV}(M)=4n\pi^2\chi^2/(|k|-|2/b|)$
		as asserted	when $k$ and $b$ have the same sign.
		
		When $k$ and $b$ have different signs,
		for any $y\in\Rational$ and $|y|<|y|/|k+2/b|$,
		we construct instead $\eta(j)=(-1)^jy$.
		This is well-defined on $\Integral/n\Integral$ since $n$ is even.
		Apply Theorem \ref{volume_existence_general_GM} again,
		and let $y$ tend to $|\chi|/|k+2/b|$,
		then we obtain $\mathrm{CSV}(M)=4n\pi^2\chi^2/(|k|-|2/b|)$
		as asserted	when $k$ and $b$ have different signs.
	\end{proof}

	\begin{lemma}\label{cyclic_k_b_infinite}
		Adopt the notations of Theorem \ref{cyclic_k_b}.
		If $|k|\leq|2/b|$, then $\mathrm{CSV}(M)=+\infty$.
	\end{lemma}
	
	\begin{proof}
		Again we assume without loss of generality that the number of vertices $n$ is even.
		Let $M$ be a formatted graph manifold in Theorem \ref{cyclic_k_b}.
		For any $d\in\Natural$,
		we take the $d$--cyclic cover $M^*_d$ of $M$ 
		as induced by the $d$--cyclic cover of its JSJ graph of degree $d$.
		Since 
		$\mathrm{CSV}(M^*_d)=d\cdot\mathrm{CSV}(M)$, 
		it suffices to show that
		$\mathrm{CSV}(M^*_d)/d$ is unbounded as $d$ increases.
		
		To this end,		
		enumerate the vertices of the JSJ graph of $M^*_d$ as $j\in\Integral/nd\Integral$,
		and the edges as $\{j,j+1\}$.
		So the covering projection between the JSJ graphs of $M^*_d$ and $M$
		takes $j\in\Integral/nd\Integral$
		to $j\bmod n\in\Integral/n\Integral$,
		and takes $\{j,j+1\}$ to $\{j\bmod n,j+1\bmod n\}$.
		Denote by $\opEuler_d\in\mathrm{End}_\Real(\Real^{\Integral/nd\Integral})$ 
		the Euler operator of $M^*_d$.
		For any $\eta\in\Real^{\Integral/nd\Integral}$ and $j\in\Integral/nd\Integral$,
		we obtain
		$$(\opEuler_d\eta)(j)=k\cdot\eta(j)-\frac{\eta(j-1)+\eta(j+1)}{b}.$$

		The eigenvalues and eigenvectors of $\opEuler_d$ can be explicitly determined.
		The eigenvalues of $\opEuler_d$ are precisely
		$$\lambda_m
		=k-\frac{2}{b}\cdot\cos\left(\frac{2m\pi}{nd}\right)
		=\left(k-\frac{2}{b}\right)+\frac{4}{b}\cdot\sin^2\left(\frac{m\pi}{nd}\right),$$
		where $m$ ranges over $\{0,1,2,\cdots,nd/2\}$.
		The bottom eigenvalue $\lambda_0$ and the top eigenvalue $\lambda_{nd/2}$ both have multiplicity $1$,
		and there are corresponding normalized eigenvectors
		$$\begin{array}{cc}\Phi_0(j)=\sqrt{1/nd},&\Phi_{nd/2}(j)=(-1)^{j}\sqrt{1/nd}.\end{array}$$
		Any other eigenvalue $\lambda_m$ has multiplicity $2$,
		with an orthonormal pair of eigenvectors
		$$\begin{array}{cc}
		\Phi_m(j)=\sqrt{2/nd}\cdot\cos(2mj\pi/nd),&
		\Psi_m(j)=\sqrt{2/nd}\cdot\sin(2mj\pi/nd).
		\end{array}$$
		These eigenvalues and eigenvectors can be checked by direct computation.
		We also remark that when $nd$ is odd,
		the spectrum of $\opEuler_d$ will be the same as above except dropping the top eigenvalue,
		and the corresponding eigenvector will be gone.
		In fact, when $k=2$ and $b=-1$,
		one may identify $\opEuler_d$ with the graph-theoretic Laplacian $\mathcal{L}_{nd}$ 
		of the cyclic graph of order $nd$, (see \cite[Chapter 7]{Nica-book}).
		In general,
		$\opEuler_d$ is a scalar multiple of $\mathcal{L}_{nd}$ plus a constant.
		
		Under the assumption $|k|\leq|2/b|$, there exists some $\lambda_m$
		(other than the top or the bottom eigenvalues) such that
		$$0<|\lambda_m|<C_1\cdot d^{-1},$$
		where $C_1=|8\pi/bn|$ is independent of $d$.
		This is because the eigenvalues are monotonically ordered
		between $\lambda_0=k-2/b$ and $\lambda_{nd/2}=k+2/b$, 
		and the difference between any consecutive pair of eigenvalues
		is at most $|8\pi/bnd|$, by elementary estimation.
		Take any such $\lambda_m$, 
		and take $\eta_d=\sqrt{nd/2}\cdot\chi\cdot\lambda_m^{-1}\cdot\Phi_m$ in $\Real^{\Integral/nd\Integral}$.
		We estimate
		$$\left|\left(\opEuler_d\eta_d\right)(j)\right|=|\lambda_m\cdot\eta_d(j)|\leq |\chi|,$$
		for all $j\in\Integral/nd\Integral$.
		
		For any $0<\epsilon<1$, we can approximate $(1-\epsilon)\cdot\eta_d\in\Real^{\Integral/nd\Integral}$
		by elements in $\Rational^{\Integral/nd\Integral}$.
		Therefore,
		Theorem \ref{volume_existence_general_GM} applies to the approximating elements,
		and yields the estimation
		$$\mathrm{CSV}(M^*_d)\geq |4\pi^2\cdot\left(\eta_d,\opEuler_d\eta_d\right)|=2nd\pi^2\chi^2\cdot|\lambda_m|^{-1}>C_2\cdot d^2,$$
		where $C_2=2n\pi^2\chi^2\cdot C_1^{-1}$ is independent of $d$.
		This shows that $\mathrm{CSV}(M^*_d)/d$ is unbounded as $d$ increases,
		as desired, so the proof is complete.
	\end{proof}

\appendix
\section{Normalization of the volume in Seifert geometry}\label{Sec-normalization_SV}
	Brooks and Goldman defined the Seifert volume in terms of the Godbillon--Vey invariant \cite{GV}
	of certain transversly projective codimension--$1$ foliation \cite{BG2}.
	On the other hand, the representation volume associated to $\Sft\times_\Integral\Real$
	is determined only after fixing a left-invariant volume form on the homogeneous space $\Sft$,
	(see \cite[Section 2]{DLSW-rep_vol}).
	In this section of appendix, we figure out the normalization explicitly.
	
	Recall the definition of Brooks and Goldman as follows.
	Suppose that $M$ is a closed oriented smooth $3$--manifold
	and $\bar{\rho}\colon \pi_1(M)\to \mathrm{PSL}(2,\Real)$ is a representation.
	Then we obtain a fiber bundle $M\times_{\bar\rho}\Real P^1$ over $M$,
	whose fiber is modeled on $\Real P^1$ with structure group $\mathrm{PSL}(2,\Real)$,
	so the bundle space $M\times_{\bar\rho}\Real P^1$ is equipped with a codimension--$1$
	foliation $\mathfrak{F}_{\bar\rho}$ transverse to the fibers. 
	The Godbillon--Vey class $\mathrm{gv}(\mathfrak{F}_{\bar\rho})$ lives in 
	$H^3(M\times_{\bar\rho}\Real P^1;\Real)$.
	When the Euler class of $\bar{\rho}$ is torsion (see Theorem \ref{lifting_representations}),
	one may naturally identify $H_3(M;\Real)$ with a direct summand of 
	$H_3(M\times_{\bar\rho}\Real P^1;\Real)\cong 
	H_3(M;\Real)\oplus H_2(M;\Real)$,
	which corresponds to the $(0,3)$--summand of the Serre spectral decomposition.
	Denote by $[M]'\in H_3(M\times_{\bar\rho}\Real P^1;\Real)$ the image of the fundamental class of $M$.
	Then the Seifert volume $\mathrm{SV}(M)$ of $M$ in the sense of Brooks and Goldman is defined as
	the maximum of $|\langle\mathrm{gv}(\mathfrak{F}_{\bar\rho}),[M]'\rangle|$
	where $\bar\rho$ ranges over all the $\mathrm{PSL}(2,\Real)$--representations
	of $\pi_1(M)$ whose Euler class is torsion,
	(see \cite[\S 3]{BG2}).
	
	We identify the Lie algebra of $\Sft$ as 
	$\mathfrak{sl}(2,\Real)$,
	consisting of all traceless real $2\times2$--matrices.
	Denote by 
	$$\begin{array}{ccc}
	{H=\left[\begin{array}{cc}1&0\\0&-1\end{array}\right],} &
	{E=\left[\begin{array}{cc}0&1\\0&0\end{array}\right],} &
	{F=\left[\begin{array}{cc}0&0\\1&0\end{array}\right]} 
	\end{array}$$
	the usual basis of $\mathfrak{sl}(2,\Real)$.
	They satisfy the relations
	$[H,E]=2E$, $[H,F]=-2F$, and $[E,F]=H$.
	Denote by $\check{E},\check{F},\check{H}$ the dual basis 
	of $\mathfrak{sl}(2,\Real)^{\vee}=\mathrm{Hom}_{\Real}(\mathfrak{sl}(2,\Real),\Real)$.
	The exterior form $\check{H}\wedge\check{E}\wedge\check{F}$ 
	can be regarded as an alternating trilinear function on $\mathfrak{sl}(2,\Real)$
	such that $\check{H}\wedge\check{E}\wedge\check{F}(H,E,F)=1$,
	and it can also be naturally regarded 
	as a left-invariant differential $3$--form on $\Sft$.
	
	\begin{proposition}\label{SV_normalization}
	The Seifert volume in the sense of Brooks and Goldman coincides
	with the representation volume with respect to the triple
	$(G,X,\omega_X)$ as defined in \cite{DLSW-rep_vol},
	where
	$$\begin{array}{ccc}
	G=\Sft\times_\Integral\Real,&
	X=\Sft,&
	\omega_{X}=\pm4\check{H}\wedge\check{E}\wedge\check{F}.
	\end{array}$$
	\end{proposition}
	
	\begin{proof}
		For any oriented closed hyperbolic surface $\Sigma$,
		the unit vector bundle $\mathrm{UT}(\Sigma)$ is an oriented circle bundle
		over $\Sigma$ of Euler number $\chi(\Sigma)<0$.
		The Seifert volume of $\mathrm{UT}(\Sigma)$ equals $-4\pi^2\chi(\Sigma)$,
		according to the definition of Brook and Goldman;
		moreover, it is realized by any discrete faithful representation 
		$\bar\rho\colon \pi_1(\mathrm{UT}(\Sigma))\to\Sft$,
		which must be central and induces a holonomy representation $\pi_1\Sigma\to\mathrm{PSL}(2,\Real)$.
		(See \cite[Proposition 2]{BG1} and \cite[Theorem 3]{BG2};
		or Theorem \ref{volume_central_bundle}).
		As the discrete faithful representation realizes $\mathrm{UT}(\Sigma)$
		as a $\Sft$--geometric manifold, 
		it suffices to check that the asserted invariant volume form
		$4\check{H}\wedge\check{E}\wedge\check{F}$ on $\Sft$
		results in the same volume value $-4\pi^2\chi(\Sigma)$ for $\mathrm{UT}(\Sigma)$.
		
		To this end, we observe
		$$\begin{array}{ccc}
		e^{tH}=\left[\begin{array}{cc} e^t& 0\\0 &e^{-t}\end{array}\right],&
		e^{t(E+F)}=\left[\begin{array}{cc} \cosh(t)& \sinh(t)\\ \sinh(t) &\cosh(t)\end{array}\right],&
		e^{t(E-F)}=\left[\begin{array}{cc} \cos(t)& \sin(t)\\ -\sin(t) &\cos(t)\end{array}\right].
		\end{array}$$
		Using the upper-half complex plane model of the hyperbolic plane $\Hyp^2$,
		it follows that 
		the tangent map of 
		$\mathrm{PSL}(2,\Real)\to \Hyp^2\colon g\mapsto g.\mathbf{i}$
		takes $H$ and $E+F$ to the tangent vectors $2\mathbf{i}$ and $2$ at $\mathbf{i}$, respectively.
		(Here we canonically identify the tangent space of $\Complex$ at $\mathbf{i}$ with $\Complex$.)
		The parametrized subgroup $e^{t(E-F)}$ rotates any unit vector at $\mathbf{i}$
		counterclockwise at constant angular speed $2$.
		Therefore, if $\mathfrak{sl}(2,\Real)$ is endowed with an inner product
		such that $H/2,(E+F)/2,(E-F)/2$ form an orthonormal basis,
		then the induced left-invariant Riemannian metric on $\mathrm{PSL}(2,\Real)$ 
		will have constant total length $2\pi$ for all left-cosets of compact subgroup $K=\mathrm{SO}(2)/\{\pm1\}$, 
		and the homogeneous space $\mathrm{PSL}(2,\Real)/K$ with the induced left-invariant metric
		will be isometric to $\Hyp^2$.
		
		This means that the desired normalized volume form
		$\omega_{\Sft}$ must make
		$|\omega_{\Sft}(H/2,(E+F)/2,(E-F)/2)|=1$.
		Since 
		$\check{H}\wedge\check{E}\wedge\check{F}(H/2,(E+F)/2,(E-F)/2)=-1/4$,
		we obtain	$\omega_{\Sft}=\pm4\check{H}\wedge\check{E}\wedge\check{F}$,
		as asserted.
	\end{proof}
	
	The sign ambiguity is inessential.
		It arises only because the Seifert volume is insensitive of orientation.
		However, it can be resolved if we orient the unit vector bundle $M=\mathrm{UT}(\Sigma)$
		of any closed oriented hyperbolic surface $\Sigma$
		by a fixed orientation of $\mathrm{PSL}(2,\Real)\cong\mathrm{UT}(\Hyp^2)$,
		and if we require
		$\mathrm{vol}_{G,X,\omega_X}(M)=
		\langle\mathrm{gv}(\mathfrak{F}_{\bar\rho}),[M]'\rangle$,
		where $\bar\rho\colon \pi_1(M)\to\mathrm{PSL}(2,\Real)$
		is induced by any holonomy representation 
		$\rho\colon \pi_1(M)\to\Sft\times_\Integral\Real$.
		This way $\omega_X$ must agree with the fixed orientation of $\mathrm{PSL}(2,\Real)$,
		so the positive sign must apply if 
		we orient $\mathfrak{sl}(2,\Real)$ with the ordered basis $H,E,F$.
		
	In conclusion, we may (and do) fix the invariant volume form $\omega_{X}$ 
	to be $4\check{H}\wedge\check{E}\wedge\check{F}$.

\bibliographystyle{amsalpha}

\begin{thebibliography}{}
\bibitem[AscFW15]{AFW-3mg}
M.~Aschenbrenner, S.~Friedl, H.~Wilton: \textit{3-Manifold Groups}. 
EMS Series of Lectures in Mathematics. 
European Mathematical Society (EMS), Z\"urich, 2015.

\bibitem[BesCG07]{BCG} 
{G.~Besson, G.~Courtois, S.~Gallot}, {\it In\'egalit\'es de Milnor Wood g\'eom\'etriques}, Comment.~Math.~Helv.~\textbf{82} (2007),  753--803.


\bibitem[BroG84a]{BG1} {R.~Brooks, W.~Goldman}, {\it The Godbillon--Vey invariant of a transversely homogeneous foliation}, Trans.~Amer.~Math.~Soc.~\textbf{286} (1984), no.~2, 651--664.

\bibitem[BroG84b]{BG2} \bysame, {\it Volumes in Seifert space},  Duke Math.~J.~\textbf{51}  (1984),  no.~3, 529--545.

\bibitem[BuyS04]{BS-graph_manifold}
S.~V.~Buyalo and P.~V.~Svetlov,
\textit{Topological and geometric properties of graph manifolds}, 
Algebra i Analiz \textbf{16} (2004), 3--68.




\bibitem[DerLW15]{DLW-cs_sep_vol}
P.~Derbez, Y.~Liu, and S.~Wang, 
\textit{Chern--Simons theory, surface separability, and volumes of 3–manifolds},
J.~Topol.~\textbf{8} (2015), 933--974.

\bibitem[DerLSW17]{DLSW-virtual_SV}
P.~Derbez, Y.~Liu, H.~Sun, S.~Wang, 
\textit{Positive simplicial volume implies virtually positive Seifert volume for 3-manifolds}, 
Geom.~Topol.~\textbf{21} (2017), 3159--3190.

\bibitem[DerLSW19]{DLSW-rep_vol}
\bysame,
\textit{Volume of representations and mapping degree}, 
Adv.~Math.~\textbf{351} (2019), 570--613. 

%




\bibitem[EisHN81]{EHN} 
{D.~Eisenbud, U.~Hirsch, W.~Neumann}, {\it Transverse foliations of Seifert bundles and self homeomorphism of the circle},  Comment.~Math.~Helv.~\textbf{56} (1981), no.~4, 638--660.


\bibitem[GodV71]{GV} 
{C.~Godbillon and J.~Vey},  {\it Un invariant des feuilletages de
codimension 1}, (French) C.~R.~Acad.~Sci.~Paris S\'er.~A-B \textbf{273}
(1971), A92--A95.

\bibitem[Gol88]{Goldman-components}
W.~Goldman,
\textit{Topological components of spaces of representations},
Invent.~Math.~\textbf{93} (1988), 577--607.

\bibitem[Gro82]{Gr} {M.~Gromov}, 
{\it Volume and bounded cohomology},  Inst.~Hautes \'{E}tudes Sci.~Publ.~Math. \textbf{56}, (1982), 5--99.

%
%

\bibitem[Hem76]{Hempel-book}
J.~Hempel, \textit{3-Manifolds}, Ann.~of Math.~Studies, vol.~86, Princeton University
Press, Princeton, NJ, 1976.

%

\bibitem[HorJ13]{HJ-book}
R.~Horn and C.~Johnson, \textit{Matrix Analysis}, 2nd edition, Cambridge University Press, New York, NY, 2013.

\bibitem[JacS79]{JS} 
{W.~Jaco and P.~B.~Shalen},  
{\it Seifert fibered space in 3-manifolds},  Mem.~Amer.~Math.~Soc.~\textbf{21}, 1979.

\bibitem[JanN85]{JN}
M.~Jankins and W.~Neumann,
\textit{Homomorphisms of Fuchsian groups to $\mathrm{PSL}(2,\mathbb{R})$},
Comment.~Math.~Helvitici \textbf{60} (1985), 480--495.



\bibitem[Kho03]{Kh} 
{V.-T.~Khoi}, {\it A cut-and-paste method for computing the Seifert volumes},  
Math.~Ann.~\textbf{326} (2003), no.~4, 759--801.

%
%
\bibitem[LueW97]{LW} {J.~Luecke, Y.-Q.~Wu}, {\it Relative Euler number and finite covers of graph manifolds}  Geometric Topology (Athens, GA, 1993),  pp.~80--103, Amer.~Math.~Soc., Providence, RI, 1997.
%
\bibitem[Mil58]{Milnor} 
{J.~Milnor} {\it On the existence of a connection with curvature zero}, 
Comment.~Math.~Helv.~\textbf{32} 1958, 215--223.



\bibitem[NeuY95]{NY-rationality}
 W.~D.~Neumann and J.~Yang, 
\textit{Problems for K-theory and Chern--Simons invariants of hyperbolic 3-manifolds}, 
L'Enseignement Math\'ematique \textbf{41} (1995), 281--296. 

\bibitem[Nic18]{Nica-book} 
B.~Nica, \textit{A Brief Introduction to Spectral Graph Theory}, 
European Mathematical Society, Z\"urich, 2018.

%
%

\bibitem[Rat06]{Ratcliffe} 
J.~Ratcliffe, \textit{Foundations of Hyperbolic Manifolds}, 2nd edition. 
Graduate Texts in Mathematics, 149. Springer, New York, 2006.

\bibitem[Rez96]{Re-rationality} {A.~Reznikov},
{\it Rationality of secondary classes},
J.~Diff.~Geom.~\textbf{43} (1996), no.~3, 674--692.


\bibitem[Sco83]{Scott-geometries}
P.~Scott, \textit{The geometries of 3-manifolds}, 
Bull.~London Math.~Soc.~\textbf{15} (1983), 401--487.

\bibitem[Ser80]{Serre} 
J.-P.~Serre, \textit{Trees}, translated from French by
J.~Stillwell, Springer--Verlag, 1980.


%

\bibitem[Thu97]{Thurston-book} 
W.~P.~Thurston, \textit{Three-Dimensional Geometry and Gopology}, Volume 1, 
Princeton University Press, Princeton, 1997.


\bibitem[Woo73]{Wood}
J.~Wood, 
\textit{Foliated $S^1$-bundles and diffeomorphisms of $S^1$}, 
Dynamical Systems (M. Peixoto, ed.), Academic Press, New York, 1973.


\end{thebibliography}

\end{document}